\documentclass[a4paper,12pt]{article}
\usepackage[top=2.5cm,bottom=2.5cm,left=2.5cm,right=2.5cm]{geometry}
\usepackage{cite, amsmath, amssymb}
\usepackage{amssymb}
\usepackage{graphicx}
\newenvironment{proof}{{\noindent \it Proof.}}{\hfill $\blacksquare$\par}
\usepackage{latexsym}
\usepackage{graphicx,booktabs,multirow}
\usepackage{tikz}
\usetikzlibrary{decorations.pathreplacing}
\usetikzlibrary{intersections}
\usepackage{enumerate}
\usepackage{graphicx,booktabs,multirow}
\usepackage{appendix}
\newtheorem{theorem}{Theorem}[section]

\newtheorem{lemma}[theorem]{Lemma}
\newtheorem{corollary}[theorem]{\rm\bfseries Corollary}

\newtheorem{conjecture}{Conjecture}[section]

\usepackage[section]{placeins}
\allowdisplaybreaks

\usepackage{rotating}

\begin{document}

\vspace*{10mm}

\noindent
{\Large \bf Mathematical and chemistry properties of geometry-based invariants}

\vspace{7mm}

\noindent
{\large \bf Hechao Liu}

\vspace{7mm}

\noindent
School of Mathematical Sciences, South China Normal University, Guangzhou, 510631, P. R. China, e-mail: {\tt hechaoliu@m.scnu.edu.cn}

\vspace{7mm}

\noindent
Received 16 August 2022

\vspace{10mm}

\noindent
{\bf Abstract} \
Recently, based on elementary geometry, Gutman proposed several geometry-based invariants (i.e., $SO$, $SO_{1}$, $SO_{2}$, $SO_{3}$, $SO_{4}$, $SO_{5}$, $SO_{6}$).
The Sombor index was defined as $SO(G)=\sum\limits_{uv\in E(G)}\sqrt{d_{u}^{2}+d_{v}^{2}}$,
the first Sombor index was defined as $SO_{1}(G)= \frac{1}{2}\sum\limits_{uv\in E(G)}|d_{u}^{2}-d_{v}^{2}|$, where $d_{u}$ denotes the degree of vertex $u$.

In this paper, we consider the mathematical and chemistry properties of these geometry-based invariants. We determine the maximum trees (resp. unicyclic graphs) with given diameter, the maximum trees with given matching number, the maximum trees with given pendent vertices, the maximum trees (resp. minimum trees) with given branching number, the minimum trees with given maximum degree and second maximum degree, the minimum unicyclic graphs with given maximum degree and girth, the minimum connected graphs with given maximum degree and pendent vertices, and some properties of maximum connected graphs with given pendent vertices with respect to the first Sombor index $SO_{1}$.

As an application, we inaugurate these geometry-based invariants and verify their chemical applicability. We used these geometry-based invariants to model the acentric factor (resp. entropy, enthalpy of vaporization, etc.) of alkanes, and obtained satisfactory predictive potential, which indicates that these geometry-based invariants can be successfully used to model the thermodynamic properties of compounds.
\vspace{5mm}



\baselineskip=0.30in

\section{Introduction}
\hskip 0.6cm
Let $G$ be a graph with vertex set $V(G)$ and edge set $E(G)$.
Let $d_{G}(u)$ (or $d_{u}$) be the degree of vertex $u$ in $G$, $N_{G}(u)=\{v|uv\in E(G)\}$ be the neighbor of vertex $u$ in $G$.
Let $\Delta(G)$ (resp. $\Delta_{2}(G)$) the maximum (resp. second maximum degree) in graph $G$.
Let $k$-vertex be the vertex with degree $k$.
Denote by $d_{G}(u,v)$ the length of a shortest path
between $u$ and $v$, then $diam(G)=\max\limits_{\{u,v\}\subseteq V(G)}\{d_{G}(u,v)\}$ is the diameter of $G$.
Denote by $n_{i}$, $m_{i,j}$ the number of vertices with degree $i$, the number of edges with degree $i$ and $j$ in $G$.
Let $PV(G)$ be the set of pendent vertices of $G$.
Denote by $C_{n}$, $S_{n}$, $P_{n}$, the cycle, star graph, path with order $n$, respectively.
The girth is the length of shortest cycle of graphs.
Let $uv\in E(G)$, then $G-uv$ denotes the graphs obtained from $G$ by deleting the edge $uv$, respectively.
In this paper, all notations and terminologies used but not defined can refer to Bondy and Murty \cite{bond2008}.

The geometry-based invariants were proposed by Gutman \cite{gumn2021,guin2022}, which were defined as follow:
\begin{eqnarray*}
SO(G) & = & \sum\limits_{uv\in E(G)}\sqrt{d_{u}^{2}+d_{v}^{2}}.\\[3mm]
SO_{1}(G) & = & \frac{1}{2}\sum\limits_{uv\in E(G)}|d_{u}^{2}-d_{v}^{2}|.\\[3mm]
SO_{2}(G) & = & \sum\limits_{uv\in E(G)}|\frac{d_{u}^{2}-d_{v}^{2}}{d_{u}^{2}+d_{v}^{2}}|.\\[3mm]
SO_{3}(G) & = & \sum\limits_{uv\in E(G)}\sqrt{2}\frac{d_{u}^{2}+d_{v}^{2}}{d_{u}+d_{v}}\pi.\\[3mm]
SO_{4}(G) & = & \frac{1}{2}\sum\limits_{uv\in E(G)}(\frac{d_{u}^{2}+d_{v}^{2}}{d_{u}+d_{v}})^{2}\pi.\\[3mm]
SO_{5}(G) & = & \sum\limits_{uv\in E(G)}\frac{2|d_{u}^{2}-d_{v}^{2}|}{\sqrt{2}+2\sqrt{d_{u}^{2}+d_{v}^{2}}}\pi.\\[3mm]
SO_{6}(G) & = & \sum\limits_{uv\in E(G)}[\frac{|d_{u}^{2}-d_{v}^{2}|}{\sqrt{2}+2\sqrt{d_{u}^{2}+d_{v}^{2}}}]^{2}\pi.
\end{eqnarray*}

The so-called ``Sombor index'' was introduced in a very recent paper, to be published in 2021. Since then, a flood of papers were created, reporting properties of the mentioned new topological indices. This is reasonable to expect, since there is a real need to determine their main properties, same as what earlier has been done with other distance-based topological indices.
One can refer to \cite{rirm2021,chli2021,lizh2022,lhua2022,lyhu2021} for more details about Sombor index.

For convenience, we call $SO_{1}$ the first Sombor index, which also called the modified Albertson index by dividing by 2. Let $2SO_{1}(G)=\sum\limits_{xy\in E(G)}\phi(x,y)$, where $\phi(x,y)=|d_{G}^{2}(x)-d_{G}^{2}(y)|$.
In \cite{ghso2020}, Ghalavand et al. determined the first eight smallest $2SO_{1}$ in trees with $n$ vertices. In \cite{ybai2020}, Yousaf et al. determined the maximum and minimum $2SO_{1}$ in trees with $n$ vertices.

The purpose of the paper \cite{gumn2021}``Geometric approach to degree-based topological indices: Sombor indices'' and \cite{guin2022}``Sombor indices-back to geometry'' is to invite other colleagues not only to do the combinatorial and algebraic studies of these new Sombor-like indices, but to think their geometric origin.
Now, as expected, we did the combinatorial and algebraic studies of these new Sombor-like indices.

The remainder of this paper is organized as follow. In Section 2, we introduce some lemmas.
In Section 3, we determine the maximum trees (resp. unicyclic graphs) with given diameter with respect to $SO_{1}$.
In Section 4, we determine the maximum trees with given matching number with respect to $SO_{1}$.
In Section 5, we determine the maximum trees with given pendent vertices with respect to $SO_{1}$.
In Section 6, we determine the maximum trees (resp. minimum trees) with given branching number with respect to $SO_{1}$.
In Section 7, we determine the minimum trees with given maximum degree and second maximum degree, the minimum unicyclic graphs with given maximum degree and girth, the minimum connected graphs with given maximum degree and pendent vertices, and some properties of maximum connected graphs with given pendent vertices with respect to $SO_{1}$.
In Section 8, we inaugurate these geometry-based invariants and verify their chemical applicability.
In Section 9, we conclude this paper.

\section{Preliminaries}
\hskip 0.6cm
By the definition of $SO_{1}$, we have
\begin{lemma}\label{l1-1}{\rm \textbf{Edge-lifting transformation}}
Let $G$ be a connected graph and $uv\in E(G)$ is not a pendent edge. Let $G^{*}=G-\{uw|w\in N_{G}(u)\setminus \{v\}\}+\{vw|w\in N_{G}(u)\setminus \{v\}\}$. Then $SO_{1}(G)<SO_{1}(G^{*})$.
\end{lemma}

\begin{lemma}\label{l1-2}{\rm \cite{ybai2020}}
Let $n\geq 5$ and $T$ be a tree with $n$ vertices and $T\ncong P_{n}$, $S_{n}$. Then $SO_{1}(P_{n})<SO_{1}(T)<SO_{1}(S_{n})$.
\end{lemma}

By Lemmas \ref{l1-1} and \ref{l1-2}, we have
\begin{theorem}\label{t1-3}
Let $n\geq 5$ and $T$ be a tree with $n$ vertices and $T\ncong S_{n}$. Then $SO_{1}(T)\leq SO_{1}(S_{n}^{*})$, with equality if and only if $T\cong S_{n}^{*}$ where $S_{n}^{*}$ is the graph obtained from $S_{n}$ by subdividing one edge.
\end{theorem}

\section{Trees and unicyclic graphs with given diameter}
\hskip 0.6cm
Let $\mathcal{T}(n,d)$ be the set of trees with order $n$ and diameter $d$.
Let $T_{n,d}^{i}$ be the tree obtained from the path $P=v_{0}v_{1}\cdots v_{d}$ by attaching $n-d-1$ pendent edges (i.e., $v_{i}u_{1}$, $v_{i}u_{2}$, $\cdots$ $v_{i}u_{n-d-1}$) to the vertex $v_{i}$, see Figure \ref{fig-21}.

Let $2\leq d\leq n-2$ and $\mathcal{U}(n,d)$ be the set of unicyclic graphs with order $n$ and diameter $d$. Let $U_{n,d}^{i}$ (resp. $R_{n,d}^{i}$, $W_{n,d}^{i}$ ) be the graphs obtained from tree $T_{n,d}^{i}$ by adding edge $u_{n-d+1}v_{i+2}$ (resp. $u_{n-d-1}v_{i+1}$, $u_{n-d-2}u_{n-d-1}$), see Figure \ref{fig-21}.

\begin{figure}[ht!]
  \centering
  \scalebox{.16}[.16]{\includegraphics{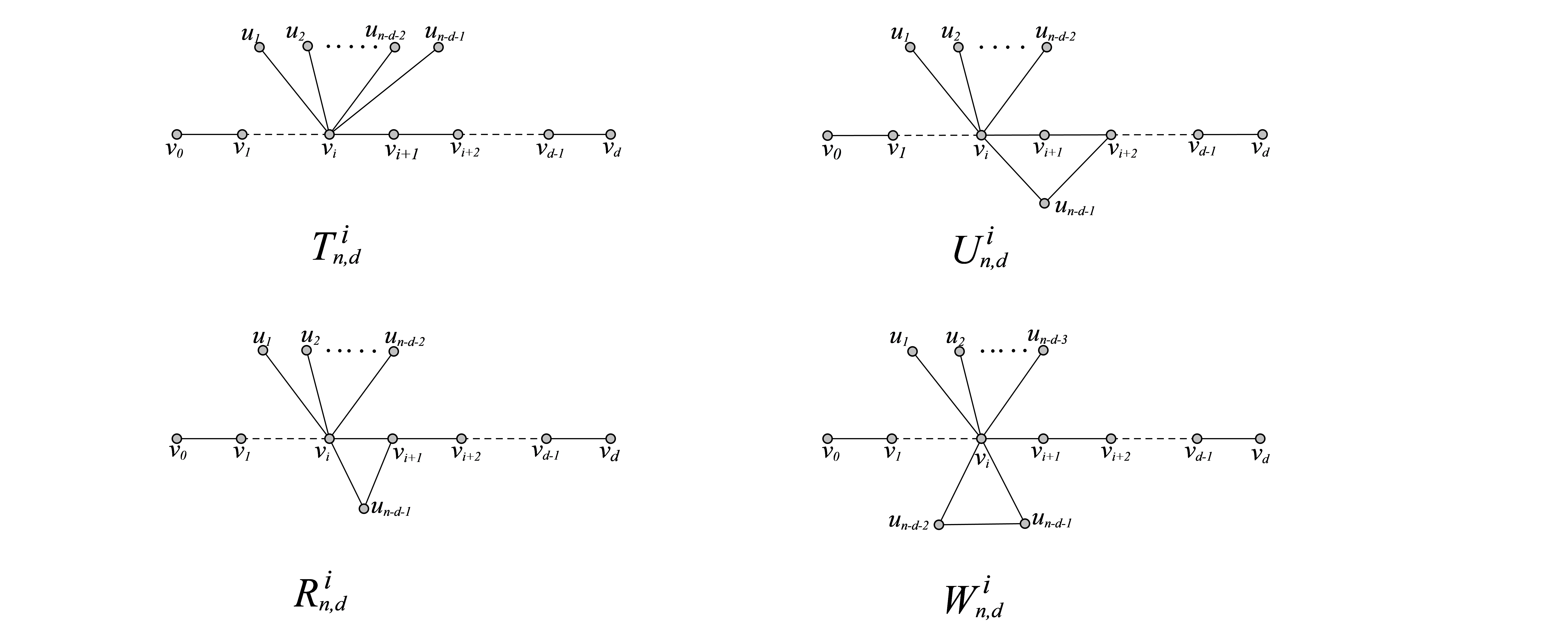}}
  \caption{The graphs $T_{n,d}^{i}$, $U_{n,d}^{i}$, $R_{n,d}^{i}$ and $W_{n,d}^{i}$.}
 \label{fig-21}
\end{figure}

\begin{theorem}\label{t2-1}
Let $2\leq d\leq n-1$ and $T\in \mathcal{T}(n,d)$. Then
$$ SO_{1}(T)\leq \frac{1}{2}(n-d+1)((n-d+1)^{2}-1),$$
with equality if and only if $T\cong T_{n,d}^{i}$ for $i=1,2,\cdots, \lfloor \frac{d}{2} \rfloor$.
\end{theorem}
\begin{proof}
We prove the result by using mathematical induction on $n-d$. If $n-d=1$, the $T\cong P_{n}$ and $SO_{1}(P_{n})=3=\frac{1}{2}(n-d+1)((n-d+1)^{2}-1)$.

Suppose that the conclusion holds for $n-d\leq k-1$ $(k\geq 2)$.
Let $n-d=k$ and $T\in \mathcal{T}(n,d)$.
Let $PV(T)$ be the set of pendent vertices of $T$. Suppose that $P$ is the diametral path of $T$,
$u\in PV(T)$ and $u\notin V(P)$.
Let $T^{*}=T-u$. Then $T^{*}\in \mathcal{T}(n-1,d)$. Thus $SO_{1}(T^{*})\leq \frac{1}{2}(n-d)((n-d)^{2}-1)$.

Let $N_{T}(u)=v$ and $N_{T}(v)\setminus \{u\}=\{v_{1},v_{2},\cdots,v_{t}\}$, where $t\geq 1$.
Note that $\Delta \leq n-d+1$ since $T\in \mathcal{T}(n,d)$.
Thus $|d^{2}_{T}(v)-(d_{T}(v)-1)^{2}|=2d_{T}(v)-1=2t+1\leq 2(n-d)+1$.
\begin{align*}
2SO_{1}(T)=&  2SO_{1}(T^{*})+\sum_{i=1}^{t}\{|d^{2}_{T}(v)-d^{2}_{T}(v_{i})|-|(d_{T}(v)-1)^{2}-d^{2}_{T}(v_{i})|\}
    +|d^{2}_{T}(v)-d^{2}_{T}(u)| \\
 \leq &  2SO_{1}(T^{*})+t\{|d^{2}_{T}(v)-(d_{T}(v)-1)^{2}|\}+(t+1)^{2}-1 \\
 \leq &  (n-d)((n-d)^{2}-1)+(n-d)(2(n-d)+1)+(n-d+1)^{2}-1 \\
  =   &  (n-d+1)((n-d+1)^{2}-1),&
\end{align*}
with equality if and only if $T^{*}\cong T_{n-1,d}^{i}$ for $i=1,2,\cdots, \lfloor \frac{d}{2} \rfloor$ and $d_{T}(v)=n-d+1$, i.e., $T\cong T_{n,d}^{i}$ for $i=1,2,\cdots, \lfloor \frac{d}{2} \rfloor$.
\end{proof}

\begin{theorem}\label{t2-2}
Let $4\leq d\leq n-2$ and $G\in \mathcal{U}(n,d)$. Then
$$ SO_{1}(G)\leq \frac{1}{2}[(n-d)(n-d+1)(n-d+2)+12],$$
with equality if and only if $G\cong U_{n,d}^{i}$ for $i=1,2,\cdots, d-3$.
\end{theorem}
\begin{proof}
We prove the result by using mathematical induction on $n-d$. If $n-d=2$, the $G\cong R^{i}_{n,n-2}$ for $1\leq i\leq \lfloor \frac{n-3}{2} \rfloor$, $R_{n,n-2}^{3}$, $C_{4}$, $U^{i}_{n,n-2}$ for $1\leq i\leq \lfloor \frac{n-4}{2} \rfloor$, $U_{n,n-2}^{n-4}$ and $SO_{1}(R^{i}_{n,n-2})=13$ for $1\leq i\leq \lfloor \frac{n-3}{2} \rfloor$, $SO_{1}(R_{n,n-2}^{3})=9$, $SO_{1}(C_{4})=0$, $SO_{1}(U^{i}_{n,n-2})=18$ for $1\leq i\leq \lfloor \frac{n-4}{2} \rfloor$, $SO_{1}(U_{n,n-2}^{4})=9$. The conclusion holds.

Suppose that the conclusion holds for $n-d\leq k-1$ $(k\geq 3)$.
Let $n-d=k$ and $G\in \mathcal{U}(n,d)$.
If $G\cong C_{n}$, then it is obvious that the conclusion holds. Thus we suppose $G\ncong C_{n}$, then
$|PV(G)|\geq 1$.

If $|PV(G)|=1$, then $SO_{1}(G)=9<36=\frac{1}{2}[(n-d)(n-d+1)(n-d+2)+12]$.

If $|PV(G)|=2$, then without loss of generality, we let $PV(G)=\{u,v\}$, $C$ be the unique cycle of $G$, $P$ be the shortest path between $u$ and $v$.

If $|V(P\cap C)|\geq 3$, then $SO_{1}(G)=18$;
If $|V(P\cap C)|=2$, then $SO_{1}(G)=13$;
If $|V(P\cap C)|=1$, then $SO_{1}(G)=27$;
If $|V(P\cap C)|=0$, then $SO_{1}(G)=13$ or $18$.
Thus we have $SO_{1}(G)<36$.

If $|PV(G)|\geq 3$, then without loss of generality, we let $P$ be the diameter path of $G$.
Let $u\notin V(P)$ be one pendent vertex of $G$.
Let $G^{*}=G-u$, then $G^{*}\in \mathcal{U}(n-1,d)$.
Thus $SO_{1}(G^{*})\leq \frac{1}{2}[(n-d-1)(n-d)(n-d+1)+12]$.

Let $N_{G}(u)=v$, $N_{G}(v)\setminus \{u\}=\{v_{1},v_{2},\cdots,v_{t}\}$ $(t\geq 1)$.
Note that $t\leq \Delta-1 \leq n-d$ since $G\in \mathcal{U}(n,d)$.
Thus $|d^{2}_{G}(v)-(d_{G}(v)-1)^{2}|=2d_{G}(v)-1=2(t+1)-1\leq 2(n-d)-1$.
\begin{align*}
2SO_{1}(G)=&  2SO_{1}(G^{*})+\sum_{i=1}^{t}\{|d^{2}_{G}(v)-d^{2}_{G}(v_{i})|-|(d_{G}(v)-1)^{2}-d^{2}_{G}(v_{i})|\}
    +|d^{2}_{G}(v)-d^{2}_{G}(u)| \\
 \leq &  2SO_{1}(G^{*})+t\{|d^{2}_{G}(v)-(d_{G}(v)-1)^{2}|\}+(t+1)^{2}-1 \\
 \leq &  (n-d-1)(n-d)(n-d+1)+12+(n-d)(2(n-d)+1)+(n-d+1)^{2}-1 \\
  =   &  (n-d)(n-d+1)(n-d+2)+12,&
\end{align*}
with equality if and only if $G^{*}\cong U_{n-1,d}^{i}$ for $i=1,2,\cdots, d-3$ and $d_{G}(v)=n-d+1$, i.e., $G\cong U_{n,d}^{i}$ for $i=1,2,\cdots, d-3$.
\end{proof}

In Theorem \ref{t2-2}, we only consider the case of $4\leq d\leq n-2$.
For $d=2$ or $3$, we can easily obtained that:
(1) If $G\in \mathcal{U}(n,2)$, then $SO_{1}(G)\leq SO_{1}(R_{n,2}^{1})$;
(2) If $G\in \mathcal{U}(n,3)$, then $SO_{1}(G)\leq SO_{1}(R_{n,3}^{1})$.

\section{Trees with given matching number}
\hskip 0.6cm
Let $\beta(G)$ be the matching number which is the number of edges in a maximum matching of a graph $G$.
Let $\mathcal{M}(n,\beta)$ be the set of trees with $n$ vertices and matching number $\beta$.
Let $M_{n}^{\beta}$ be the tree obtained from star $S_{n-\beta+1}$ by subdividing $\beta-1$ pendent vertices of $S_{n-\beta+1}$. It is obvious that $M_{2\beta}^{\beta}$ has a perfect matching. Let $M$ be a matching, we call a vertex is M-saturated if the vertex is incident with an edge of $M$.

\begin{lemma}\label{l4-1}{\rm\cite{huli2002}}
\rm{$(i)$} If $\beta \geq 2$ and $T\in \mathcal{M}(2\beta,\beta)$, then $T$ has a pendent vertex whose unique neighbor is of degree two;
\rm{$(ii)$} If $n>2\beta$ and $T\in \mathcal{M}(n,\beta)$, then there exists a $\beta$-matching $M$ and a pendent vetex $v$ such that $M$ does not saturate $v$.
\end{lemma}

By Lemma \ref{l1-1}, we can easily obtained the following result.
\begin{lemma}\label{l4-2}
Let $M$ is a $\beta$-matching of $T$ and $T\in \mathcal{M}(n,\beta)$ has the maximum first Sombor index $SO_{1}(T)$. Then

\rm{$(i)$} If $u\notin PV(T)$, then $u$ is $M$-saturated;

\rm{$(ii)$} If $uv\in M$, then $uv$ is a pendent edge of $T$.
\end{lemma}

\begin{theorem}\label{t4-3}
Let $T\in \mathcal{M}(2\beta,\beta)$. Then
$ SO_{1}(T)\leq \frac{1}{2}(\beta-1)\beta(\beta+1),$
with equality if and only if $T\cong M_{2\beta}^{\beta}$.
\end{theorem}
\begin{proof}
We prove the result by using mathematical induction on $\beta$.
It is obvious that the conclusion holds for $\beta=1$ or $2$.
Suppose the conclusion holds for $\mathcal{M}(2k,k)$ where $k<\beta$.
Let $T\in \mathcal{M}(2\beta,\beta)$.
Since $T\in \mathcal{M}(2\beta,\beta)$ and Lemma \ref{l4-2}, then there exists a pendent vertex $u$ and $N_{T}(u)=v$, $d_{T}(u)=2$.
Let $T^{*}=T-u-v$. Then $T^{*}\in \mathcal{M}(2\beta-2,\beta-1)$. Thus $SO_{1}(T^{*})\leq \frac{1}{2}(\beta-2)(\beta-1)\beta$.

Let $z\in N_{T}(v)\setminus \{u\}$, $d_{T}(z)=t$, $N_{T}(z)\setminus \{v\}=\{z_{1},z_{2},\cdots,z_{t-1}\}$ and $d_{T}(z_{i})=k_{i}$.
By Lemma \ref{l4-2}, we have $t\leq n-1-(\beta-1)=2\beta-1-(\beta-1)=\beta$. Then
\begin{align*}
2SO_{1}(T)=&  2SO_{1}(T^{*})+3+(t^{2}-2^{2})+\sum_{i=1}^{t-1}[|t^{2}-k_{i}^{2}|-|(t-1)^{2}-k_{i}^{2}|] \\
 \leq &  2SO_{1}(T^{*})+3+(t^{2}-2^{2})+(t-1)(t^{2}-(t-1)^{2}) \\
 \leq &  (\beta-2)(\beta-1)\beta+3+\beta^{2}-4+(\beta-1)(2\beta-1) \\
  =   &  (\beta-1)\beta(\beta+1)\beta,&
\end{align*}
the above equalities hold simultaneously if and only if $T^{*}\cong M_{2\beta-2}^{\beta-1}$ and $t=\beta$, i.e., $T\cong M_{2\beta}^{\beta}$.
\end{proof}

\begin{theorem}\label{t4-4}
Let $T\in \mathcal{M}(n,\beta)$. Then
$ SO_{1}(T)\leq \frac{1}{2}(n-\beta-1)(n-\beta)(n-\beta+1),$
with equality if and only if $T\cong M_{n}^{\beta}$.
\end{theorem}
\begin{proof}
It is obvious that the conclusion holds for $\beta=2$ or $3$.
By Theorem \ref{t4-4}, the conclusion holds for $n=2\beta$.
Suppose the conclusion holds for $\mathcal{M}(k,\beta)$ where $k<n$ and $k>2\beta$.
Let $T\in \mathcal{M}(n,\beta)$ and $M$ is a $\beta$-matching of $T$.
Since $T\in \mathcal{M}(n,\beta)$ and Lemma \ref{l4-2}, then there exists a pendent vertex $u$
such that $T-u$ also has a $\beta$-matching.
Let $N_{T}(u)=z$, $d_{T}(z)=$, $N_{T}(z)\setminus \{u\}=\{z_{1},z_{2},\cdots,z_{t-1}\}$ and $d_{T}(z_{i})=k_{i}$.
Let $T^{*}=T-u$. Then $T^{*}\in \mathcal{M}(n-1,\beta)$. Thus $SO_{1}(T^{*})\leq \frac{1}{2}(n-\beta-2)(n-\beta-1)(n-\beta)$. By Lemma \ref{l4-2}, We also have that $t\leq n-1-(\beta-1)=n-\beta$. Then
\begin{align*}
2SO_{1}(T)=&  2SO_{1}(T^{*})+(t^{2}-1^{2})+\sum_{i=1}^{t-1}[|t^{2}-k_{i}^{2}|-|(t-1)^{2}-k_{i}^{2}|] \\
 \leq &  2SO_{1}(T^{*})+(t^{2}-1^{2})+(t-1)(2t-1)  \\
 \leq &  (n-\beta-2)(n-\beta-1)(n-\beta)+(n-\beta)^{2}-1+(n-\beta-1)(2(n-\beta)-1)\\
  =   &  (n-\beta-1)(n-\beta)(n-\beta+1),&
\end{align*}
the above equalities hold simultaneously if and only if $T^{*}\cong M_{n-1}^{\beta}$ and $t=n-\beta$, i.e., $T\cong M_{n}^{\beta}$.
\end{proof}

\section{Trees with given number of pendent vertices}
\hskip 0.6cm
Let $\mathcal{Y}(n,p)$ be the set of trees with $n$ vertices and $p$ pendent vertices.
Let $Y_{n}^{p}$ be the broom graph which is obtained from star $S_{p}$ by replacing one pendent edge by a path of length $n-p$.

\begin{theorem}\label{t5-1}
Let $T\in \mathcal{Y}(n,p)$. Then
$ SO_{1}(T)\leq \frac{1}{2}(p-1)p(p+1),$
with equality if and only if $T\cong Y_{n}^{p}$.
\end{theorem}
\begin{proof}
We prove the result by using mathematical induction on $p$.
It is obvious that the conclusion holds for $p=2$ or $n-1$. Suppose $3\leq p\leq n-2$.
Suppose the conclusion holds for $p'<p$ and $n'<n$.
Let $T\in \mathcal{Y}(n,p)$ and $v\in VP(T)$.
Let $N_{T}(v)=u$, $N_{T}(u)\setminus \{v\}=\{u_{1},u_{2},\cdots,u_{t-1}\}$. Then $d_{T}(u)=t\leq p$.
Let $T^{*}=T-u$.

If $d_{T}(u)=2$. Then $T^{*}\in \mathcal{Y}(n-1,p)$. Thus $SO_{1}(T^{*})\leq \frac{1}{2}(p-2)(p-1)p$. Then $2SO_{1}(T^{*})=2SO_{1}(T)+3+|d_{T}(u_{1})^{2}-2^{2}|-|d_{T}(u_{1})^{2}-1^{2}|
=(p-2)(p-1)p<(p-1)p(p+1)$.

If $d_{T}(u)\geq 3$. Then $T^{*}\in \mathcal{Y}(n-1,p-1)$. Thus $SO_{1}(T^{*})\leq \frac{1}{2}(p-2)(p-1)p$. Then
\begin{align*}
2SO_{1}(T)=&  2SO_{1}(T^{*})+(t^{2}-1^{2})+\sum_{i=1}^{t-1}[|t^{2}-d_{T}(u_{i})^{2}|-|(t-1)^{2}-d_{T}(u_{i})^{2}|] \\
 \leq &  2SO_{1}(T^{*})+(t^{2}-1^{2})+(t-1)(2t-1)  \\
 \leq &  (p-2)(p-1)p+p^{2}-1+(p-1)(2p-1)\\
  =   &  (p-1)p(p+1),&
\end{align*}
the above equalities hold simultaneously if and only if $T^{*}\cong Y_{n-1}^{p-1}$ and $t=p$, i.e., $T\cong Y_{n}^{p}$.
\end{proof}

Further, we have
\begin{theorem}\label{t5-2}
Let $T$ be a tree with $p$ pendent vertices. Then
$ SO_{1}(T)\leq \frac{1}{2}(p-1)p(p+1).$
\end{theorem}

\section{Trees with given branching number}
\hskip 0.6cm
A vertex $v$ with $d(v)\geq 3$ is called a branching vertex.
Let $b(T)$ be the branching number which is the number of vertices $v\in V(T)$ with $d(v)\geq 3$.
Let $\mathcal{H}(n,b)$ be the set of trees with $n$ vertices and $b$ branching vertices, where $1\leq b\leq \frac{n-2}{2}$.
Let $H_{max}^{b}$ (resp. $H_{min}^{b}$) be the tree with maximum (resp. minimum) $SO_{1}$ in $\mathcal{H}(n,b)$.

In the following, we first consider the properties of $H_{min}^{b}$.

\begin{lemma}\label{l6-1}
Let $1\leq b\leq \frac{n-2}{2}$. Then $\Delta(H_{min}^{b})\leq 3$.
\end{lemma}
\begin{proof}
Suppose to the contrary that $\Delta(H_{min}^{b})\geq 4$.
Let $u\in V(H_{min}^{b})$ and $d_{}(u)=\Delta(H_{min}^{b})\triangleq x\geq 4$ and
{\bf there are at most one vertex with degree $\Delta(H_{min}^{b})$ in $N_{H_{min}^{b}}(u)$} (since $T$ is a tree, we always find the vertex $u$ satisfying the conditions).

Suppose that $P=v_{0}v_{1}\cdots v_{i-1}(=u_{x-1})u(=v_{i})v_{i+1}(=u_{x})\cdots v_{l}$ is a longest path containing $u$ in $H_{min}^{b}$.
Let $N_{H_{min}^{b}}(u)=\{u_{1},u_{2},\cdots,u_{x-2},u_{x-1},u_{x}\}$, $d_{H_{min}^{b}}(u_{i})=k_{i}$ for $1\leq i\leq x$.

Let $w_{1}\in V(H_{min}^{b})\setminus V(P)$ be a  pendent vertex connected to $u$ via $u_{1}$, i.e., there is a path $w_{1}w_{2}\cdots w_{r}u$ in $H_{min}^{b}$. Suppose $k_{2}<x$.

Let $T^{*}=H_{min}^{b}-u_{2}u+u_{2}w_{1}$. Then $T^{*}\in \mathcal{H}(n,b)$.

\noindent {\bf Case 1}. $w_{1}\neq u_{1}$.
\begin{eqnarray*}
&  & 2SO_{1}(T^{*})-2SO_{1}(H_{min}^{b})\\
& = & |k_{2}^{2}-2^{2}|-|x^{2}-k_{2}^{2}|+|d_{T^{*}}^{2}(w_{2})-2^{2}|
-|d_{H_{min}^{b}}^{2}(w_{2})-1^{2}| \\
& \quad & +\sum_{i=1,i\neq2}^{x}(|(x-1)^{2}-k_{i}^{2}|-|x^{2}-k_{i}^{2}|)\\
& \leq &  |k_{2}^{2}-4|-x^{2}+k_{2}^{2}-3+(x-3)[(x-1)^{2}-x^{2}] \\
& = & [|k_{2}^{2}-4|+k_{2}^{2}-2x^{2}]-x^{2}+7x-6<0.
\end{eqnarray*}

\noindent {\bf Case 2}. $w_{1}= u_{1}$.
\begin{eqnarray*}
&  & 2SO_{1}(T^{*})-2SO_{1}(H_{min}^{b})\\
& = & |k_{2}^{2}-2^{2}|-|x^{2}-k_{2}^{2}|+|x^{2}-2^{2}|
-|x^{2}-1^{2}| \\
& \quad & +\sum_{i=3}^{x}(|(x-1)^{2}-k_{i}^{2}|-|x^{2}-k_{i}^{2}|)\\
& \leq &  |k_{2}^{2}-4|-x^{2}+k_{2}^{2}-3+(x-4)[(x-1)^{2}-x^{2}] \\
& = & [|k_{2}^{2}-4|+k_{2}^{2}-2x^{2}]-x^{2}+9x-7<0.
\end{eqnarray*}

This is a contradiction, thus $\Delta(H_{min}^{b})\leq 3$.
\end{proof}

By Handshaking Lemma and Lemma \ref{l6-1}, we have
\begin{theorem}\label{t6-2}
Let $1\leq b\leq \frac{n-2}{2}$. Then the graph $H_{min}^{b}$ has the degree sequence $\pi=\{\underbrace{3,3,\cdots,3}_{b},\underbrace{2,2,\cdots,2}_{n-2b-2},\underbrace{1,1,\cdots,1}_{b+2}\}$.
\end{theorem}

Since $H_{min}^{b}$ has the degree sequence $\pi=\{\underbrace{3,3,\cdots,3}_{b},\underbrace{2,2,\cdots,2}_{n-2b-2},\underbrace{1,1,\cdots,1}_{b+2}\}$.
Then $m_{1,2}=b+2$, $m_{2,2}=n-3b-4$, $m_{2,3}=b+2$, $m_{3,3}=b-1$ for $b\leq \frac{n-3}{3}$.
$m_{1,2}=n-2b-2$, $m_{1,3}=3b-n+4$, $m_{2,3}=n-2b-2$, $m_{3,3}=b-1$ for $b> \frac{n-3}{3}$.
Thus we have $SO_{1}(H_{min}^{b})=4b+8$.

\begin{theorem}\label{t6-3}
Let $1\leq b\leq \frac{n-2}{2}$ and $T\in \mathcal{H}(n,b)$. Then $SO_{1}(T)\geq 4b+8$,
with equality if and only if $T$ has the degree sequence $\pi=\{\underbrace{3,3,\cdots,3}_{b},\underbrace{2,2,\cdots,2}_{n-2b-2},\underbrace{1,1,\cdots,1}_{b+2}\}$.
\end{theorem}

By Theorem \ref{t6-3}, we have the following corollary.

\begin{corollary}\label{c6-4}
$12=SO_{1}(H_{min}^{1})<SO_{1}(H_{min}^{2})<\cdots<SO_{1}(H_{min}^{\lfloor \frac{n-2}{2}\rfloor})$.
\end{corollary}

In the following, we consider the properties of $H_{max}^{b}$.

\begin{lemma}\label{l6-5}
Let $1\leq b\leq \frac{n-2}{2}$. Then $n_{2}(H_{max}^{b})=0$.
\end{lemma}
\begin{proof}
Suppose to the contrary that $n_{2}(H_{max}^{b})\geq 1$, then there exists $u\in V(H_{max}^{b})$
and $d_{H_{max}^{b}}(u)=2$. Let $N_{H_{max}^{b}}(u)=\{v_{1},v_{2}\}$.
Let $w\in V(H_{max}^{b})$
and $d_{H_{max}^{b}}(w)=\Delta(H_{max}^{b})$ and $N_{H_{max}^{b}}(w)=\{w_{1},w_{2},\cdots,w_{\Delta}\}$.

If $w\in N_{H_{max}^{b}}(u)$ , let $T^{*}=N_{H_{max}^{b}}-uv_{2}+v_{1}v_{2}$, then $T^{*}\in \mathcal{H}(n,b)$.
By Edge-lifting transformation of Lemma \ref{l1-1}, we have $SO_{1}(T^{*})>SO_{1}(H_{max}^{b})$, which is a contradiction.

If $w\notin N_{H_{max}^{b}}(u)$, then there exists a path contain $u$, $w$ in $H_{max}^{b}$. Suppose the path go through $v_{1}$, then $d_{H_{max}^{b}}(v_{1})\geq 2$. Let $T^{*}=N_{H_{max}^{b}}-uv_{2}+wv_{2}$, then $T^{*}\in \mathcal{H}(n,b)$.
\begin{eqnarray*}
&  & 2SO_{1}(T^{*})-2SO_{1}(H_{min}^{b})\\
& = & |d_{T^{*}}^{2}(w)-d_{T^{*}}^{2}(v_{2})|+|d_{T^{*}}^{2}(v_{1})-d_{T^{*}}^{2}(u)|
-|d_{H_{max}^{b}}^{2}(v_{1})-d_{H_{max}^{b}}^{2}(u)|
-|d_{H_{max}^{b}}^{2}(v_{2})-d_{H_{max}^{b}}^{2}(u)| \\
& \quad & +\sum_{i=1}^{\Delta}[|(d_{H_{max}^{b}}(w)+1)^{2}-d_{H_{max}^{b}}^{2}(w_{i})|-
|d_{H_{max}^{b}}^{2}(w)-d_{H_{max}^{b}}^{2}(w_{i})|]\\
& = &  (\Delta+1)^{2}-d_{H_{max}^{b}}^{2}(v_{2})+\Delta(2\Delta+1)+d_{H_{max}^{b}}^{2}(v_{1})
-1-d_{H_{max}^{b}}^{2}(v_{1})+4-|d_{H_{max}^{b}}^{2}(v_{2})-4| \\
& = & (\Delta+1)^{2}+\Delta(2\Delta+1)+3-d_{H_{max}^{b}}^{2}(v_{2})
-|d_{H_{max}^{b}}^{2}(v_{2})-4|>0.
\end{eqnarray*}

Thus $n_{2}(H_{max}^{b})=0$.
\end{proof}

\begin{lemma}\label{l6-6}
Let $1\leq b\leq \frac{n-2}{2}$ and $i\geq 4$. Then $n_{i}(H_{max}^{b})\leq 1$.
\end{lemma}
\begin{proof}
Suppose that there exists vertices $u,v\in V(H_{max}^{b})$ such that $d_{H_{max}^{b}}(u)=\Delta(H_{max}^{b})\geq 4$ and $d_{H_{max}^{b}}(u)\geq 4$.
Let $d_{H_{max}^{b}}(u)=t$ and $N_{H_{max}^{b}}(v)=\{v_{1},v_{2},\cdots,v_{t}\}$.

If $u\notin N_{H_{max}^{b}}(v)$, then let $u=v_{t}$. Let $T^{*}=H_{max}^{b}-\{vv_{1},vv_{2},\cdots,vv_{t-3}\}+\{uv_{1},uv_{2},\cdots,uv_{t-3}\}$, then $T^{*}\in \mathcal{H}(n,b)$.
\begin{eqnarray*}
&  & 2SO_{1}(T^{*})-2SO_{1}(H_{max}^{b})\\
& = & \sum_{i=1}^{t-3}[|d_{T^{*}}^{2}(u)-d_{T^{*}}^{2}(v_{i})|
-|d_{H_{max}^{b}}^{2}(v)-d_{H_{max}^{b}}^{2}(v_{i})|]
+|(\Delta+t-3)^{2}-(d_{H_{max}^{b}}(v_{i})-(t-3))^{2}|\\
& \quad & -|\Delta^{2}-d_{H_{max}^{b}}^{2}(v)|+(\Delta-1)[(\Delta+t-3)^{2}-\Delta^{2}]\\
& \quad & +\sum_{i=t-2}^{t-1}[|(d_{H_{max}^{b}}-(t-3))^{2}-d_{H_{max}^{b}}^{2}(v_{i})|
-|d_{H_{max}^{b}}^{2}(v)-d_{H_{max}^{b}}^{2}(v_{i})|]\\
& \geq &  -(t-3)[(\Delta+t-3)^{2}-t^{2}]+2(t-3)(\Delta+t)+(\Delta-1)[2\Delta(t-3)+(t-3)^{2}] \\
& \quad & -2(t-3)(2t-(t-3))\\
& = &  (t-3)(-t\Delta+5t+\Delta^{2}+3\Delta-12)>0.
\end{eqnarray*}

If $u\in N_{H_{max}^{b}}(v)$. Let $T^{*}=H_{max}^{b}-\{vv_{1},vv_{2},\cdots,vv_{t-3}\}+\{uv_{1},uv_{2},\cdots,uv_{t-3}\}$, then $T^{*}\in \mathcal{H}(n,b)$.
\begin{eqnarray*}
&  & 2SO_{1}(T^{*})-2SO_{1}(H_{max}^{b})\\
& = & \sum_{i=1}^{t-3}[|d_{T^{*}}^{2}(u)-d_{T^{*}}^{2}(v_{i})|
-|d_{H_{max}^{b}}^{2}(v)-d_{H_{max}^{b}}^{2}(v_{i})|] +\Delta[(\Delta+t-3)^{2}-\Delta^{2}]\\
& \quad & +\sum_{i=t-2}^{t}[|(d_{H_{max}^{b}}-(t-3))^{2}-d_{H_{max}^{b}}^{2}(v_{i})|
-|d_{H_{max}^{b}}^{2}(v)-d_{H_{max}^{b}}^{2}(v_{i})|]\\
& \geq &  -(t-3)[(\Delta+t-3)^{2}-t^{2}]+\Delta[2\Delta(t-3)+(t-3)^{2}] -3(t-3)(t+3)\\
& = &  (t-3)(-t\Delta+3t+\Delta^{2}+3\Delta-18)>0.
\end{eqnarray*}

This is a contradiction. Thus $n_{i}(H_{max}^{b})\leq 1$.
\end{proof}

\begin{theorem}\label{t6-7}
Let $1\leq b\leq \frac{n-2}{2}$. Then the graph $H_{max}^{b}$ has the degree sequence $\pi=\{n-2b+1,\underbrace{3,3,\cdots,3}_{b-1},\underbrace{1,1,\cdots,1}_{n-b}\}$, and further
$SO_{1}(H_{max}^{b})=\frac{1}{2}[(n-2b+1)^{3}-9(n-2b+1)+8(n-b)]$.
\end{theorem}
\begin{proof}
Since $n_{1}=2(n-1)-(n-2b+1)-3(b-1)=n-b$ and $\Delta\leq n-2b+1$, then $H_{max}^{b}$ has the degree sequence $\pi=\{n-2b+1,\underbrace{3,3,\cdots,3}_{b-1},\underbrace{1,1,\cdots,1}_{n-b}\}$

By the structure of tree and the properties of $SO_{1}$, we have
$2SO_{1}(H_{max}^{b})=(n-2b+1)[(n-2b+1)^{2}-1]+[n-b-(n-2b+1)](3^{2}-1^{2})
=(n-2b+1)^{3}-9(n-2b+1)+8(n-b)$.
\end{proof}

\begin{theorem}\label{t6-8}
Let $1\leq b\leq \frac{n-2}{2}$ and $T\in \mathcal{H}(n,b)$. Then $SO_{1}(T)\leq \frac{1}{2}[(n-2b+1)^{3}-9(n-2b+1)+8(n-b)]$,
with equality if and only if $T$ has the degree sequence $\pi=\{n-2b+1,\underbrace{3,3,\cdots,3}_{b-1},\underbrace{1,1,\cdots,1}_{n-b}\}$
\end{theorem}

By Theorem \ref{t6-8}, we have the following corollary.

\begin{corollary}\label{c6-9}
$SO_{1}(H_{max}^{1})>SO_{1}(H_{max}^{2})>\cdots>SO_{1}(H_{max}^{\lfloor \frac{n-2}{2}\rfloor})$.
\end{corollary}

\section{Trees and connected graphs with given maximum degree}
\hskip 0.6cm
Recall that $\Delta$ (resp. $\Delta_{2}$) is the maximum (resp. second maximum degree) in graph $G$.
A vertex $v$ with $d(v)\geq 3$ is called a branching vertex.
If a tree has only one branching vertex, then we call it a starlike tree.
Then we call the branching vertex as the center of this starlike tree.
Let $\mathcal{ST}_{n,\Delta}$ be the set of starlike trees with $n$ vertices and maximum degree $\Delta$.
Let $\mathcal{DST}_{n,\Delta,\Delta_{2}}$ be the $n$-verices graphs obtained from two starlike trees with maximum degree $\Delta-1$ and $\Delta_{2}-1$ by adding one edge between their centers.

We denote a path $v_{1}v_{2}\cdots v_{t}$ by $(v_{1},v_{t})$.
Let $\phi_{G}(v_{1},v_{t})=\sum\limits_{i=1}^{t}|d_{G}^{2}(v_{i})-d_{G}^{2}(v_{i+1})|$.
Then it is obvious that $\phi_{G}(v_{1},v_{t})\geq |d_{G}^{2}(v_{1})-d_{G}^{2}(v_{t})|$, with equality if and only if $d(v_{1})\geq d(v_{2})\geq \cdots \geq d(v_{t})$ or $d(v_{1})\leq d(v_{2})\leq \cdots \leq d(v_{t})$.

Let $\mathcal{T}_{n,\Delta,\Delta_{2}}$ be the set of trees with $n$ vertices, maximum degree $\Delta$ and second maximum degree $\Delta_{2}$.

\begin{theorem}\label{t7-1}
Let $\Delta \geq 3$ and $T\in \mathcal{T}_{n,\Delta,\Delta_{2}}$. Then $SO_{1}(T)\geq
\frac{1}{2}[(\Delta-1)(\Delta^{2}-1)+(\Delta_{2}-1)(\Delta_{2}^{2}-1)+\Delta^{2}-\Delta_{2}^{2}]$
, with equality if and only if $T\in \mathcal{DST}_{n,\Delta,\Delta_{2}}$.
\end{theorem}
\begin{proof}
Suppose that $d(u)=\Delta$, $d(v)=\Delta_{2}$. Since $T$ is a tree, there exist a $(u,v)$ path in $T$.
Then there exists $\Delta-1$ edge-disjoint path $(u,u_{i})$ and $d(u_{i})=1$ for $i=1,2,\cdots,\Delta-1$. There exists $\Delta_{2}-1$ edge-disjoint path $(v,v_{i})$ and $d(v_{i})=1$ for $i=1,2,\cdots,\Delta_{2}-1$. Thus
\begin{align*}
2SO_{1}(T)\geq&  \sum_{i=1}^{\Delta-1}\phi_{T}(u,u_{i})+\sum_{j=1}^{\Delta_{2}-1}\phi_{T}(v,v_{j})
+\phi_{T}(u,v)\\
 \geq &  \sum_{i=1}^{\Delta-1}(d_{T}^{2}(u)-d_{T}^{2}(u_{i}))+\sum_{j=1}^{\Delta_{2}-1}(d_{T}^{2}(v)-d_{T}^{2}(v_{j}))
+(d_{T}^{2}(u)-d_{T}^{2}(v))  \\
  =   & (\Delta-1)(\Delta^{2}-1)+(\Delta_{2}-1)(\Delta_{2}^{2}-1)+\Delta^{2}-\Delta_{2}^{2}  ,&
\end{align*}
with equality if and only if $T\in \mathcal{DST}_{n,\Delta,\Delta_{2}}$.
\end{proof}

Let $\mathcal{T}_{n,\Delta}$ be the set of trees with $n$ vertices, maximum degree $\Delta$.
Let $f(x)=(\Delta-1)(\Delta^{2}-1)+(x-1)(x^{2}-1)+\Delta^{2}-x^{2}$. It is obvious that $f'(x)>0$ for $x\geq 2$. Thus $\min f(x)=\min\{f(1),f(2)\}=\Delta^{3}-\Delta$. Thus
\begin{corollary}\label{c7-2}
Let $\Delta \geq 3$ and $T\in \mathcal{T}_{n,\Delta}$. Then $SO_{1}(T)\geq
\frac{1}{2}(\Delta^{3}-\Delta)$
, with equality if and only if $T\in \mathcal{ST}_{n,\Delta}$.
\end{corollary}

Further, we have
\begin{corollary}\label{c7-3}{\rm \cite{ybai2020}}
Let $n \geq 3$ and $T\in \mathcal{T}_{n}$. Then $SO_{1}(T)\geq 3$
, with equality if and only if $T\cong P_{n}$.
\end{corollary}

Let $\mathcal{A}_{n}(g,\Delta)$ be the set of unicyclic graphs obtained by identifying one pendent vertex of a starlike trees in $\mathcal{ST}_{n-g+1,\Delta}$ with one vertex of cycle $C_{g}$.
Let $\mathcal{B}_{n}(g,\Delta)$ be the set of $n$-vertices unicyclic graphs obtained by attaching $\Delta-2$ paths to one vertex of cycle $C_{g}$.
Suppose $G\in \mathcal{A}_{n}(g,\Delta)$ and $l$ is the length of shortest path from the center of starlike tree to $C_{g}$. Then $2SO_{1}(G)=\Delta^{3}-\Delta+2$ if $l=1$; $2SO_{1}(G)=\Delta^{3}-\Delta+12$ if $l\geq 2$.
Suppose $G\in \mathcal{B}_{n}(g,\Delta)$. Then $2SO_{1}(G)=\Delta^{3}-\Delta-6$.
Let $\mathcal{U}_{n,\Delta,g}$ be the set of unicyclic graphs with $n$ vertices, maximum degree $\Delta$ and girth $g$.
\begin{theorem}\label{t7-4}
Let $G\in \mathcal{U}_{n,\Delta,g}$. Then $SO_{1}(G)\geq
\frac{1}{2}(\Delta^{3}-\Delta-6)$
, with equality if and only if $G\in \mathcal{B}_{n}(g,\Delta)$.
\end{theorem}
\begin{proof}
Suppose that $d(u)=\Delta$ and $v$ is the nearest vertex apart from $u$ in $C_{g}$.
Let $d_{G}(u,v)=l$ and $N_{C_{g}}(v)=\{w,z\}$.
Then there exists $d_{G}(w)-2$ edge-disjoint path $(w,w_{i})$ and $d(w_{i})=1$ for $i=1,2,\cdots,d_{G}(w)-2$. There exists $d_{G}(z)-2$ edge-disjoint path $(z,z_{i})$ and $d(z_{i})=1$ for $i=1,2,\cdots,d_{G}(z)-2$.

If $l=0$, then there exists $\Delta-2$ edge-disjoint path $(u,u_{i})$ and $d(u_{i})=1$ for $i=1,2,\cdots,\Delta-2$. Then
\begin{align*}
2SO_{1}(T)\geq&  \sum_{i=1}^{\Delta-2}\phi_{G}(u,u_{i})+\sum_{i=1}^{d_{G}(w)-2}\phi_{G}(w,w_{i})
+\sum_{i=1}^{d_{G}(z)-2}\phi_{G}(z,z_{i})+\phi_{G}(u,w)+\phi_{G}(u,z)\\
 \geq &  \sum_{i=1}^{\Delta-2}|d_{G}^{2}(u)-d_{G}^{2}(u_{i})|+\sum_{i=1}^{d_{G}(w)-2}
 |d_{G}^{2}(w)-d_{G}^{2}(w_{i})|+\sum_{i=1}^{d_{G}(z)-2}|d_{G}^{2}(z)-d_{G}^{2}(z_{i})|  \\
\quad &  +|d_{G}^{2}(u)-d_{G}^{2}(w)|+|d_{G}^{2}(u)-d_{G}^{2}(z)|\\
\geq &  (\Delta-2)(\Delta^{2}-1)+ (d_{G}(w)-2)(d_{G}^{2}(w)-1)+(d_{G}(z)-2)(d_{G}^{2}(z)-1)       \\
\quad &  +\Delta^{2}-d_{G}^{2}(v)+\Delta^{2}-d_{G}^{2}(z)    \\
 =   &   (\Delta-2)(\Delta^{2}-1)+ 2\Delta^{2}-8+(d_{G}(w)-2)(d_{G}^{2}(w)-d_{G}(w)-3)           \\
 \quad & +(d_{G}(z)-2)(d_{G}^{2}(z)-d_{G}(z)-3)\\
 \geq   & \Delta^{3}-\Delta-6  ,&
\end{align*}
with equality if and only if $G\in \mathcal{B}_{n}(g,\Delta)$.

If $l\geq 1$, then there exists $\Delta-1$ edge-disjoint path $(u,u_{i})$ and $d(u_{i})=1$ for $i=1,2,\cdots,\Delta-1$. Then
\begin{align*}
2SO_{1}(T)\geq&  \sum_{i=1}^{\Delta-1}\phi_{G}(u,u_{i})+\sum_{i=1}^{d_{G}(w)-2}\phi_{G}(w,w_{i})
+\sum_{i=1}^{d_{G}(z)-2}\phi_{G}(z,z_{i})+\phi_{G}(u,v)+\phi_{G}(v,w)+\phi_{G}(v,z)\\
 \geq &  \sum_{i=1}^{\Delta-1}|d_{G}^{2}(u)-d_{G}^{2}(u_{i})|+\sum_{i=1}^{d_{G}(w)-2}
 |d_{G}^{2}(w)-d_{G}^{2}(w_{i})|
+\sum_{i=1}^{d_{G}(z)-2}|d_{G}^{2}(z)-d_{G}^{2}(z_{i})|  \\
\quad &  +|d_{G}^{2}(u)-d_{G}^{2}(v)|+|d_{G}^{2}(v)-d_{G}^{2}(w)|+|d_{G}^{2}(v)-d_{G}^{2}(z)|\\
\geq &  (\Delta-1)(\Delta^{2}-1)+ (d_{G}(w)-2)(d_{G}^{2}(w)-1)+(d_{G}(z)-2)(d_{G}^{2}(z)-1)       \\
\quad &  +\Delta^{2}-d_{G}^{2}(v)+d_{G}^{2}(v)-d_{G}^{2}(w)+d_{G}^{2}(v)-d_{G}^{2}(z)    \\
 =   &   (\Delta-1)(\Delta^{2}-1)+ \Delta^{2}+ d_{G}^{2}(v)-8+(d_{G}(w)-2)(d_{G}^{2}(w)-d_{G}(w)-3)           \\
 \quad & +(d_{G}(z)-2)(d_{G}^{2}(z)-d_{G}(z)-3)\\
 \geq   & \Delta^{3}-\Delta+2  .&
\end{align*}

Thus $SO_{1}(G)\geq
\frac{1}{2}(\Delta^{3}-\Delta-6)$
, with equality if and only if $G\in \mathcal{B}_{n}(g,\Delta)$.
\end{proof}

Let $\mathcal{U}_{n,\Delta}$ be the set of unicyclic graphs with $n$ vertices and maximum degree $\Delta$. By Theorem \ref{t7-4}, we have
\begin{corollary}\label{c7-5}
Let $G\in \mathcal{U}_{n,\Delta}$. Then $SO_{1}(G)\geq
\frac{1}{2}(\Delta^{3}-\Delta-6)$
, with equality if and only if $G\in \mathcal{B}_{n}(g,\Delta)$ where $g\in \{3,4,\cdots,n-\Delta+2\}$.
\end{corollary}

Let $\mathcal{U}_{n,g}$ be the set of unicyclic graphs with $n$ vertices and girth $g$.
Note that if $G\ncong C_{n}$, then $\Delta\geq 3$. Then by Theorem \ref{t7-4}, we have
\begin{corollary}\label{c7-6}
Let $G\in \mathcal{U}_{n,g}$ and $G\ncong C_{n}$. Then $SO_{1}(G)\geq 9$
, with equality if and only if $G\in \mathcal{B}_{n}(g,3)$.
\end{corollary}

Let $\mathcal{G}_{n,k,\Delta}$ be the set of connected graphs with $n$ vertices, maximum degree $\Delta$ and $k$ pendent vertices.
\begin{theorem}\label{t7-7}
Let $G\in \mathcal{U}_{n,k,\Delta}$ and $G\ncong P_{n}$. Then $SO_{1}(G)\geq
\frac{1}{2}(8k+\Delta^{2}-9)$.
\end{theorem}
\begin{proof}
Suppose that $d(u_{1})=d(u_{2})=\cdots=d(u_{k})=1$ and $(u_{1},v_{1}),(u_{2},v_{2}),\cdots,(u_{k},v_{k})$ be the $k$ pendent paths of $G$.
Let $d_{G}(w)=\Delta$. Then
$2SO_{1}(G)\geq \sum\limits_{i=1}^{k}\phi_{G}(u_{i},v_{i})+\phi_{G}(w,v_{1})\geq \sum\limits_{i=2}^{k}|d_{G}^{2}(v_{i})-1|+d_{G}^{2}(u_{1})-1+\Delta^{2}-d_{G}^{2}(u_{1})
\geq 8(k-1)+\Delta^{2}-1=8k+\Delta^{2}-9$.
\end{proof}

Let $\mathcal{G}_{n,k}$ be the set of connected graphs with $n$ vertices and $k$ pendent vertices.
\begin{corollary}\label{c7-8}
Let $G\in \mathcal{G}_{n,k}$ and $G\ncong P_{n}$. Then $SO_{1}(G)\geq4k$,
with equality if and only if $G\in \mathcal{U}_{n,k}$ $(1\leq k\leq \frac{n+2}{2})$, $\Delta=3$ and all $2$-vertex is in pendent paths.
\end{corollary}
\begin{proof}
Similar to the proof of Theorem \ref{t7-7},we have $2SO_{1}(G)\geq \sum\limits_{i=1}^{k}\phi_{G}(u_{i},v_{i})\geq \sum\limits_{i=1}^{k}|d_{G}^{2}(v_{i})-1|
\geq 8k$, with equality if and only if $G\in \mathcal{U}_{n,k}$ $(1\leq k\leq \frac{n+2}{2})$, $\Delta=3$ and all $2$-vertex is in pendent paths.
\end{proof}

In the following, we obtained the properties about the maximum connected graphs with $n$ vertices and $k$ pendent vertices with respect to $SO_{1}$.

\begin{theorem}\label{t7-9}
Let $1\leq k\leq n-3$ and $G\in \mathcal{G}_{n,k}$. If $G$ has the maximum $SO_{1}(G)$, then
$\Delta(G)=n-1$, $\Delta_{2}(G)=n-k-1$, and the vertex with degree $n-1$ is unique.
\end{theorem}
\begin{proof}
Suppose that $d_{u}=\Delta(G)$ and $d_{v}=\Delta_{2}(G)$. We first proof $\Delta(G)=n-1$.
If $\Delta(G)<n-1$, then there exists one vertex, say $w$, such that $uw\notin E(G)$.

If $N_{G}(w)\nsubseteq N_{G}(u)$, then there exists one vertex, say $z$, such that $wz\in E(G)$ and $uz\notin E(G)$. Let $G^{*}=G-wz+uz+uw$, then $G^{*}\in \mathcal{G}_{n,k}$ and
\begin{align*}
2SO_{1}(G^{*})-2SO_{1}(G)=&  (d_{u}+2)^{2}-d_{w}^{2}+(d_{u}+2)^{2}-d_{z}^{2}-|d_{w}^{2}-d_{z}^{2}|\\
\quad & +\sum_{x\in N_{G}(u)}\{((d_{u}+2)^{2}-d_{x}^{2})-(d_{u}^{2}-d_{x}^{2})\}\\
= &  2(d_{u}+2)^{2}+4d_{u}(d_{u}+1)-d_{z}^{2}-d_{w}^{2}-|d_{w}^{2}-d_{z}^{2}|  \\
\geq & 2(d_{u}+2)^{2}+4d_{u}(d_{u}+1)-2\max\{d_{z}^{2},d_{w}^{2}\}>0. \ \ (\text{Since $d_{u}=\Delta(G)$}) &
\end{align*}
This is a contradiction with that $G$ has the maximum $SO_{1}(G)$.

Next we consider $N_{G}(w)\subseteq N_{G}(u)$.

If $d_{w}\geq d_{x}$ for all $x\in N_{G}(w)$. Let $G^{*}=G+uw$. Since $d_{w}\geq d_{x}\geq 2$ for all $x\in N_{G}(w)$, then $G^{*}\in \mathcal{G}_{n,k}$. Since $d_{u}=\Delta(G)$, then
$2SO_{1}(G^{*})-2SO_{1}(G)=(d_{u}+1)^{2}-(d_{w}+1)^{2}+\sum\limits_{x\in N_{G}(u)}\{((d_{u}+1)^{2}-d_{x}^{2})-(d_{u}^{2}-d_{x}^{2})\}+\sum\limits_{y\in N_{G}(w)}\{((d_{w}+1)^{2}-d_{y}^{2})-(d_{w}^{2}-d_{y}^{2})\}>0$.
This is a contradiction with that $G$ has the maximum $SO_{1}(G)$.

If there exists one vertex $x_{0}\in N_{G}(w)$ such that $d_{w}< d_{x_{0}}$.

If $d_{x_{0}}\geq 3$, let $G^{*}=G-x_{0}w+uw$, then $G^{*}\in \mathcal{G}_{n,k}$. Then
$2SO_{1}(G^{*})-2SO_{1}(G)=(d_{u}+1)^{2}-d_{w}^{2}+\{(d_{u}+1)^{2}-(d_{x_{0}}-1)^{2}\}
-\{d_{u}^{2}-d_{x_{0}}^{2}\}-\{d_{x_{0}}^{2}-d_{w}^{2}\}+\sum\limits_{x\in N_{G}(u)\setminus \{x_{0}\}}\{((d_{u}+1)^{2}-d_{x}^{2})-(d_{u}^{2}-d_{x}^{2})\}\}+\sum\limits_{y\in N_{G}(x_{0})\setminus \{u,w\}}\{(|d_{x_{0}}-1)^{2}-d_{y}^{2}|-|d_{x_{0}}^{2}-d_{y}^{2}|\}\geq 3d_{u}^{2}+3d_{u}+1-2d_{x_{0}}-(d_{x_{0}}-1)^{2}\geq 3d_{u}+1>0$.
This is a contradiction with that $G$ has the maximum $SO_{1}(G)$.

If $d_{x_{0}}=2$, then $d_{w}=1$. Let one of the vertices with maximum degree in $N_{G}(u)\setminus \{x_{0}\}$ be $x_{1}$, then $d_{x_{1}}\geq 2$.
Let $G^{*}=G-x_{0}w+x_{0}x_{1}+uw$, then $G^{*}\in \mathcal{G}_{n,k}$. Then
$2SO_{1}(G^{*})-2SO_{1}(G)=(d_{u}+1)^{2}-1+(d_{u}+1)^{2}-4+\{(d_{u}+1)^{2}-(d_{x_{1}}+1)^{2}\}
+(d_{x_{1}}+1)^{2}-4-(d_{u}^{2}-d_{x_{1}^{2}})-(d_{u}^{2}-4)-(4-1)
+\sum\limits_{x\in N_{G}(u)\setminus \{x_{0},x_{1}\}}\{((d_{u}+1)^{2}-d_{x}^{2})-(d_{u}^{2}-d_{x}^{2})\}\}+\sum\limits_{y\in N_{G}(x_{1})\setminus \{u\}}\{(|d_{x_{1}}+1)^{2}-d_{y}^{2}|-|d_{x_{1}}^{2}-d_{y}^{2}|\}
\geq (d_{u}-2)(2d_{u}+1)-(d_{x_{1}}-1)(2d_{x_{1}}+1)\geq 3d_{u}-6+2d_{u}^{2}+d_{x_{1}}>0$.
This is a contradiction with that $G$ has the maximum $SO_{1}(G)$.

Thus $\Delta(G)=n-1$. Since $k\geq 1$, then the vertex with degree $n-1$ is unique.

Next we proof $\Delta_{2}(G)=n-k-1$.

If $d_{v}=\Delta_{2}(G)<n-k-1$. Let $U=\{x\in V(G)|d_{x}\neq 1, d_{x}\neq n-1\}$. Then $|U|=n-k-1\geq 2$ since $k\leq n-3$.
Then there exists one vertex, say $x_{0}$, such that $vx_{0}\notin E(G)$.

If $N_{G}(x_{0})\nsubseteq N_{G}(v)$, then there exists one vertex, say $x_{1}\notin N_{G}(v)$, such that $x_{0}x_{1}\in E(G)$. Let $G^{*}=G-x_{0}x_{1}+x_{0}v+x_{1}v$, then $G^{*}\in \mathcal{G}_{n,k}$ and
\begin{align*}
2SO_{1}(G^{*})-2SO_{1}(G)=&  (d_{v}+2)^{2}-d_{x_{1}}^{2}+(d_{v}+2)^{2}-d_{x_{0}}^{2}-|d_{x_{0}}^{2}-d_{x_{1}}^{2}|
+\{d_{u}^{2}-(d_{v}+2)^{2}\}\\
\quad & -\{d_{u}^{2}-d_{v}^{2}\}+\sum_{x\in N_{G}(v)\setminus \{u\}}\{((d_{v}+2)^{2}-d_{x}^{2})-(d_{v}^{2}-d_{x}^{2})\}\\
= &  4d_{v}^{2}+4d_{v}-d_{x_{1}}^{2}-d_{x_{0}}^{2}-|d_{x_{0}}^{2}-d_{x_{1}}^{2}|  \\
\geq & 4d_{v}^{2}+4d_{v}-2\max\{d_{x_{0}}^{2},d_{x_{1}}^{2}\}>0. \ \ (\text{Since $d_{v}=\Delta_{2}(G)$}) &
\end{align*}
This is a contradiction with that $G$ has the maximum $SO_{1}(G)$.

If $N_{G}(x_{0})\subseteq N_{G}(v)$, let $x_{2}\in N_{G}(x_{0})\setminus \{u\}$. Let $G^{*}=G-x_{0}x_{2}+x_{0}v$, then $G^{*}\in \mathcal{G}_{n,k}$.
\begin{align*}
2SO_{1}(G^{*})-2SO_{1}(G)=&  (d_{v}+1)^{2}-(d_{x_{2}}-1)^{2}+\{(d_{v}+1)^{2}-d_{x_{0}}^{2}\}-\{d_{v}^{2}-d_{x_{2}}^{2}\}\\
\quad & -|d_{x_{2}}^{2}-d_{x_{0}}^{2}|+\{d_{u}^{2}-(d_{v}+1)^{2}\}-\{d_{u}^{2}-d_{v}^{2}\}
+\{d_{u}^{2}-(d_{x_{2}}-1)^{2}\}\\
\quad & -\{d_{u}^{2}-d_{x_{2}}^{2}\}+\sum_{x\in N_{G}(v)\setminus \{u,x_{2}\}}\{((d_{v}+1)^{2}-d_{x}^{2})-(d_{v}^{2}-d_{x}^{2})\}\\
\quad & +\sum_{y\in N_{G}(x_{2})\setminus \{u,v,x_{0}\}}\{|(d_{x_{2}}-1)^{2}-d_{y}^{2}|-|d_{x_{2}}^{2}-d_{y}^{2}|\}\\
\geq &  3d_{v}^{2}+d_{v}-2d_{x_{2}}^{2}+11d_{x_{2}}-d_{x_{0}}^{2}-4-|d_{x_{2}}^{2}-d_{x_{0}}^{2}|  \\
\geq & d_{v}^{2}-d_{v}-d_{x_{2}}^{2}+11d_{x_{2}}-4+2d_{v}^{2}-2\max\{d_{x_{2}}^{2},d_{x_{0}}^{2}\}>0. &
\end{align*}
This is a contradiction with that $G$ has the maximum $SO_{1}(G)$.

Thus $\Delta_{2}(G)=n-k-1$. This completes the proof.
\end{proof}

Let $H_{n,k,t}$ be the graph obtained by connecting the center of star $S_{k+1}$ and every vertex of complete split graph $CS_{n-k-1,t}$ where $1\leq k\leq n-3$ and $1\leq t\leq n-k-2$.
Further,we propose the following conjecture.
\begin{conjecture}\label{c7-10}
Let $1\leq k\leq n-3$ and $G\in \mathcal{U}_{n,k}$. If $G$ has the maximum $SO_{1}(G)$, then
$G\cong H_{n,k,t}$, where $1\leq t\leq n-k-2$.
\end{conjecture}

\section{Regression models and chemical applicability}
\hskip 0.6cm
In this section, the linear regression models of Sombor-index-like invariants and boiling points (BP) of 21 benzenoid hydrocarbons are presented. The linear regression models of Sombor-index-like invariants and AcenFac (resp. Entropy, SNar, HNar, HVAP, DHVAP) of 18 octane isomers are also presented. Note that the relationship between AcenFac (resp. Entropy, SNar, HNar, HVAP, DHVAP) of octane isomers and $SO$ can also be find in \cite{absf2021,dengt2021,liyh2023,redz2021}. The relationship between boiling points (BP) of benzenoid hydrocarbons and $SO$ can also be find in \cite{lich2021,liyh2023}. For the sake of uniformity, we did not delete the relevant content about Sombor index. We also corrected some data in these papers.

The 21 benzenoid hydrocarbons and 18 octane isomers we considered are shown in Figure \ref{fig-71} and Figure \ref{fig-72}.
The experimental values of boiling points of benzenoid hydrocarbons of Table 1 are taken from \cite{mila2021}.
The experimental values of AcenFac (resp. Entropy, SNar, HNar, HVAP, DHVAP) of octane isomers of Table 2 are taken from \cite{dengt2021} and \cite{mila2021}.

With the data of Table \ref{table2}, scatter plots between BP and $SO$ (resp. $SO_{1}$, $SO_{2}$, $SO_{3}$, $SO_{4}$, $SO_{5}$, $SO_{6}$) are shown in Figure \ref{fig-73}.
The mathematical relationship-related coefficient ($R$) between boiling points and $SO$ (resp. $SO_{1}$, $SO_{2}$, $SO_{3}$, $SO_{4}$, $SO_{5}$, $SO_{6}$) is about 0.9929 (resp. 0.7905, 0.7905, 0.9930, 0.9874, 0.7905, 0.7905), and
$$BP= 5.099\times SO(G)+57.41,$$
$$BP= 13.54\times SO_{1}(G)+164.1,\quad  BP= 88\times SO_{2}(G)+164.1,$$
$$BP= 1.612\times SO_{3}(G)+56.75,\quad  BP= 1.618\times SO_{4}(G)+92.86,$$
$$BP= 9.293\times SO_{5}(G)+164.1,\quad  BP= 32.06\times SO_{6}(G)+164.1.$$

On the other hand, with the data of Table \ref{table1}, scatter plots between AcenFac (resp. Entropy, SNar, HNar, HVAP, DHVAP) of octane isomers and $SO$ (resp. $SO_{1}$, $SO_{2}$, $SO_{3}$, $SO_{4}$, $SO_{5}$, $SO_{6}$) are shown in Figure \ref{fig-74}-\ref{fig-79}.

The absolute value of correlation coefficient $|R|$ between $SO$ and AcenFac (resp. Entropy, SNar, HNar, HVAP, DHVAP) is about 0.9594 (resp. 0.9465, 0.9842, 0.9618, 0.9031, 0.9469), and
$$AcenFac= -0.01171\times SO(G)+0.6093,\quad  Entropy= -1.473\times SO(G)+139.8,$$
$$SNar= -0.1129\times SO(G)+6.169,\quad  HNar= -0.02947\times SO(G)+2.096.$$
$$HVAP= -0.6303\times SO(G)+83.88,\quad  DHVAP= -0.125\times SO(G)+12.05.$$

The absolute value of correlation coefficient $|R|$ between $SO_{1}$ and AcenFac (resp. Entropy, SNar, HNar, HVAP, DHVAP) is about 0.9192 (resp. 0.8898, 0.9311, 0.9098, 0.9174, 0.9522), and
$$AcenFac= -0.0002884\times SO_{1}(G)+0.4003,\quad  Entropy= -0.3558\times SO_{1}(G)+113.4,$$
$$SNar= -0.02746\times SO_{1}(G)+4.146,\quad  HNar= -0.007165\times SO_{1}(G)+1.568.$$
$$HVAP= -0.1645\times SO_{1}(G)+72.85,\quad  DHVAP= -0.0323\times SO_{1}(G)+9.85.$$

The absolute value of correlation coefficient $|R|$ between $SO_{2}$ and AcenFac (resp. Entropy, SNar, HNar, HVAP, DHVAP) is about 0.9202 (resp. 0.8433, 0.9355, 0.9512, 0.9198, 0.9560), and
$$AcenFac= -0.03143\times SO_{2}(G)+0.4536,\quad  Entropy= -3.671\times SO_{2}(G)+119.2,$$
$$SNar= -0.3003\times SO_{2}(G)+4.658,\quad  HNar= -0.08155\times SO_{2}(G)+1.714.$$
$$HVAP= -1.796\times SO_{2}(G)+75.9,\quad  DHVAP= -0.3531\times SO_{2}(G)+10.45.$$

The absolute value of correlation coefficient $|R|$ between $SO_{3}$ and AcenFac (resp. Entropy, SNar, HNar, HVAP, DHVAP) is about 0.9482 (resp. 0.9399, 0.9791, 0.9551, 0.9112, 0.9509), and
$$AcenFac= -0.0002752\times SO_{3}(G)+0.5523,\quad  Entropy= -0.3477\times SO_{3}(G)+132.8,$$
$$SNar= -0.02671\times SO_{3}(G)+5.633,\quad  HNar= -0.006959\times SO_{3}(G)+1.955.$$
$$HVAP= -0.1512\times SO_{3}(G)+81.06,\quad  DHVAP= -0.02985\times SO_{3}(G)+11.48.$$

The absolute value of correlation coefficient $|R|$ between $SO_{4}$ and AcenFac (resp. Entropy, SNar, HNar, HVAP, DHVAP) is about 0.9466 (resp. 0.9422, 0.9634, 0.9318, 0.8837, 0.93), and
$$AcenFac= -0.001418\times SO_{4}(G)+0.4417,\quad  Entropy= -0.1799\times SO_{4}(G)+118.8,$$
$$SNar= -0.01357\times SO_{4}(G)+4.544,\quad  HNar= -0.003504\times SO_{4}(G)+1.669.$$
$$HVAP= -0.07568\times SO_{4}(G)+74.81,\quad  DHVAP= -0.01507\times SO_{4}(G)+10.25.$$

The absolute value of correlation coefficient $|R|$ between $SO_{5}$ and AcenFac (resp. Entropy, SNar, HNar, HVAP, DHVAP) is about 0.9303 (resp. 0.8890, 0.9501, 0.9402, 0.9393, 0.9726), and
$$AcenFac= -0.002446\times SO_{5}(G)+0.4163,\quad  Entropy= -0.2979\times SO_{5}(G)+115.2,$$
$$SNar= -0.02348\times SO_{5}(G)+4.304,\quad  HNar= -0.006205\times SO_{5}(G)+1.612.$$
$$HVAP= -0.1412\times SO_{5}(G)+73.81,\quad  DHVAP= -0.02765\times SO_{5}(G)+10.04.$$

The absolute value of correlation coefficient $|R|$ between $SO_{6}$ and AcenFac (resp. Entropy, SNar, HNar, HVAP, DHVAP) is about 0.9029 (resp. 0.9043, 0.9295, 0.8959, 0.9053, 0.9379), and
$$AcenFac= -0.002724\times SO_{6}(G)+0.3845,\quad  Entropy= -0.3477\times SO_{6}(G)+111.6,$$
$$SNar= -0.02635\times SO_{6}(G)+4.003,\quad  HNar= -0.006783\times SO_{6}(G)+1.529.$$
$$HVAP= -0.1561\times SO_{6}(G)+71.96,\quad  DHVAP= -0.03059\times SO_{6}(G)+9.674.$$

\begin{figure}[ht!]
  \centering
  \scalebox{.15}[.15]{\includegraphics{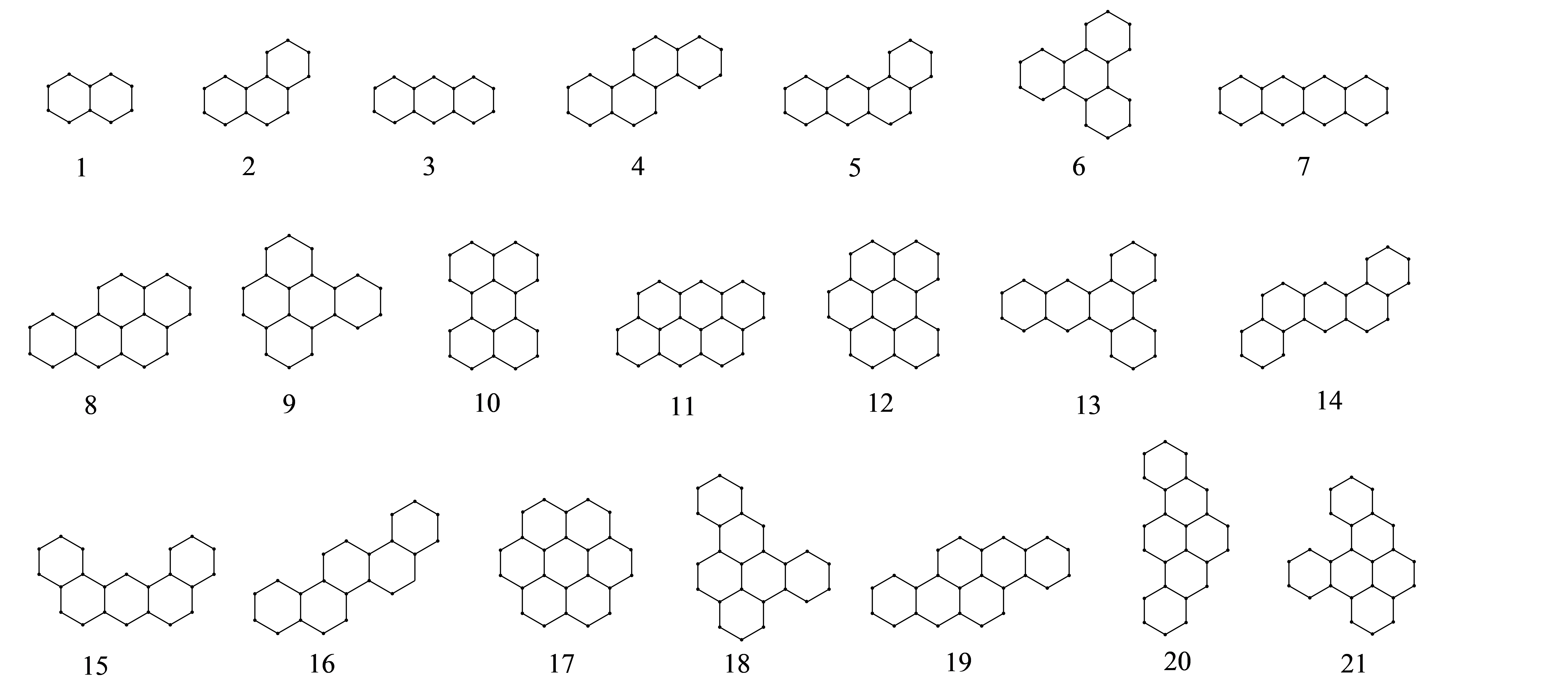}}
  \caption{21 benzenoid hydrocarbons.}
 \label{fig-71}
\end{figure}

\begin{figure}[ht!]
  \centering
  \scalebox{.16}[.16]{\includegraphics{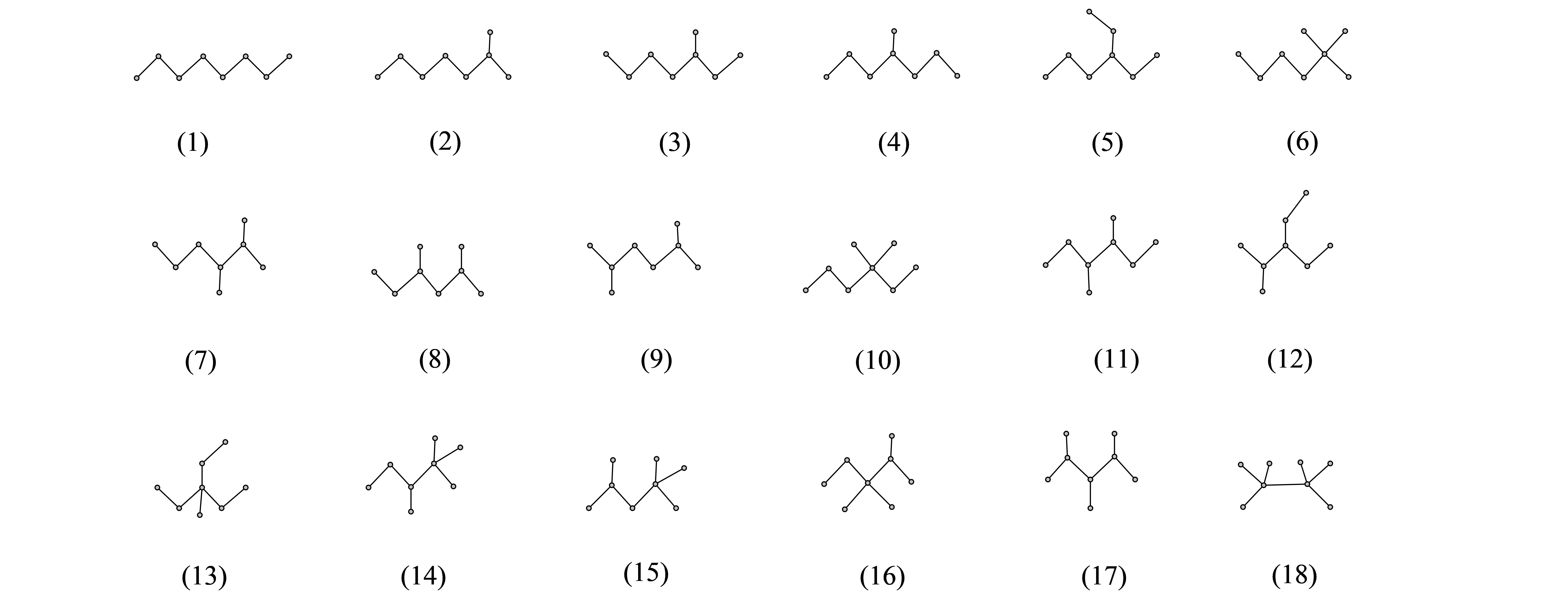}}
  \caption{18 octane isomers.}
 \label{fig-72}
\end{figure}

\begin{sidewaystable}[h]
	\centering
    \caption{Experimental values of AcenFac (resp. Entropy, SNar, HNar, HVAP, DHVAP), and Sombor-index-like invariants of $18$ octane isomers in Figure \ref{fig-72}}
    \vskip 0.2cm
    \setlength{\tabcolsep}{2mm}{
	\begin{tabular}{c|ccccccccccccc}\hline
	  No.     &  $AcenFac$  & $Entropy$  &	$SNar$  &	$HNar$  &	$HVAP$  &	$DHVAP$  &	$SO(G)$  &	$SO_{1}(G)$  &	$SO_{2}(G)$  &	$SO_{3}(G)$  &	$SO_{4}(G)$  &	$SO_{5}(G)$  &	$SO_{6}(G)$\\  \hline

	$1$    &  $0.397898$  &    $111.67$  &    $4.159$  &  $1.6$  &    $73.19$  &  $9.915$  &    $18.6143$  &     $3$ &    $1.2$  &    $59.2384$  &    $40.1426$ &    $6.4045$  &    $1.63204$ \\  \hline

    $2$    &  $0.377916$  &    $109.84$  &	  $3.871$  &  $1.5$  & $70.3$  &  $9.484$  &    $20.6515$  &    $12$ &    $2.5846$  &    $67.8280$  &    $53.4664$ &    $19.8351$  &    $8.58625$   \\ \hline

    $3$    &  $0.371002$  &    $111.26$  &	  $3.871$  &  $1.5$  &	  $71.3$  &  $9.521$ &    $20.5024$  &    $12$ &    $2.7692$  &    $66.7913$  &    $52.3477$ &    $20.1844$  &    $7.1007$  \\ \hline

    $4$    &  $0.371504$  &    $109.32$  &	  $3.871$  &  $1.5$  &	  $70.91$  &  $9.483$ &    $20.5024$  &    $12$ &    $2.7692$  &    $66.7913$  &    $52.3477$ &    $20.1844$  &    $7.1007$  \\ \hline

    $5$    &  $0.362472$  &    $109.43$  &	  $3.871$  &  $1.5$  &	  $71.7$  &  $9.476$ &    $20.3533$  &    $12$ &    $2.9538$  &    $65.7547$  &    $51.2289$ &    $20.5336$  &    $5.61515$  \\ \hline

    $6$    &  $0.339426$  &    $103.42$  &	  $3.466$  &  $1.391$  &	$67.7$  &  $8.915$ &   $24.7344$ &    $30$ &    $3.8471$  &    $85.3034$  &    $88.8582$ &    $39.7493$  &    $27.755$  \\ \hline

    $7$    &  $0.348247$  &    $108.02$  &	  $3.584$  &  $1.412$  &	$70.2$  &  $9.272$ &   $22.3995$ &    $16$ &    $3.3846$  &    $74.4923$  &    $64.8547$ &    $26.3304$  &    $11.9435$  \\ \hline

    $8$    &  $0.344223$  &    $106.98$  &	  $3.584$  &  $1.412$  &	$68.5$  &  $9.029$ &   $22.5396$ &    $21$ &    $4.1538$  &    $75.3809$  &    $65.6715$ &    $33.615$  &    $14.0549$  \\ \hline

    $9$    &  $0.356830$  &    $105.72$  &	  $3.584$  &  $1.412$  &	$68.6$  &  $9.051$ &   $22.6886$ &    $21$ &    $3.9692$  &    $76.4176$  &    $66.7903$ &    $33.2657$  &    $15.5405$  \\ \hline

    $10$   &  $0.322596$  &    $104.74$  &	  $3.466$  &  $1.391$  &	$68.5$  &  $8.973$ &   $24.4910$ &    $30$ &    $4.1647$  &    $83.5262$  &    $86.2332$ &    $40.4744$  &    $25.2129$  \\ \hline

    $11$   &  $0.340345$  &    $106.59$  &	  $3.584$  &  $1.412$  &	$70.2$  &  $9.316$ &   $22.2504$ &    $16$ &    $3.5692$  &    $73.4557$  &    $63.7359$ &    $26.6796$  &    $10.458$  \\ \hline

    $12$   &  $0.332433$  &    $106.06$  &	  $3.584$  &  $1.412$  &	$69.7$  &  $9.209$ &   $22.2504$ &    $16$ &    $3.5692$  &    $73.4557$  &    $63.7359$ &    $26.6796$  &    $10.458$  \\ \hline

    $13$   &  $0.306899$  &    $101.48$  &	  $3.466$  &  $1.391$  &	$69.3$  &  $9.081$ &   $24.2477$ &    $30$ &    $4.4823$  &    $81.7490$  &    $83.6083$ &    $41.1995$  &    $22.6709$  \\ \hline

    $14$   &  $0.300816$  &    $101.31$  &	  $3.178$  &  $1.315$  &	$67.3$  &  $8.826$ &   $26.3732$ &    $34$ &    $4.7116$  &    $91.2484$  &    $99.3103$ &    $46.4613$  &    $29.1333$  \\ \hline

    $15$   &  $0.305370$  &    $104.09$  &	  $3.178$  &  $1.315$  &	$64.87$  &  $8.402$ &   $26.7716$ &    $39$ &    $5.2316$  &    $93.8929$  &    $102.182$ &    $53.18$  &    $34.7092$  \\ \hline

    $16$   &  $0.293177$  &    $102.06$  &	  $3.178$  &  $1.315$  &	$68.1$  &  $8.897$ &   $26.2790$ &    $34$ &    $4.8447$  &    $90.5079$  &    $97.804$ &    $46.8371$  &    $28.0768$  \\ \hline

    $17$   &  $0.317422$  &    $102.39$  &	  $3.296$  &  $1.333$  &	$68.37$  &  $9.014$ &   $24.2967$ &    $20$ &    $4$  &    $82.1933$  &    $77.3617$ &    $32.4764$  &    $16.7863$  \\ \hline

    $18$   &  $0.255294$  &    $93.06$  &	  $2.773$  &  $1.231$  &	$66.2$  &  $8.41$ &   $30.3955$ &    $45$ &    $5.2941$  &    $108.406$   &    $134.083$ &    $58.5364$  &    $45.4455$   \\ \hline

	\end{tabular}}
	
	\label{table1}
\end{sidewaystable}

\begin{sidewaystable}[h]
	\centering
    \caption{Experimental values of boiling point and Sombor-index-like invariants of $21$ benzenoid hydrocarbons in Figure \ref{fig-71}}
    \setlength{\tabcolsep}{3mm}{
	\begin{tabular}{c|cccccccc}\hline
	  No.     &  $BP(^{o}C)$  & $SO(G)$  &	$SO_{1}(G)$  &	$SO_{2}(G)$ &	$SO_{3}(G)$  &	$SO_{4}(G)$  &	$SO_{5}(G)$ &	$SO_{6}(G)$   \\  \hline

	$1$    &  $218$  &    $35.6354$  &    $10$  &    $1.53846$  &    $112.849$  &    $94.311$  &    $14.5692$  &    $4.22279$\\  \hline

    $2$    &  $338$  &    $54.1602$  &	  $15$  &    $2.30769$  &    $171.495$  &    $150.105$  &    $21.8538$  &    $6.33419$\\ \hline

    $3$    &  $340$  &    $54.3003$  &	  $20$  &    $3.07692$  &    $172.384$  &    $150.922$  &    $29.1383$  &    $8.44559$\\ \hline

    $4$    &  $431$  &    $72.6850$  &	  $20$  &    $3.07692$  &    $230.141$  &    $205.900$  &    $29.1383$  &    $8.44559$\\ \hline

    $5$    &  $425$  &    $72.8251$  &	  $25$  &    $3.84615$  &    $231.03$  &     $206.717$  &    $36.4229$  &    $10.557$\\ \hline

    $6$    &  $429$  &    $72.5450$  &	  $15$  &    $2.30769$  &    $229.253$  &    $205.083$  &    $21.8538$  &    $6.33419$\\ \hline

    $7$    &  $440$  &    $72.9651$  &	  $30$  &    $4.61538$  &    $231.918$  &    $207.534$  &    $43.7075$  &    $12.6684$\\ \hline

    $8$    &  $496$  &    $85.5530$  &	  $25$  &    $3.84615$  &    $271.016$  &    $249.128$  &    $36.4229$  &    $10.557$\\ \hline

    $9$    &  $493$  &    $85.4130$  &	  $20$  &    $3.07692$  &    $270.127$  &    $248.311$  &    $29.1383$  &    $8.44559$\\ \hline

    $10$   &  $497$  &    $85.4130$  &	  $20$  &    $3.07692$  &    $270.127$  &    $248.311$  &    $29.1383$  &    $8.44559$\\ \hline

    $11$   &  $547$  &    $98.4209$  &	  $30$  &    $4.61538$  &    $311.89$  &     $292.357$  &    $43.7075$  &    $12.6684$\\ \hline

    $12$   &  $542$  &    $98.2809$  &	  $25$  &    $3.84615$  &    $311.002$  &    $291.540$  &    $36.4229$  &    $10.557$\\ \hline

    $13$   &  $535$  &    $91.2098$  &	  $25$  &    $3.84615$  &    $288.787$  &    $261.695$  &    $36.4229$  &    $10.557$\\ \hline

    $14$   &  $536$  &    $91.3499$  &	  $30$  &    $4.61538$  &    $289.676$  &    $262.511$  &    $43.7075$  &    $12.6684$\\ \hline

    $15$   &  $531$  &    $91.3499$  &	  $30$  &    $4.61538$  &    $289.676$  &    $262.511$  &    $43.7075$  &    $12.6684$\\ \hline

    $16$   &  $519$  &    $91.2098$  &	  $25$  &    $3.84615$  &    $288.787$  &    $261.695$  &    $36.4229$  &    $10.557$\\ \hline

    $17$   &  $590$  &    $111.149$  &	  $30$  &    $4.61538$  &    $351.876$  &    $334.768$  &    $43.7075$  &    $12.6684$\\ \hline

    $18$   &  $592$  &    $103.938$  &	  $25$  &    $3.84615$  &    $328.773$  &    $304.106$  &    $36.4229$  &    $10.557$\\ \hline

    $19$   &  $596$  &    $104.078$  &	  $30$  &    $4.61538$  &    $329.662$  &    $304.923$  &    $43.7075$  &    $12.6684$\\ \hline

    $20$   &  $594$  &    $104.078$  &	  $30$  &    $4.61538$  &    $329.662$  &    $304.923$  &    $43.7075$  &    $12.6684$\\ \hline

    $21$   &  $595$  &    $103.938$  &	  $25$  &    $3.84615$  &    $328.773$  &    $304.106$  &    $36.4229$  &    $10.557$\\ \hline

	\end{tabular}}
	
	\label{table2}
\end{sidewaystable}

\begin{figure}[ht!]
  \centering
  \scalebox{.08}[.08]{\includegraphics{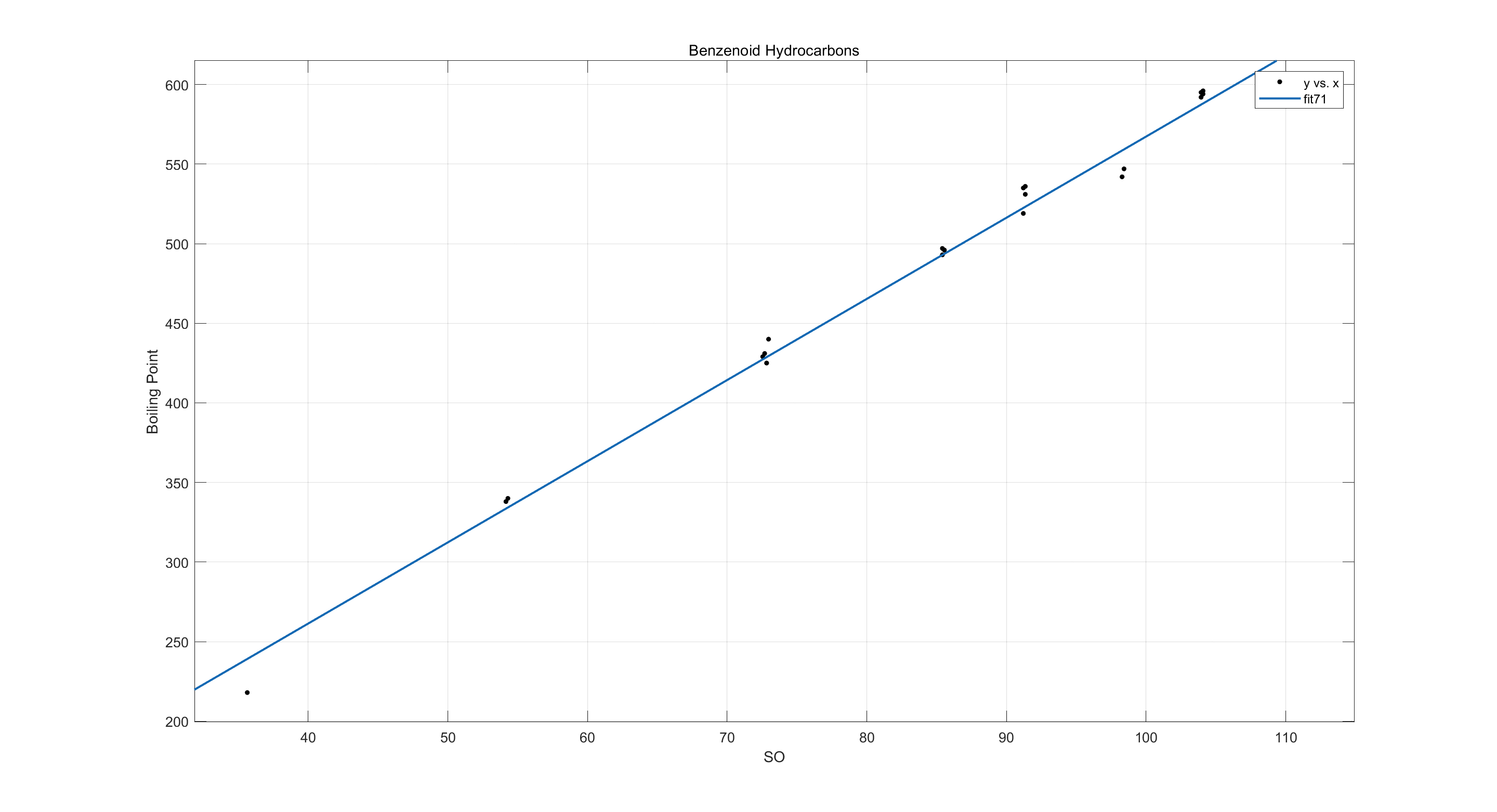}}
  \scalebox{.08}[.08]{\includegraphics{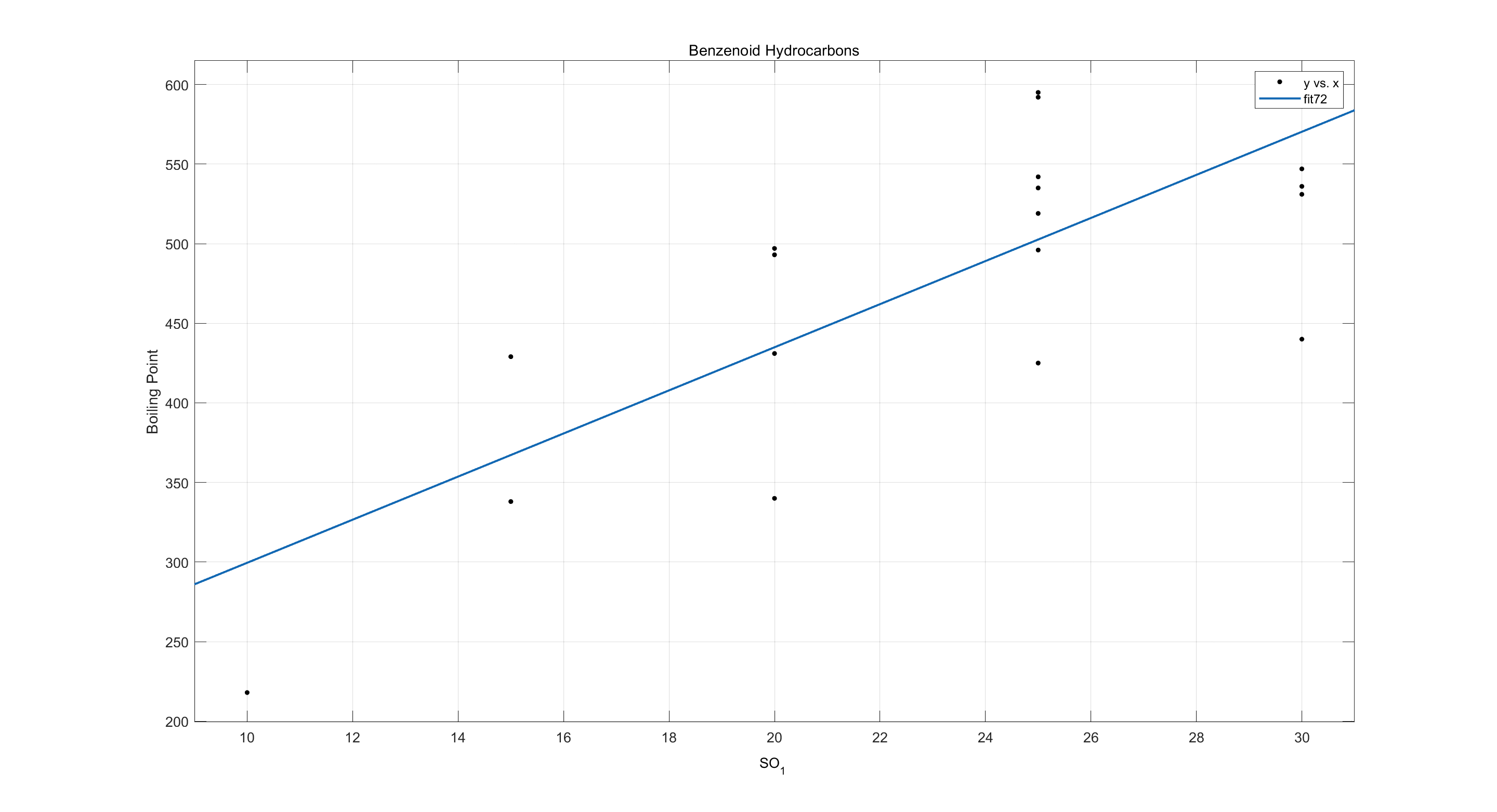}}
  \scalebox{.08}[.08]{\includegraphics{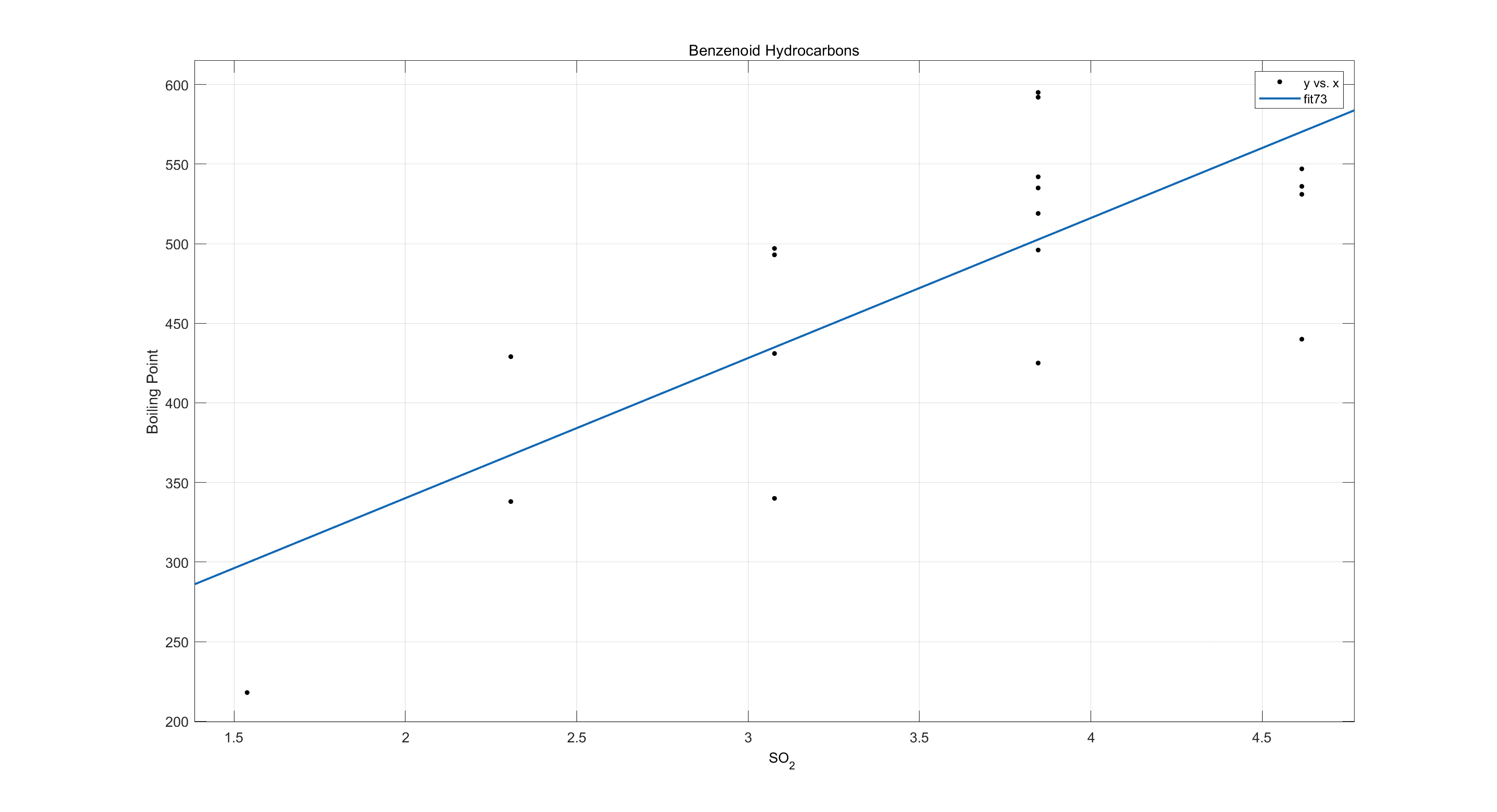}}
  \scalebox{.08}[.08]{\includegraphics{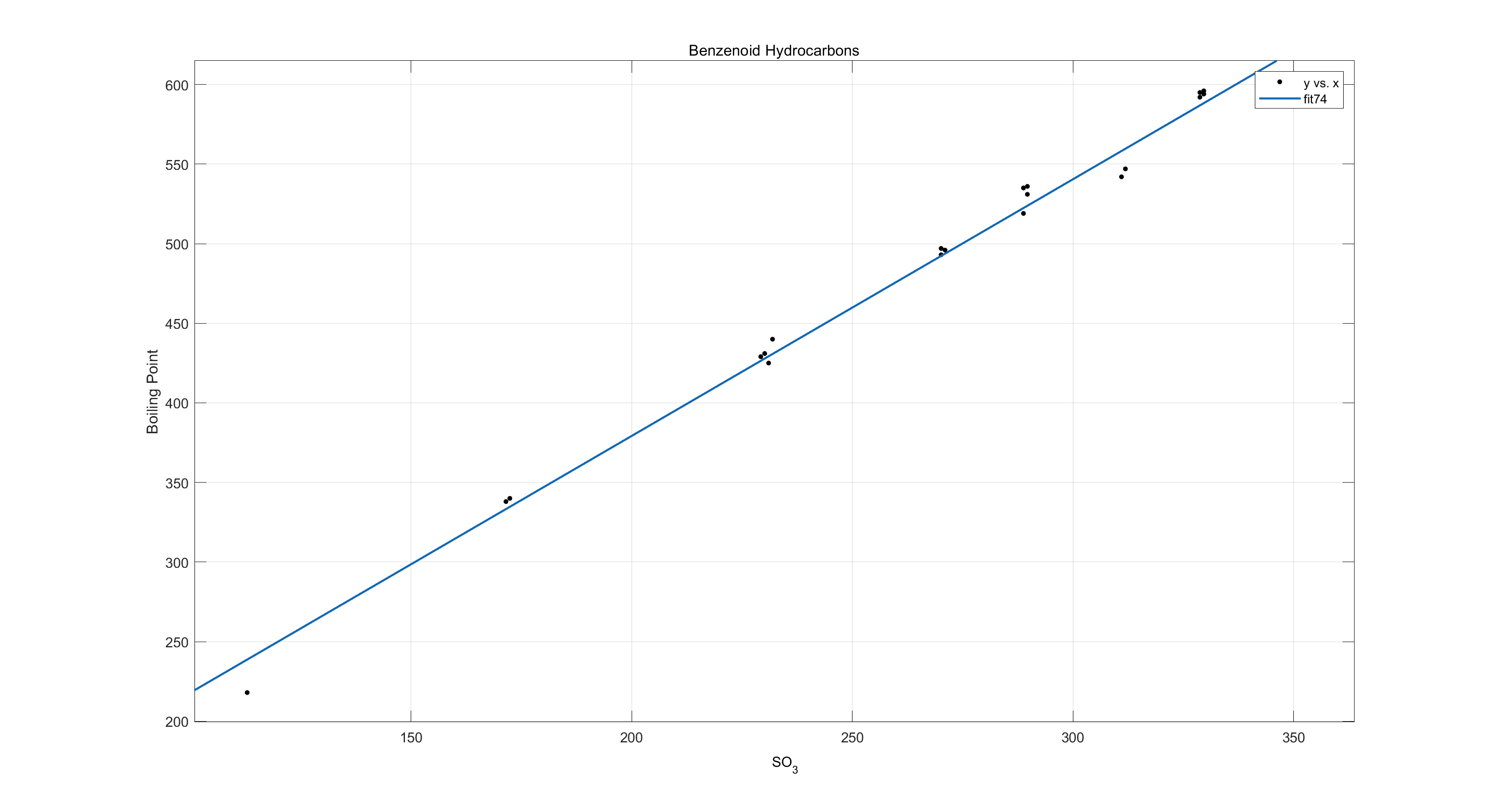}}
  \scalebox{.08}[.08]{\includegraphics{fit74.png}}
  \scalebox{.08}[.08]{\includegraphics{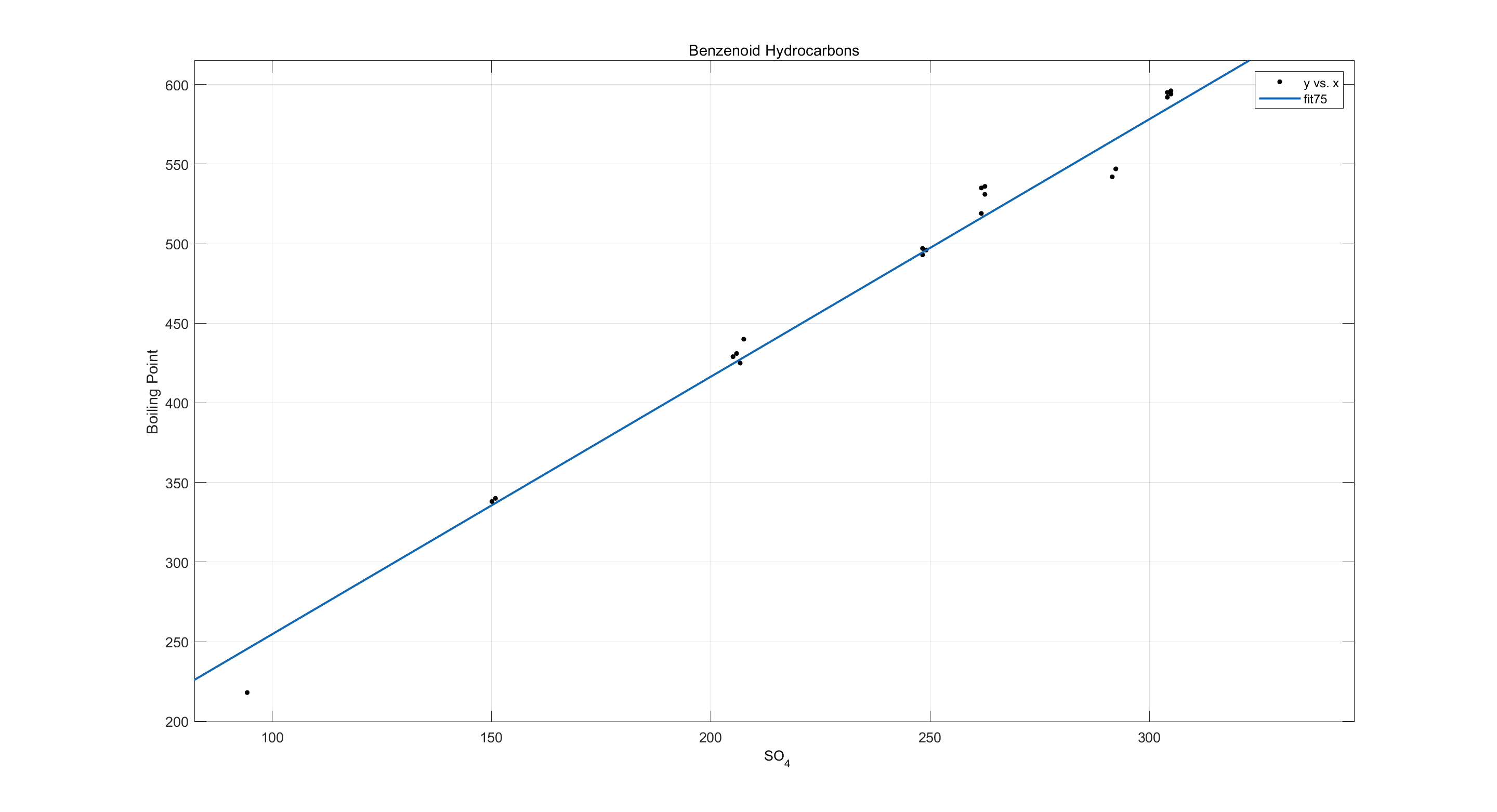}}
  \scalebox{.08}[.08]{\includegraphics{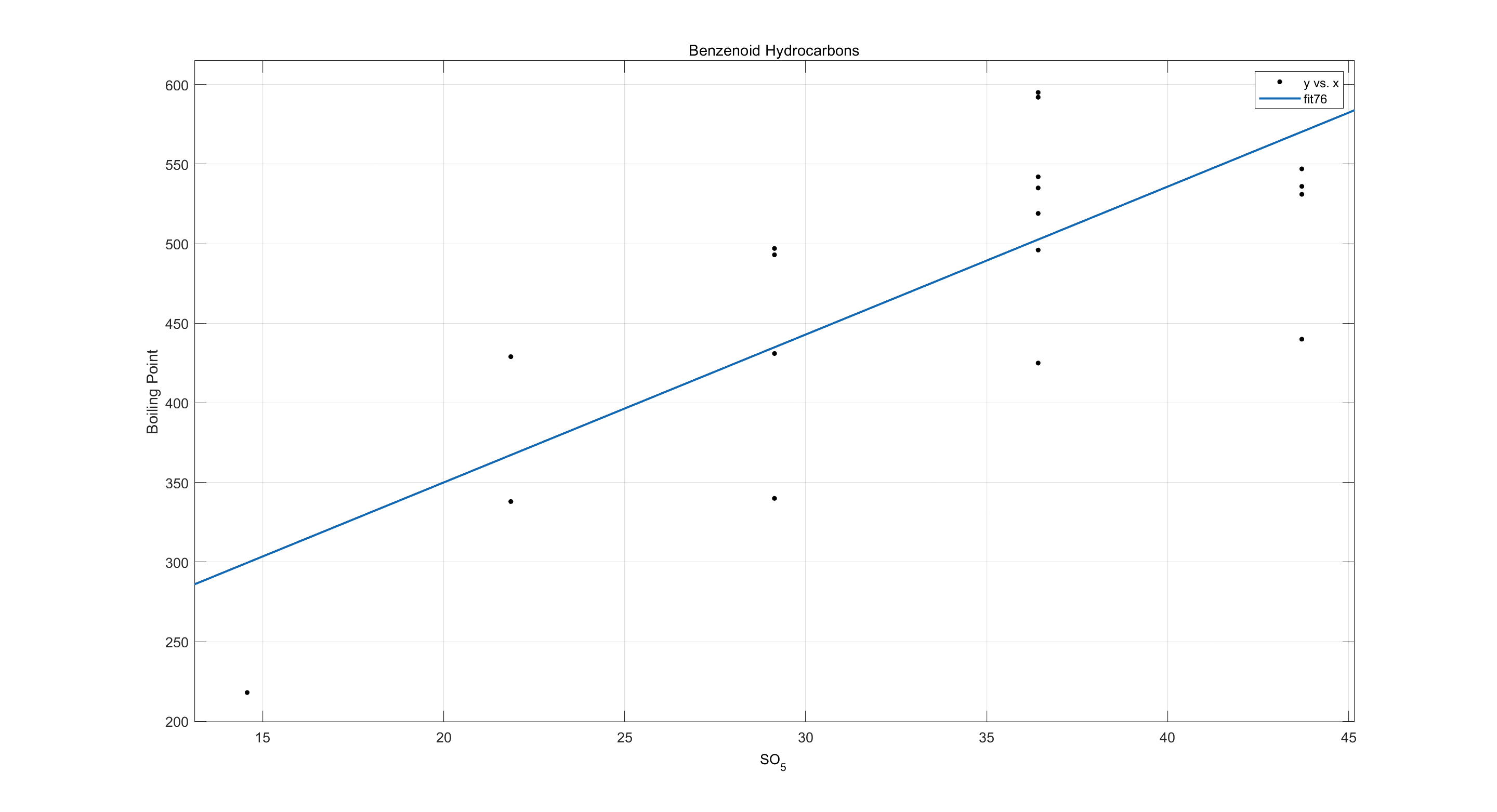}}
  \caption{Scatter plot between BP of benzenoid hydrocarbons and $SO$ (resp. $SO_{i}$ where $1\leq i\leq 6$).}
 \label{fig-73}
\end{figure}

\begin{figure}[ht!]
  \centering
  \scalebox{.08}[.08]{\includegraphics{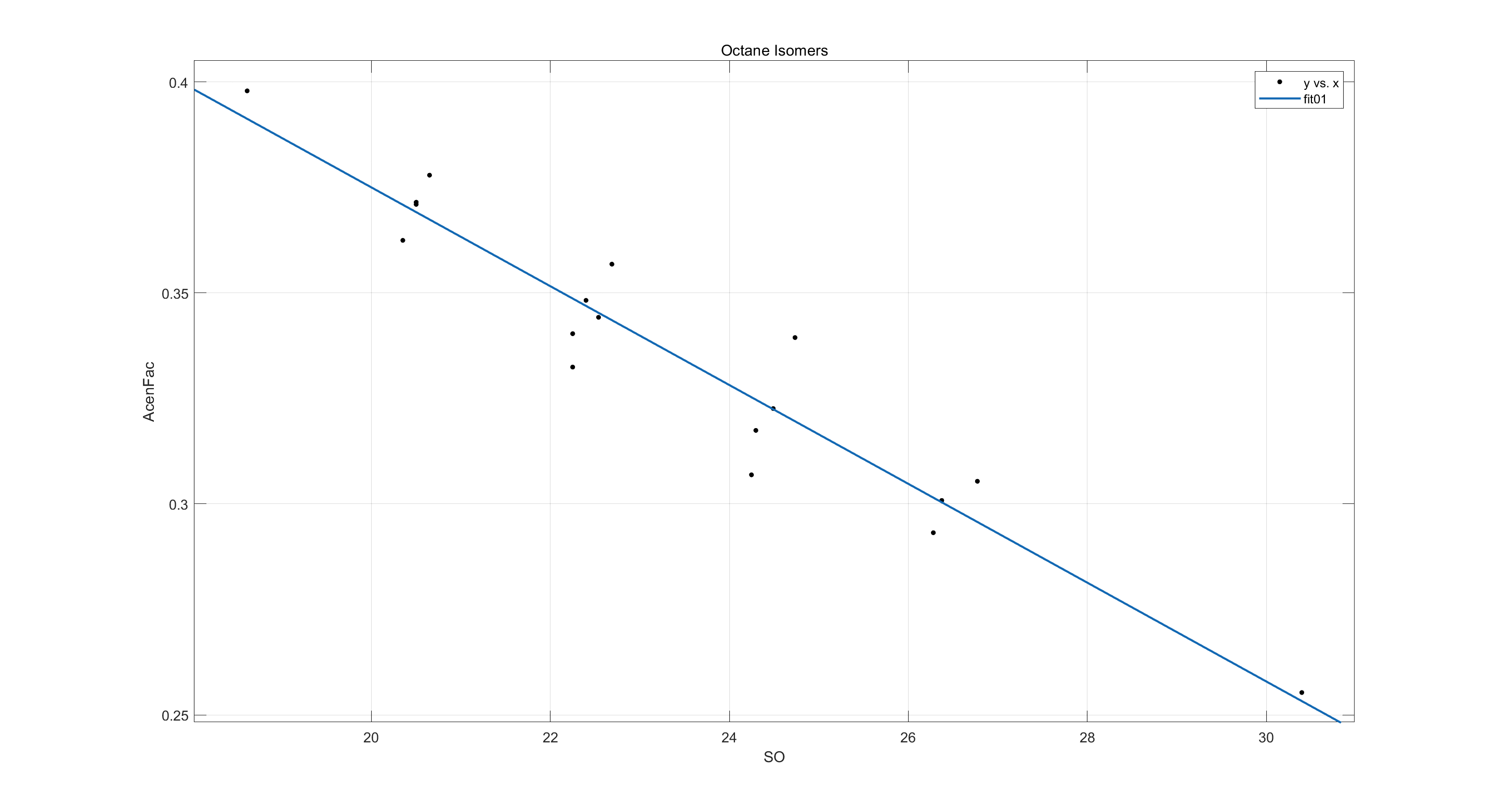}}
  \scalebox{.08}[.08]{\includegraphics{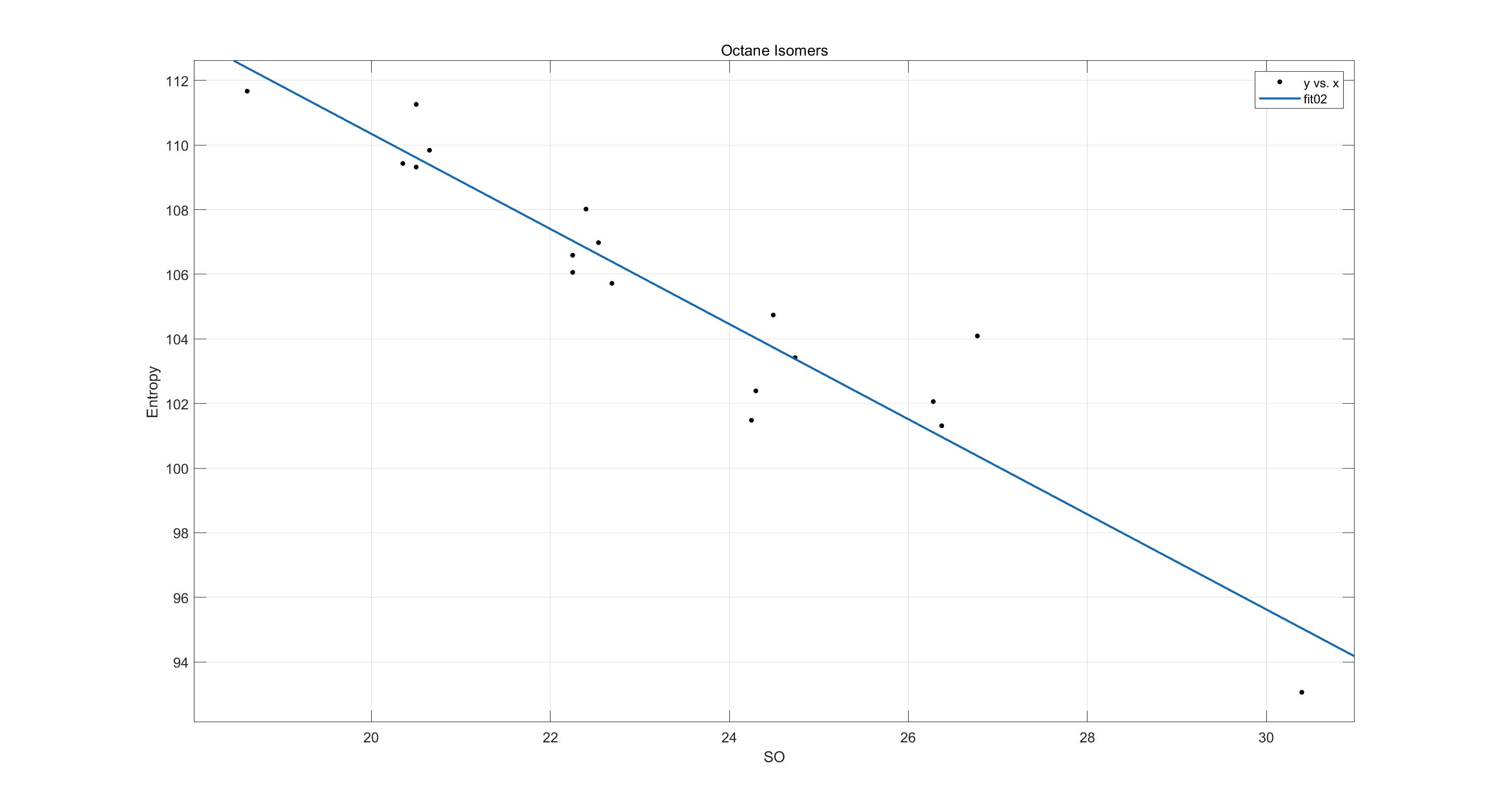}}
  \scalebox{.08}[.08]{\includegraphics{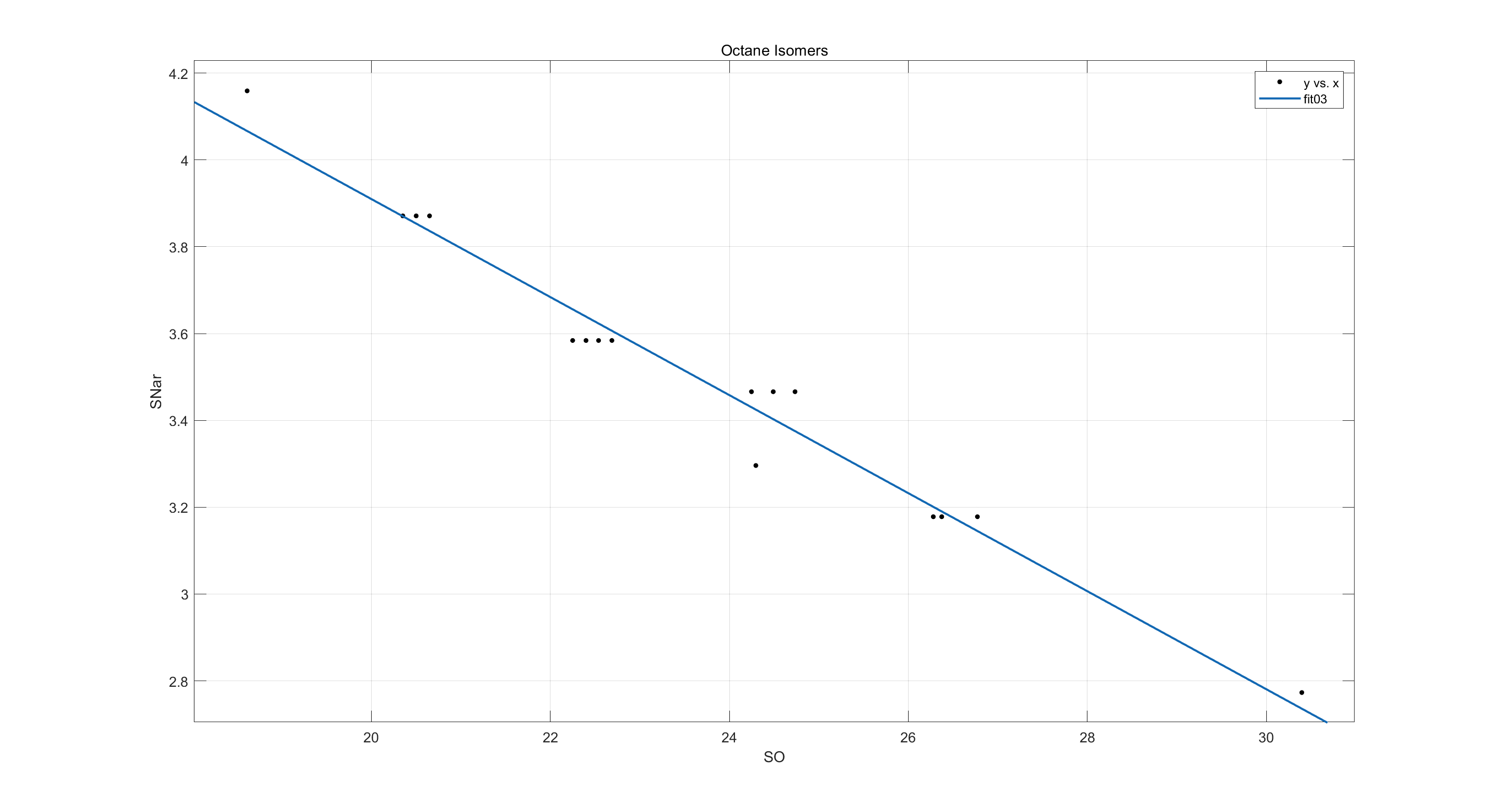}}
  \scalebox{.08}[.08]{\includegraphics{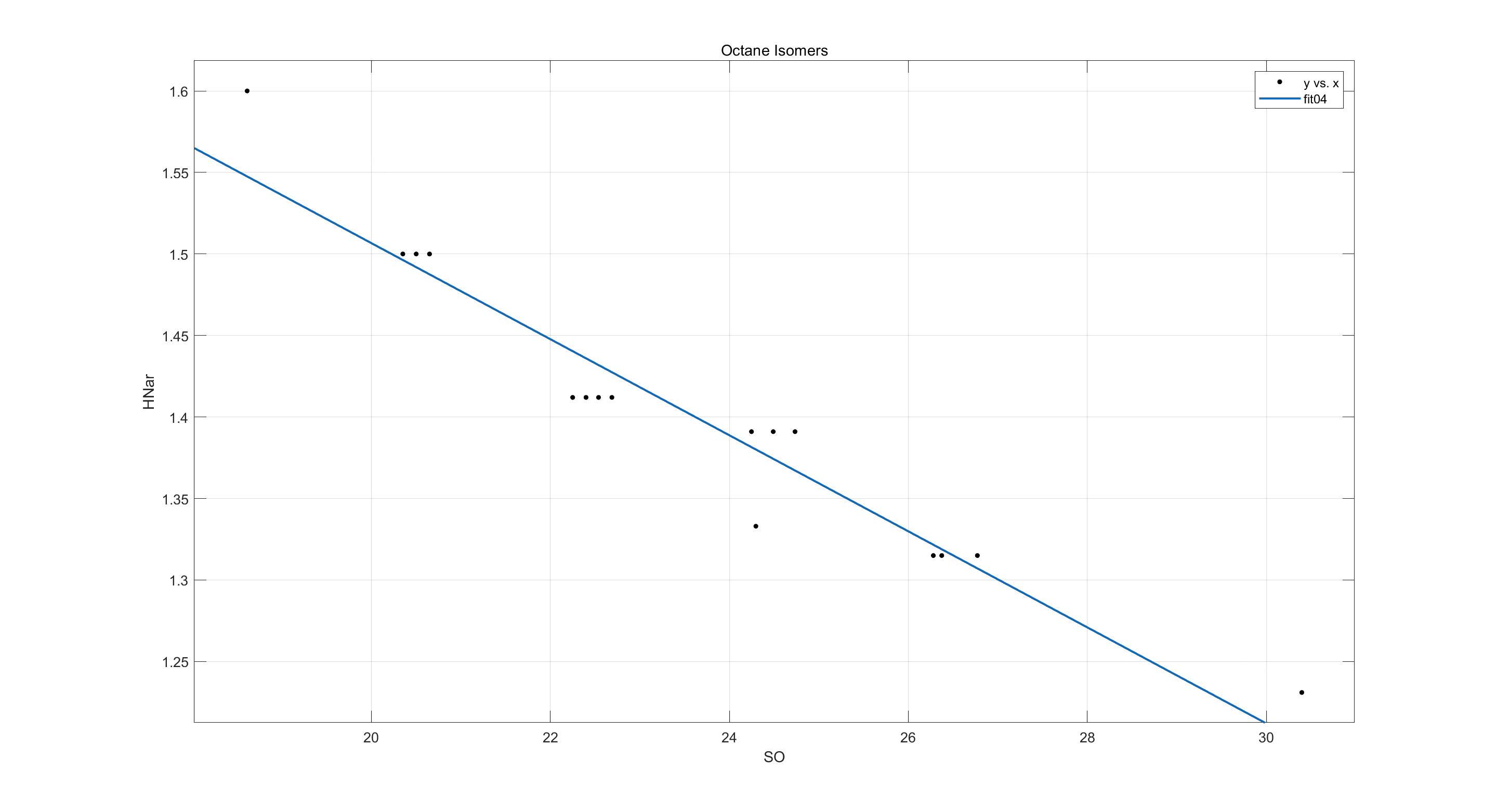}}
  \scalebox{.08}[.08]{\includegraphics{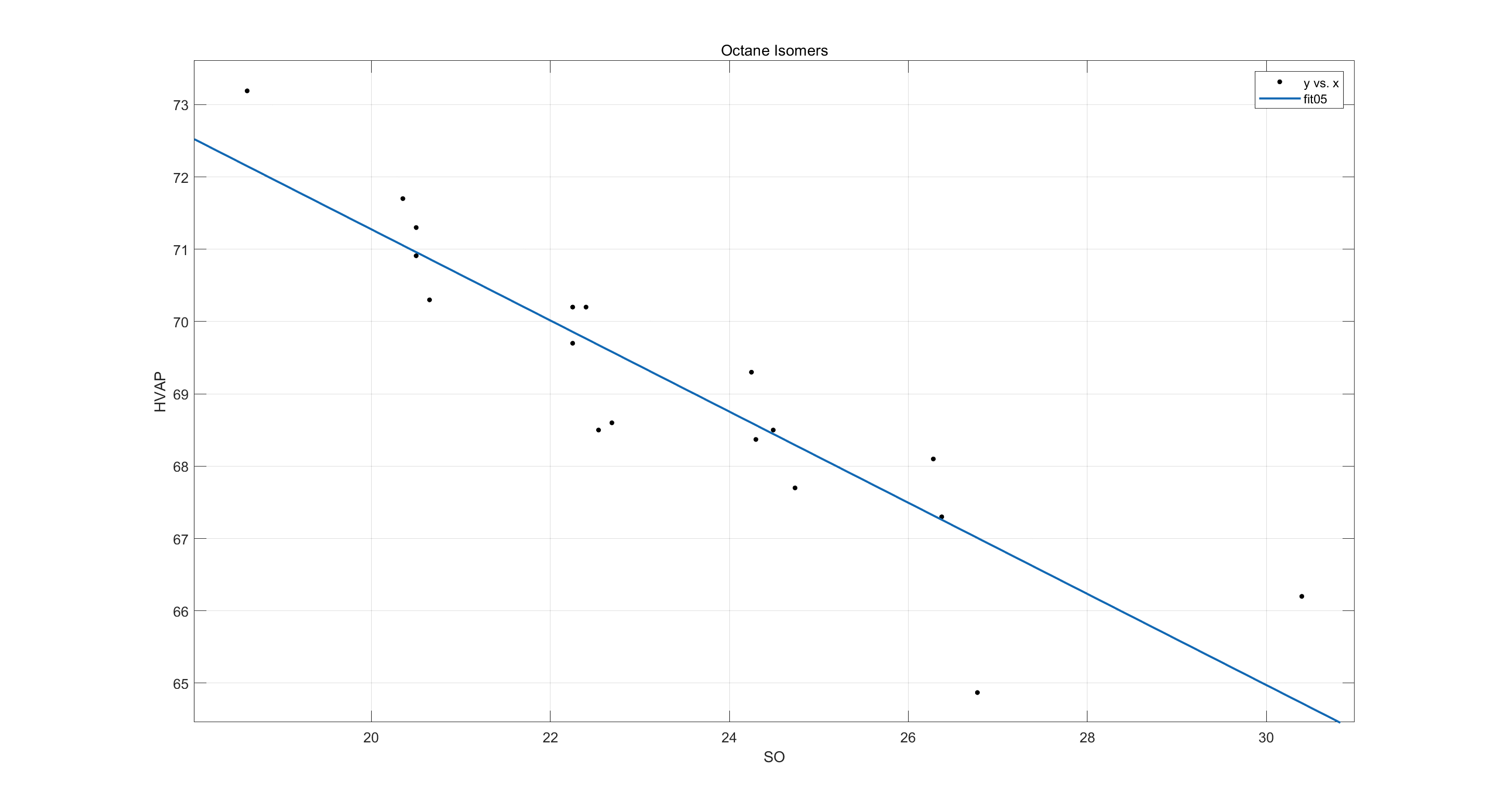}}
  \scalebox{.08}[.08]{\includegraphics{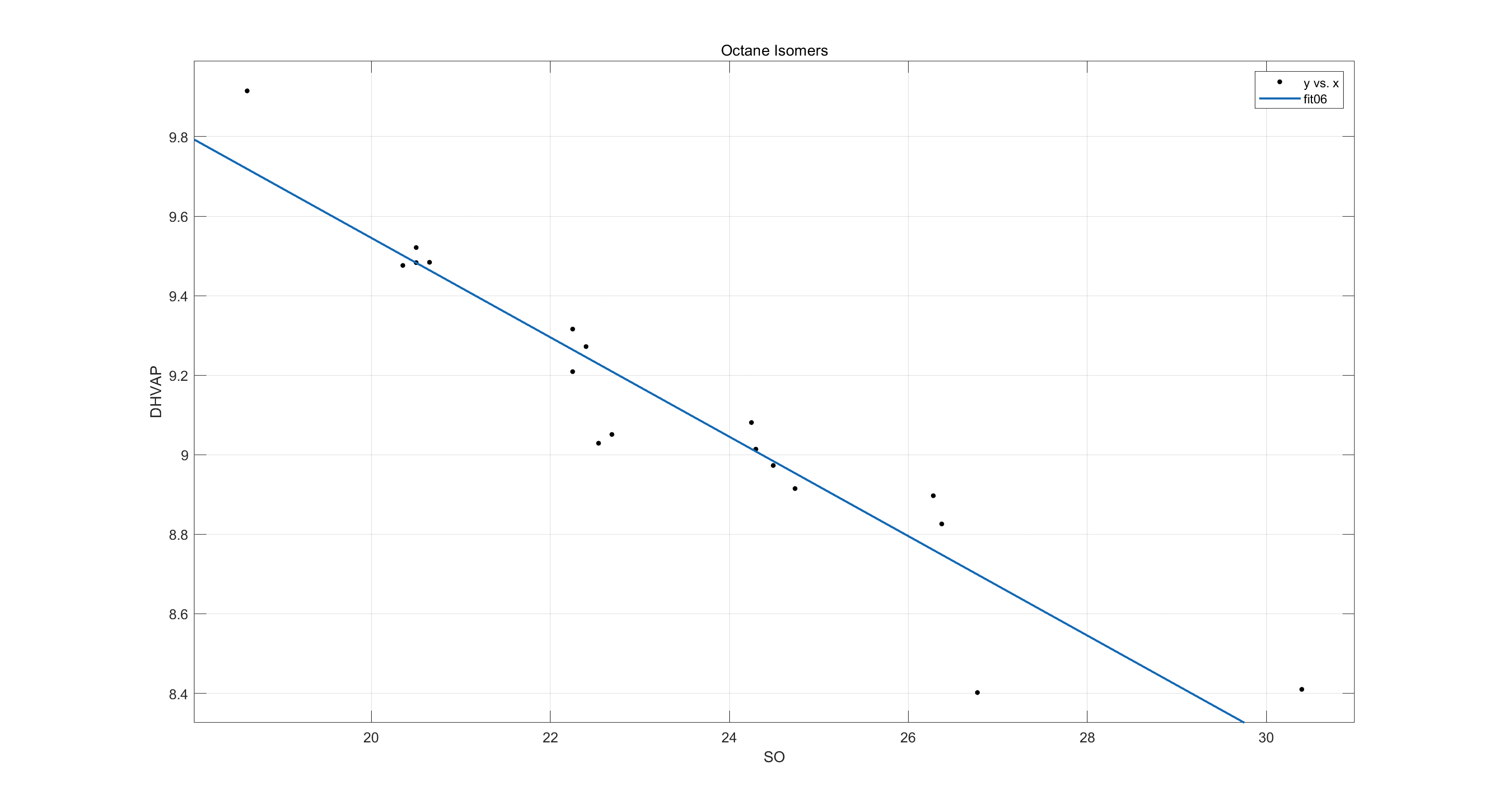}}
  \caption{Scatter plot between AcenFac (resp. Entropy, SNar, HNar, HVAP, DHVAP) of octane isomers and $SO(G)$.}
 \label{fig-74}
\end{figure}

\begin{figure}[ht!]
  \centering
  \scalebox{.08}[.08]{\includegraphics{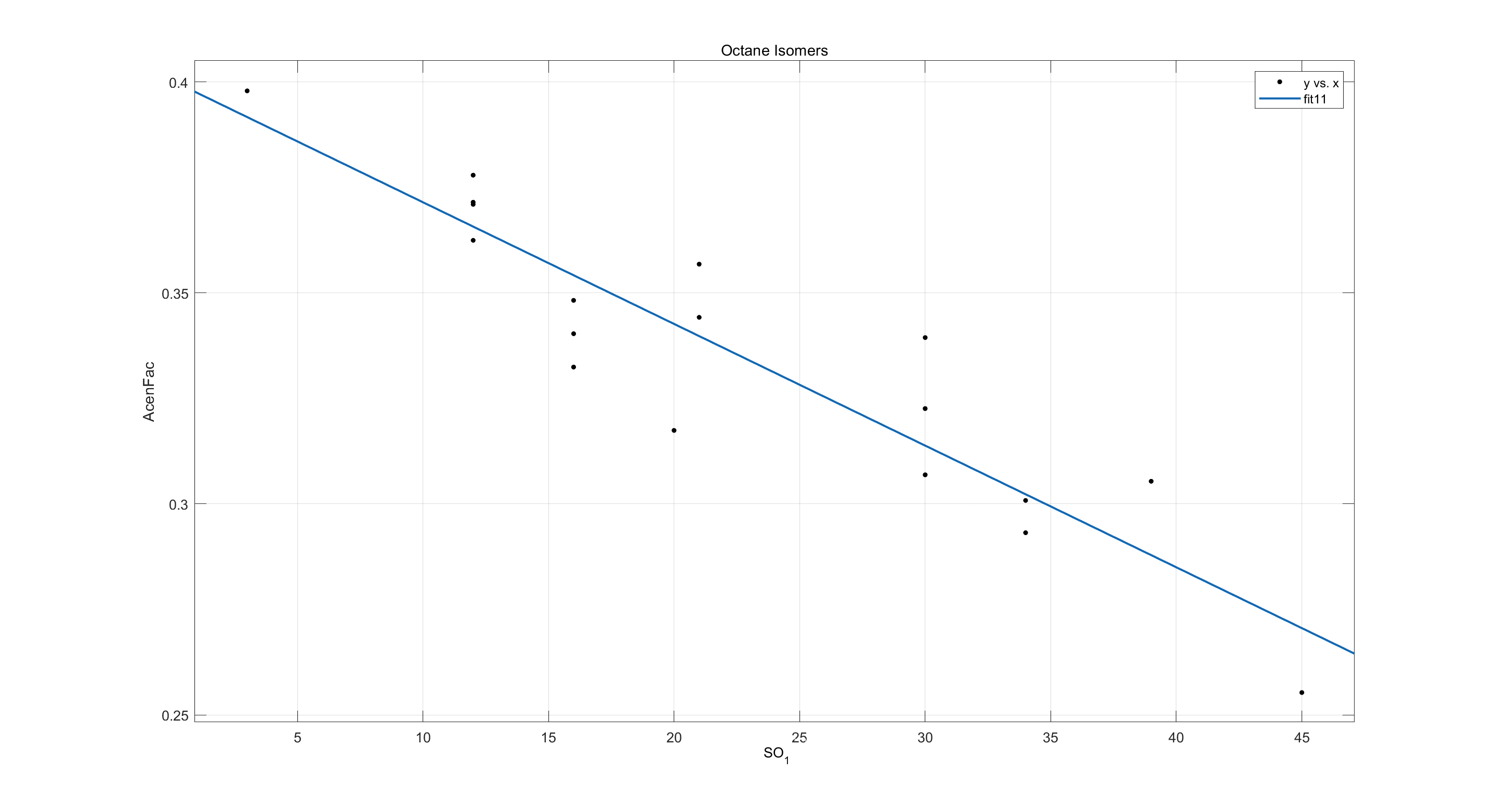}}
  \scalebox{.08}[.08]{\includegraphics{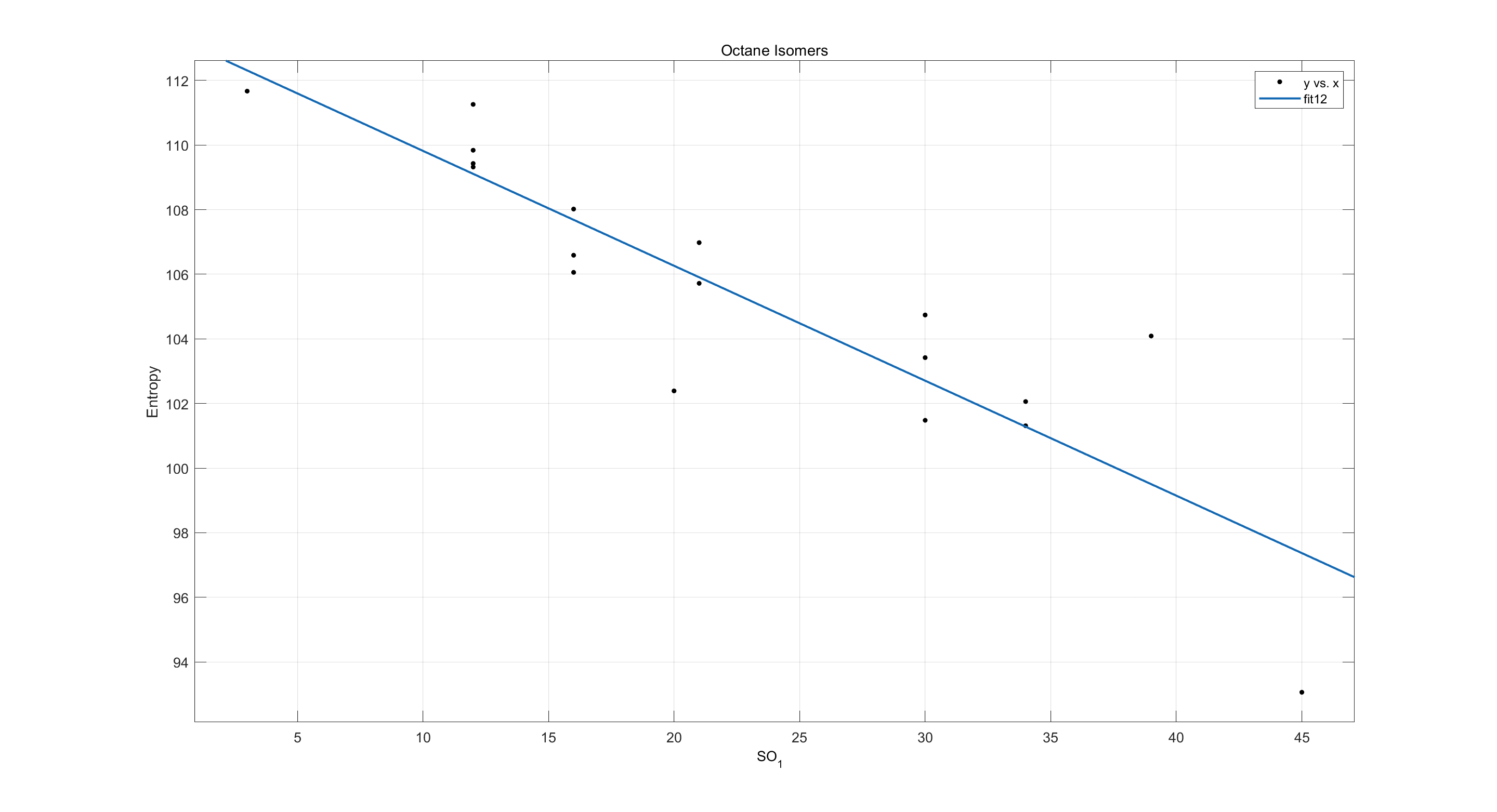}}
  \scalebox{.08}[.08]{\includegraphics{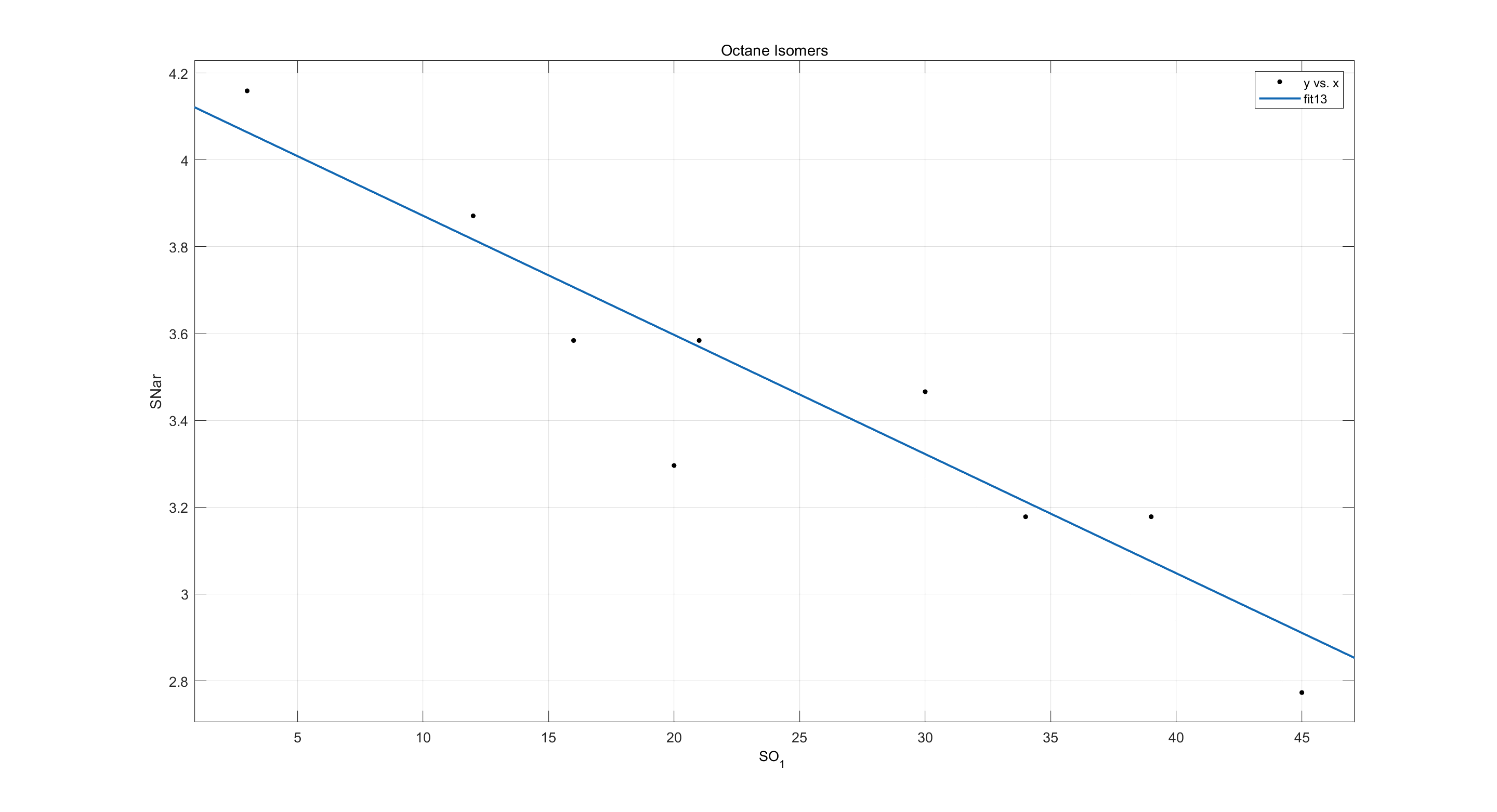}}
  \scalebox{.08}[.08]{\includegraphics{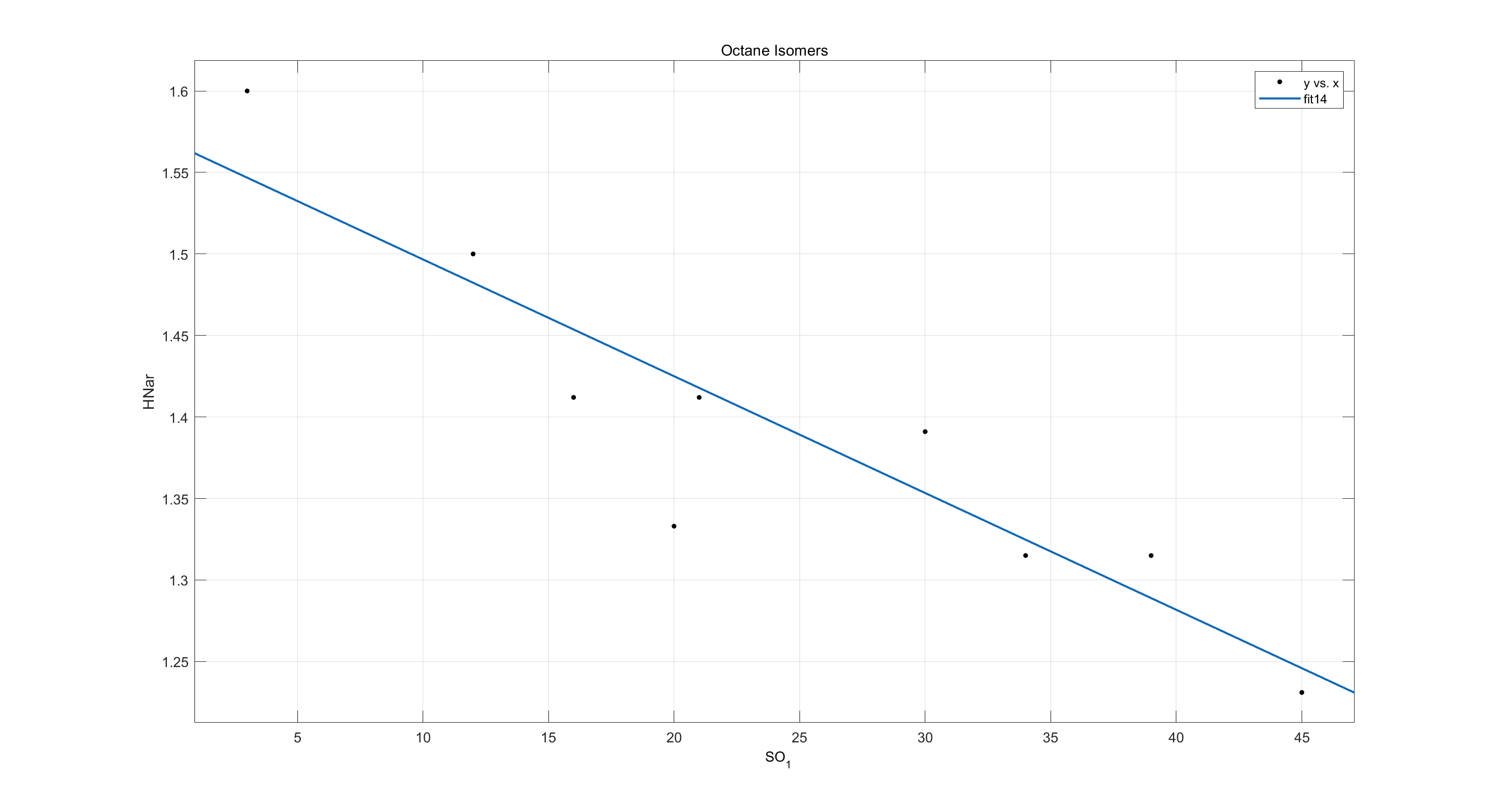}}
  \scalebox{.08}[.08]{\includegraphics{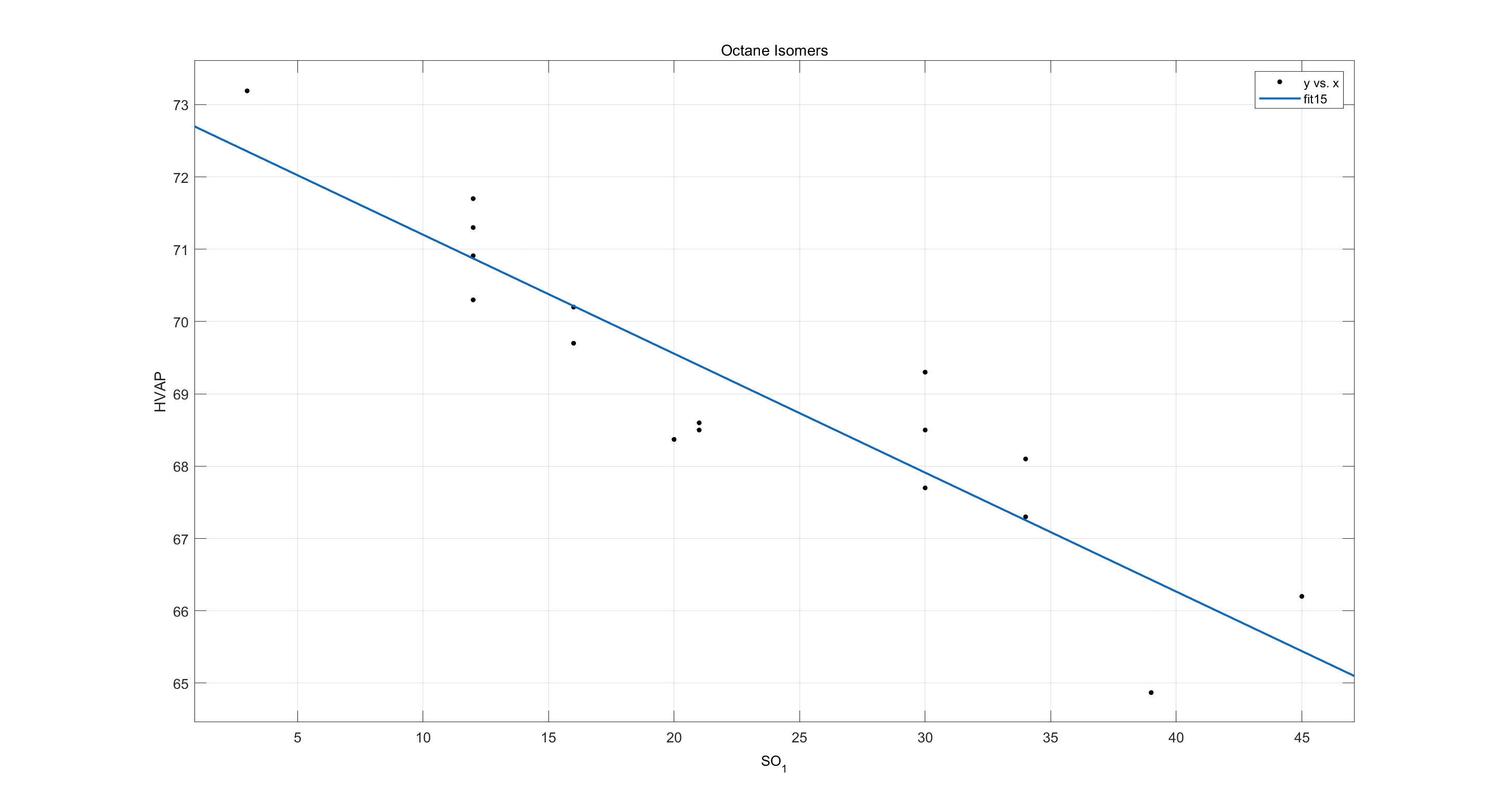}}
  \scalebox{.08}[.08]{\includegraphics{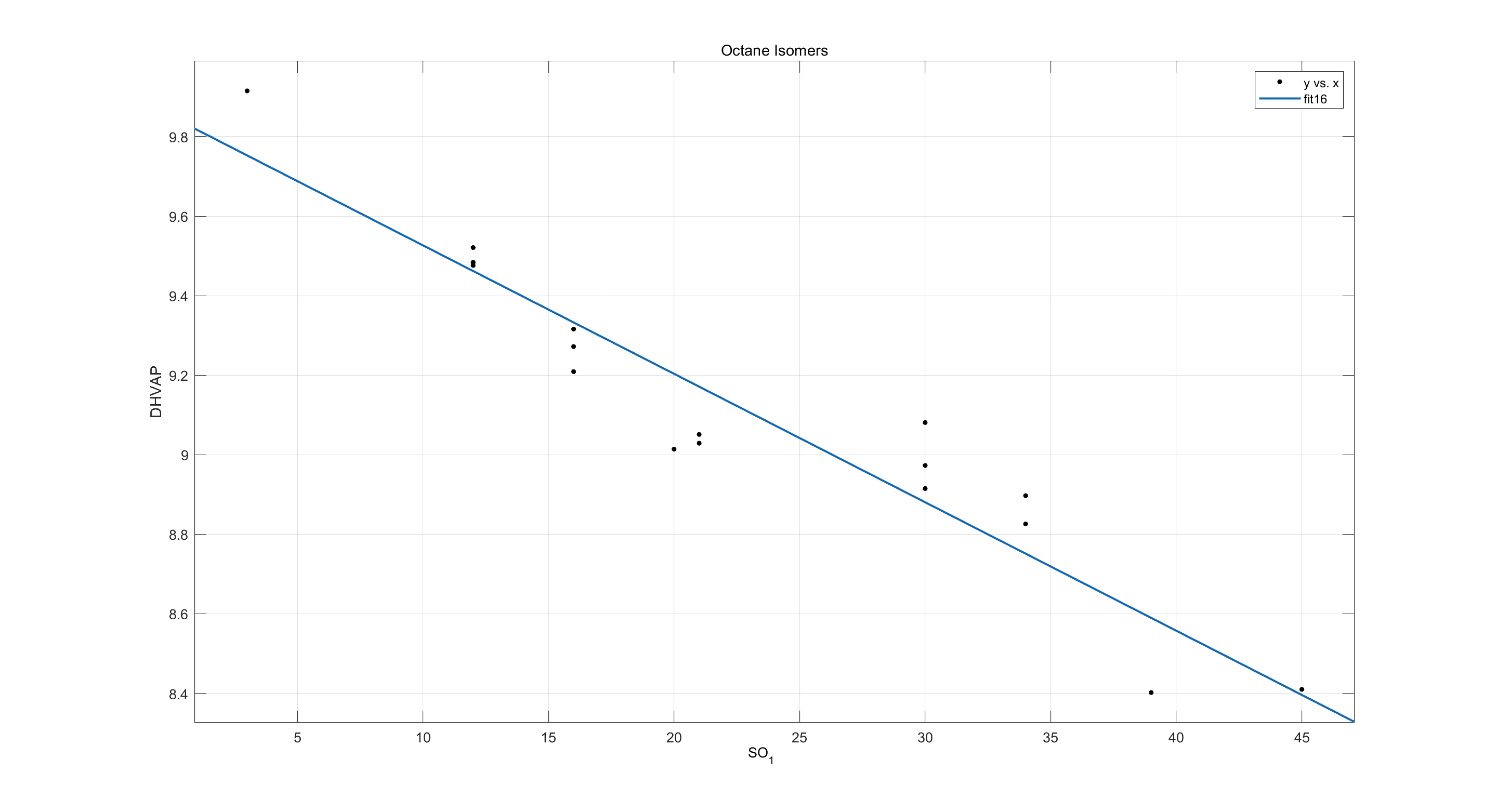}}
  \caption{Scatter plot between AcenFac (resp. Entropy, SNar, HNar, HVAP, DHVAP) of octane isomers and $SO_{1}(G)$.}
 \label{fig-75}
\end{figure}

\begin{figure}[ht!]
  \centering
  \scalebox{.08}[.08]{\includegraphics{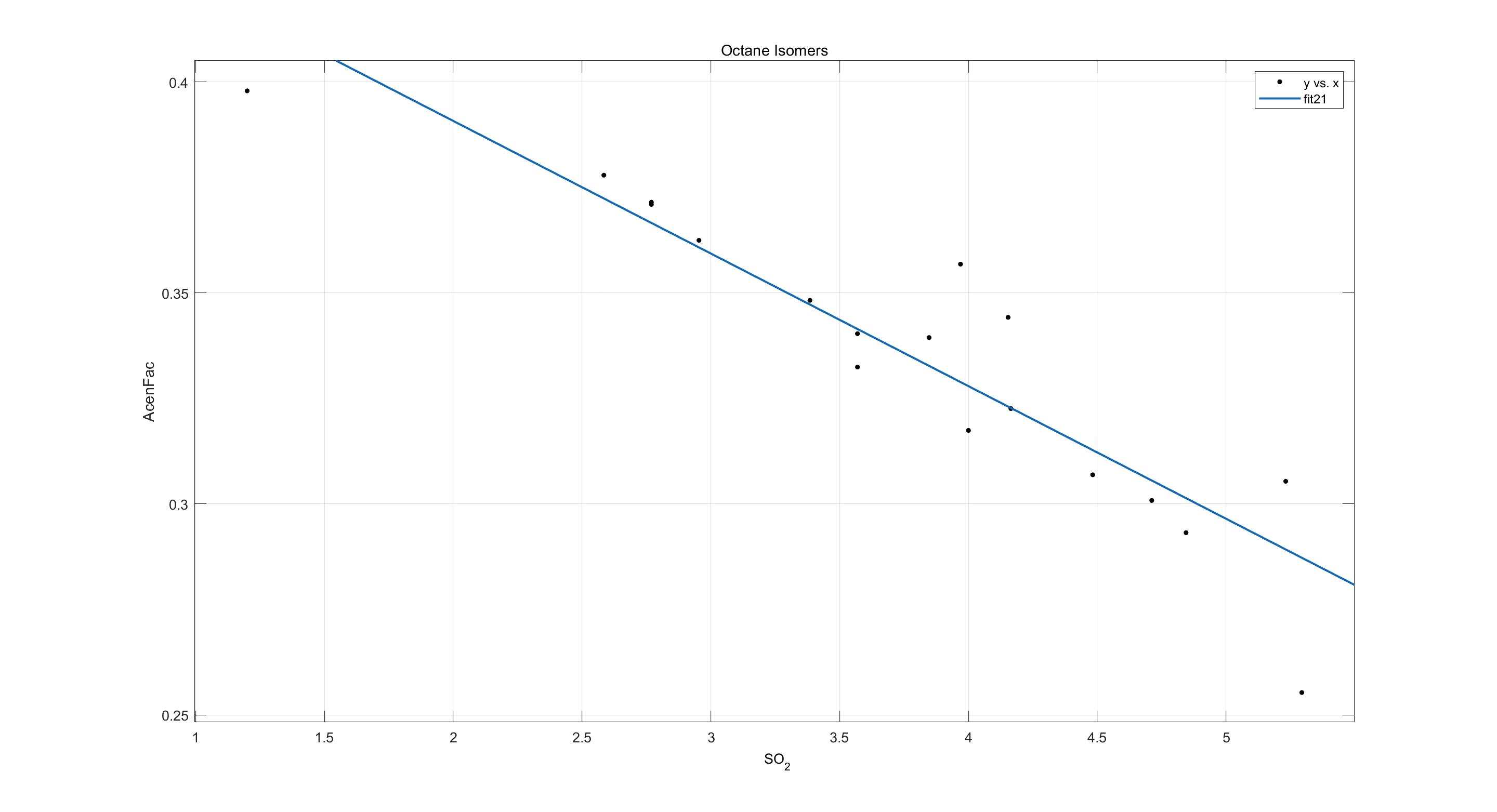}}
  \scalebox{.08}[.08]{\includegraphics{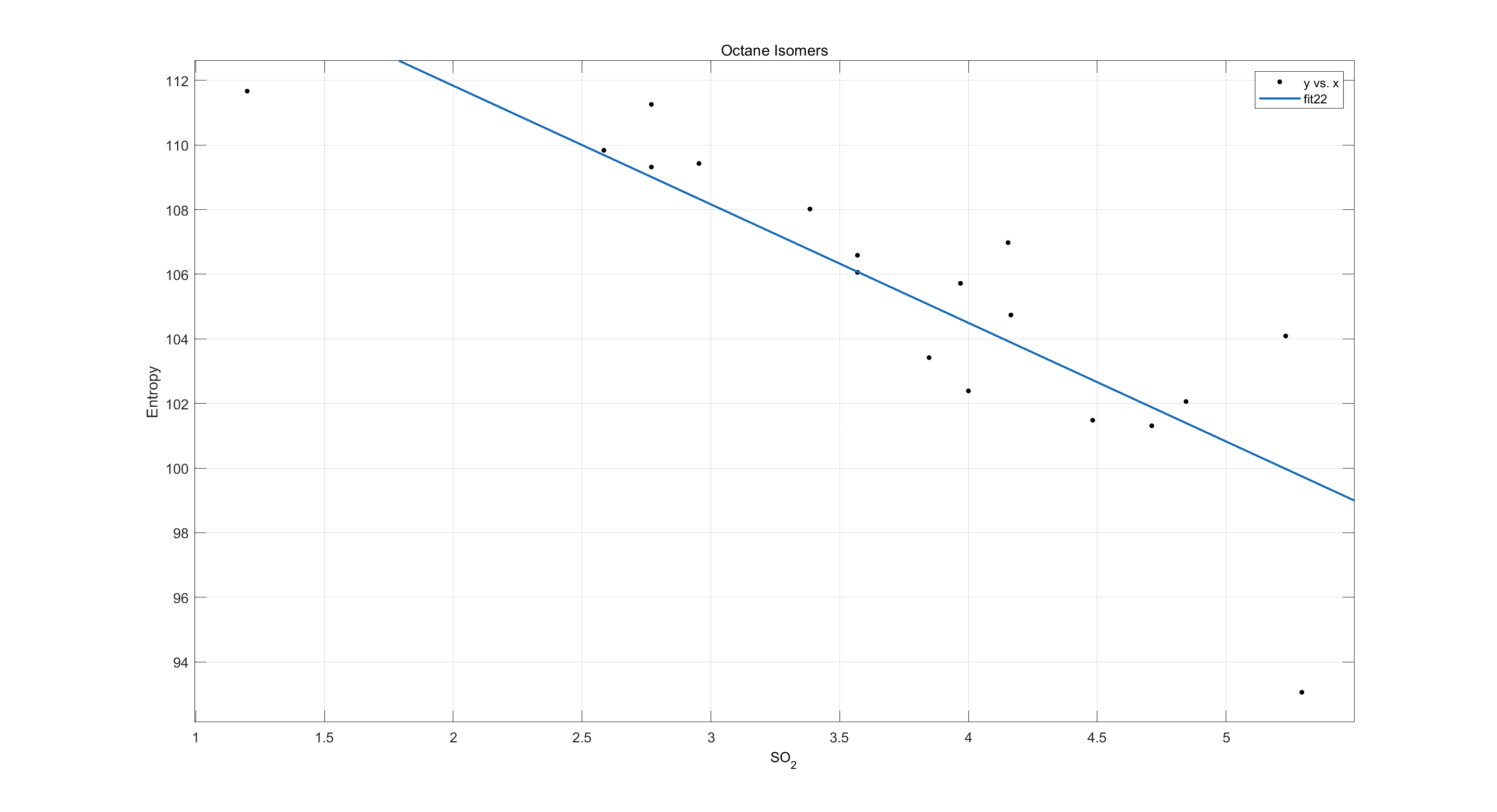}}
  \scalebox{.08}[.08]{\includegraphics{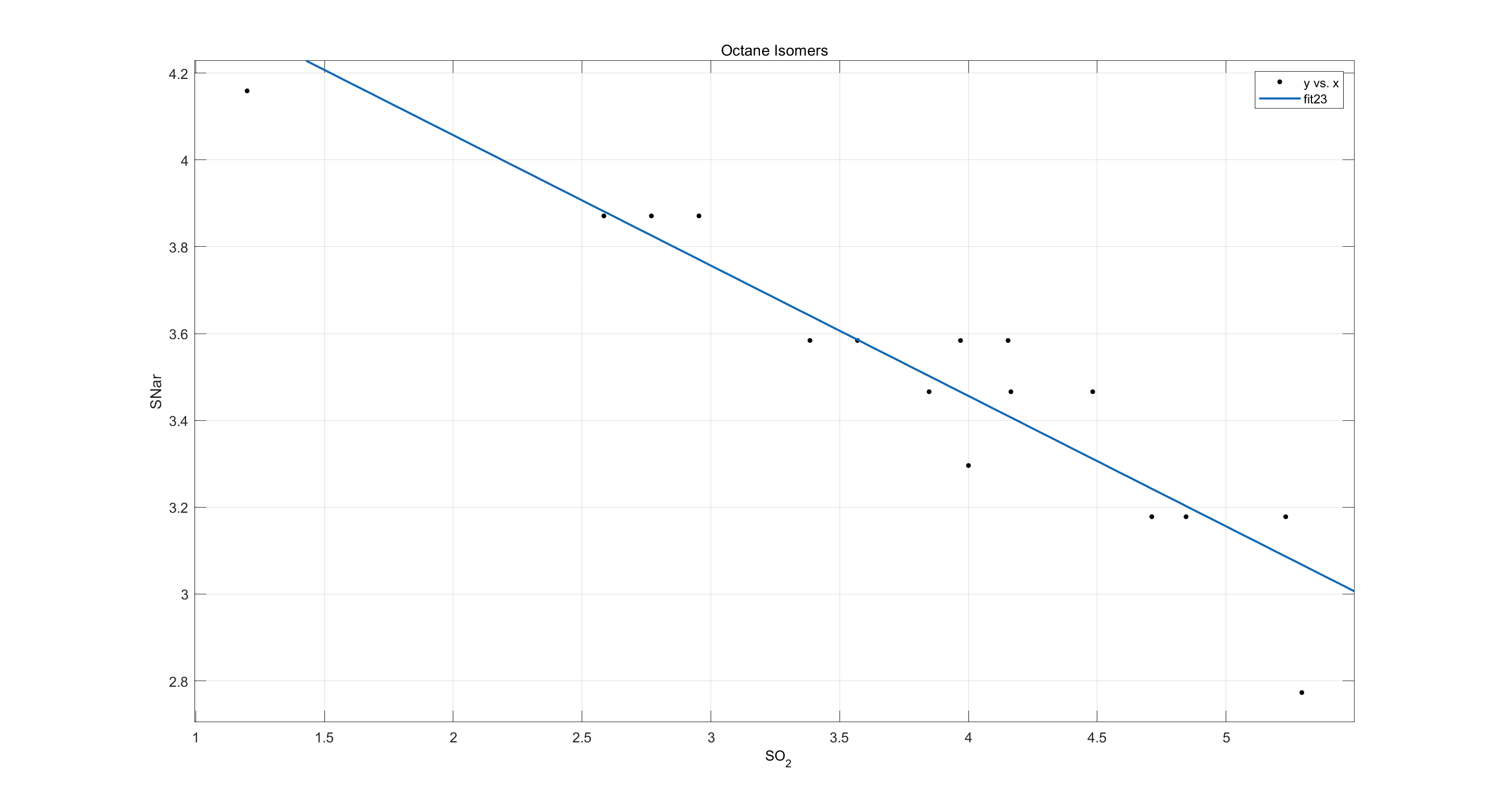}}
  \scalebox{.08}[.08]{\includegraphics{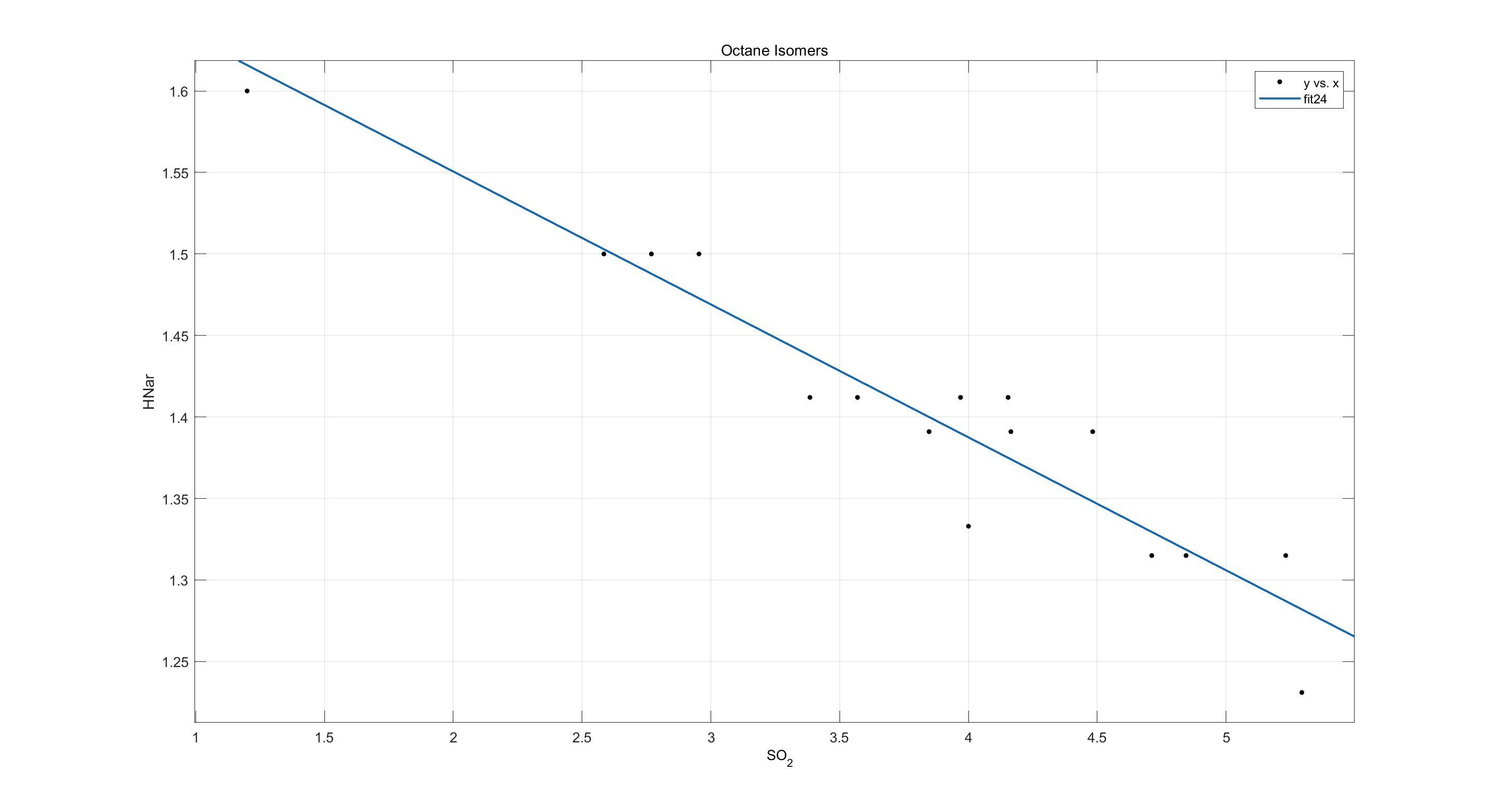}}
  \scalebox{.08}[.08]{\includegraphics{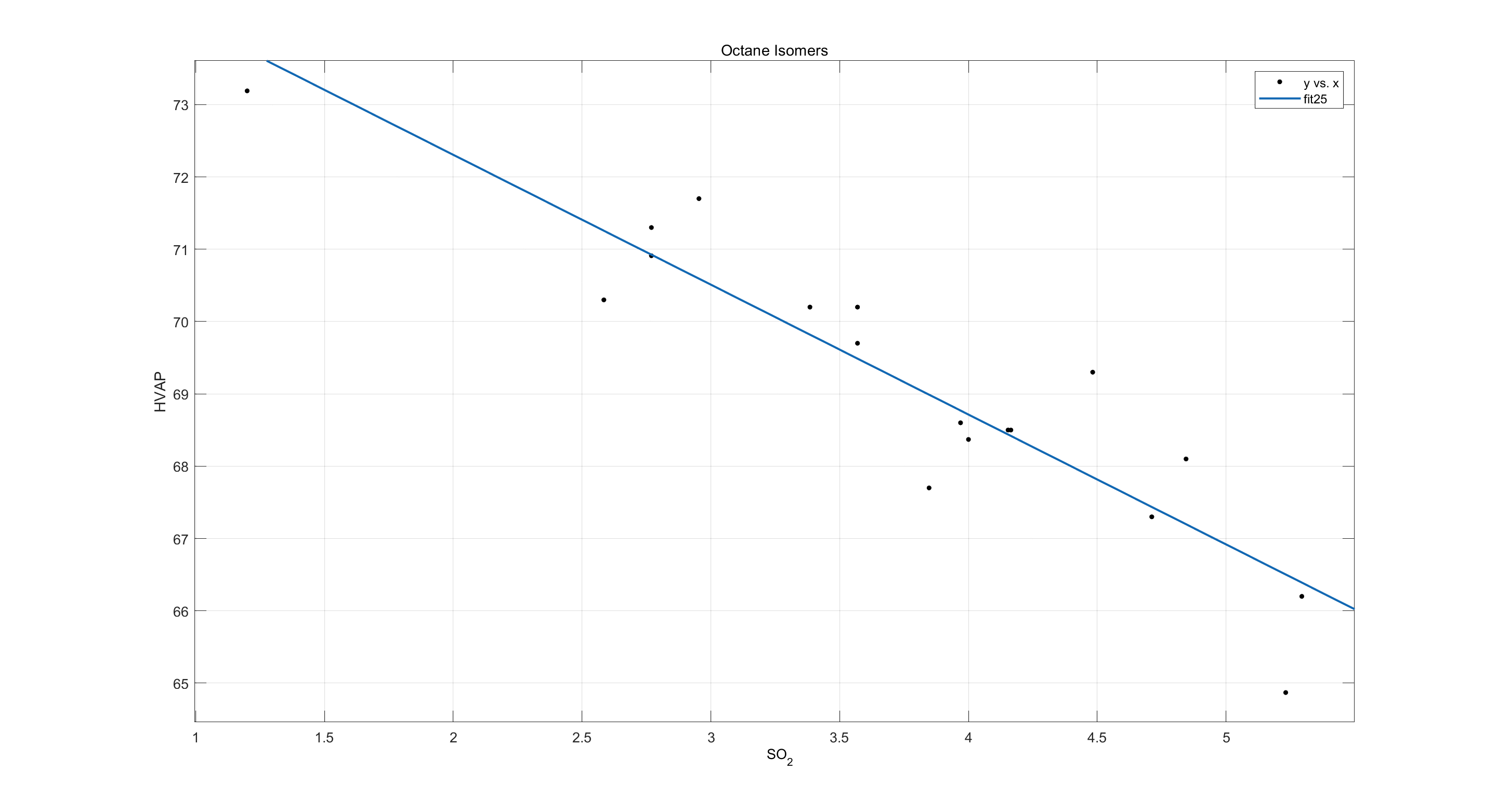}}
  \scalebox{.08}[.08]{\includegraphics{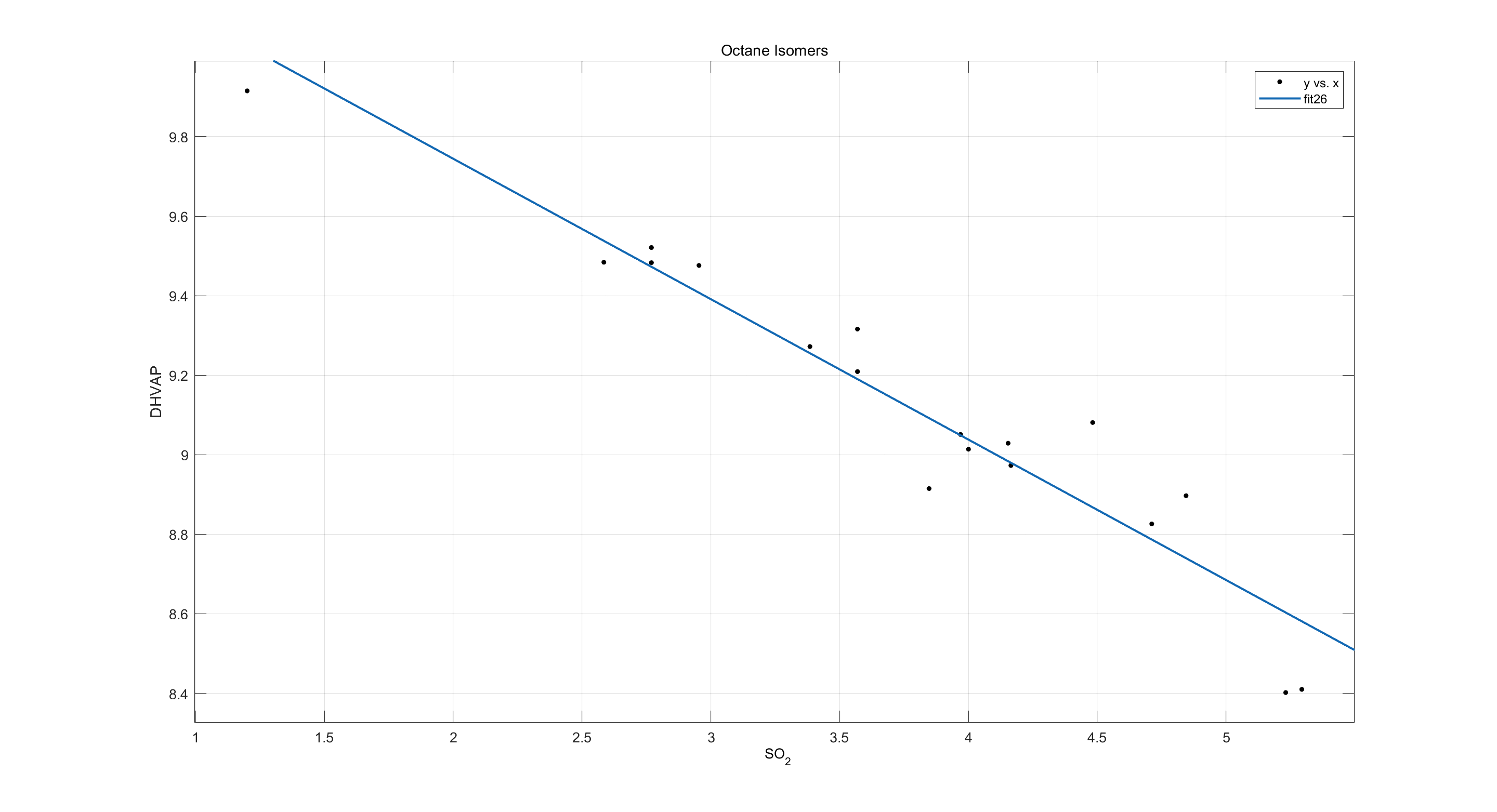}}
  \caption{Scatter plot between AcenFac (resp. Entropy, SNar, HNar, HVAP, DHVAP) of octane isomers and $SO_{2}(G)$.}
 \label{fig-76}
\end{figure}

\begin{figure}[ht!]
  \centering
  \scalebox{.08}[.08]{\includegraphics{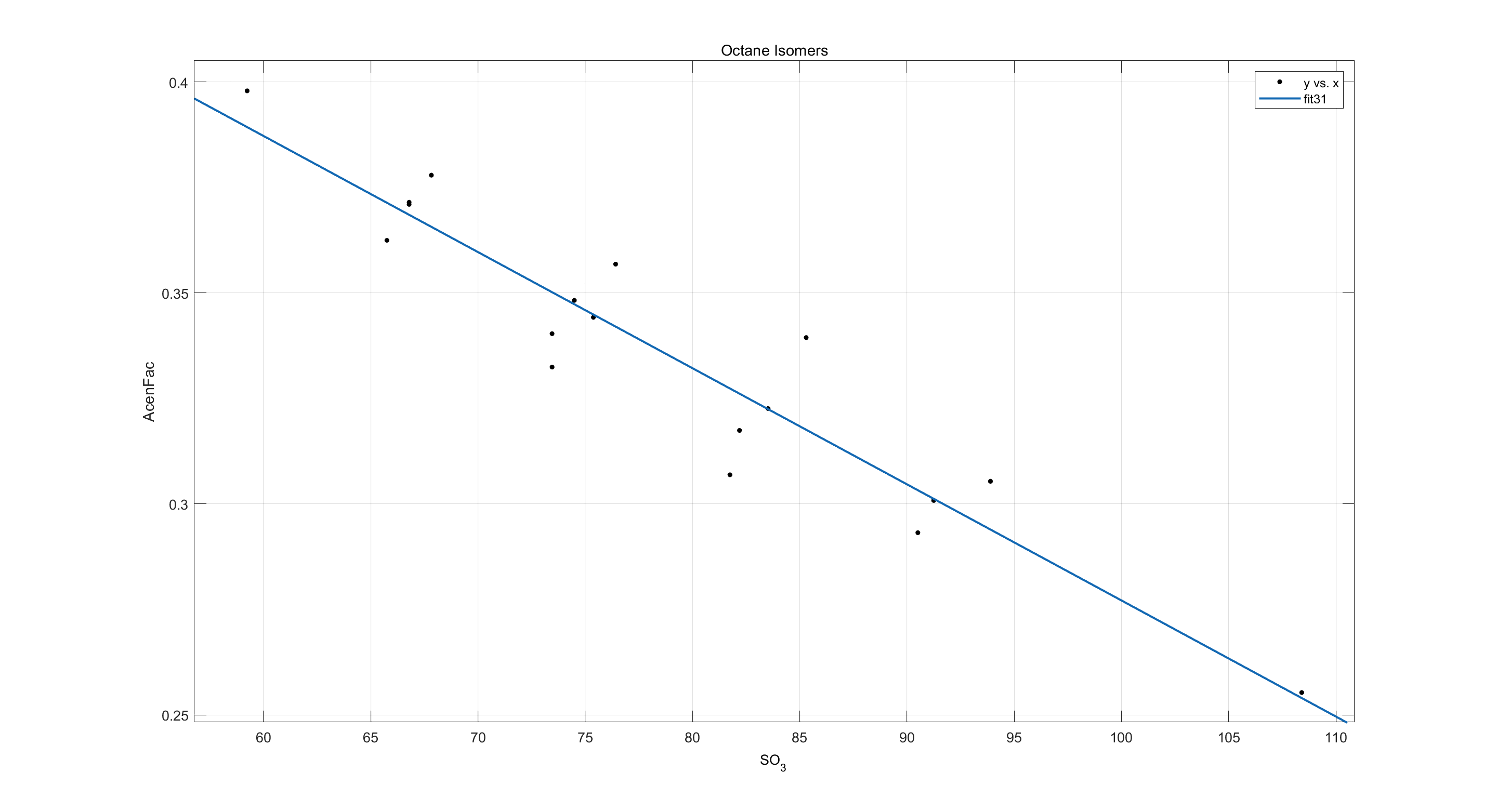}}
  \scalebox{.08}[.08]{\includegraphics{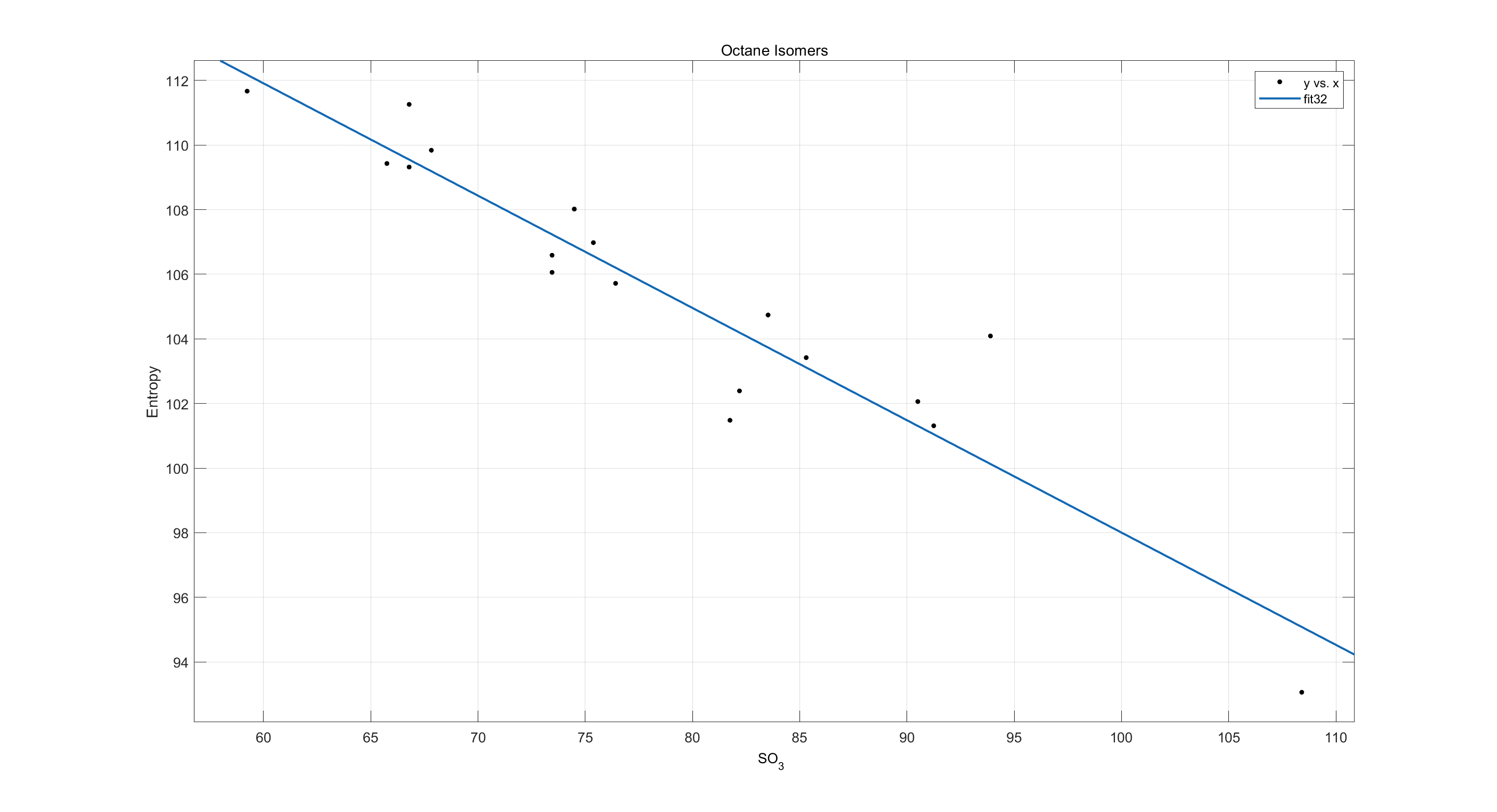}}
  \scalebox{.08}[.08]{\includegraphics{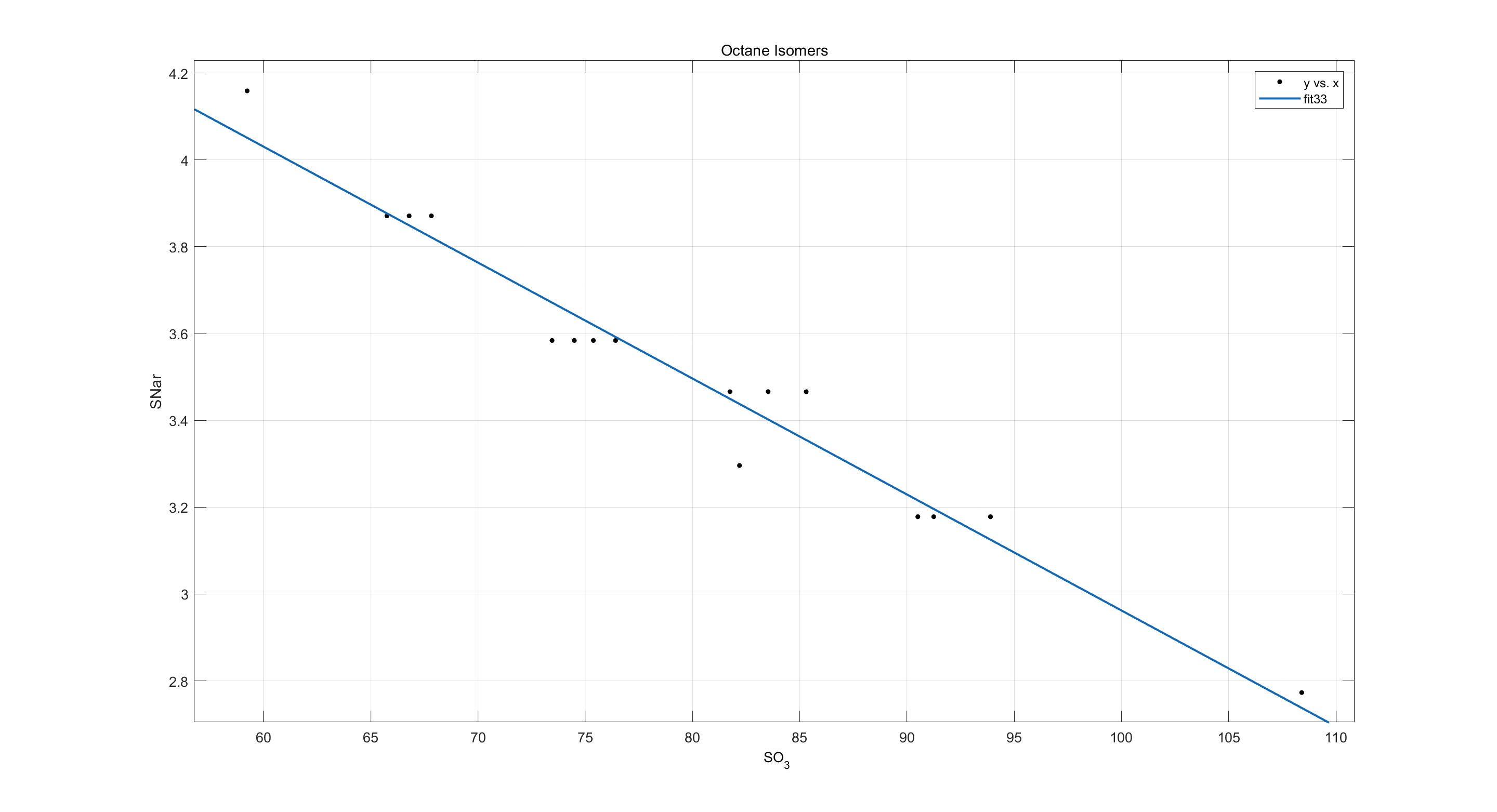}}
  \scalebox{.08}[.08]{\includegraphics{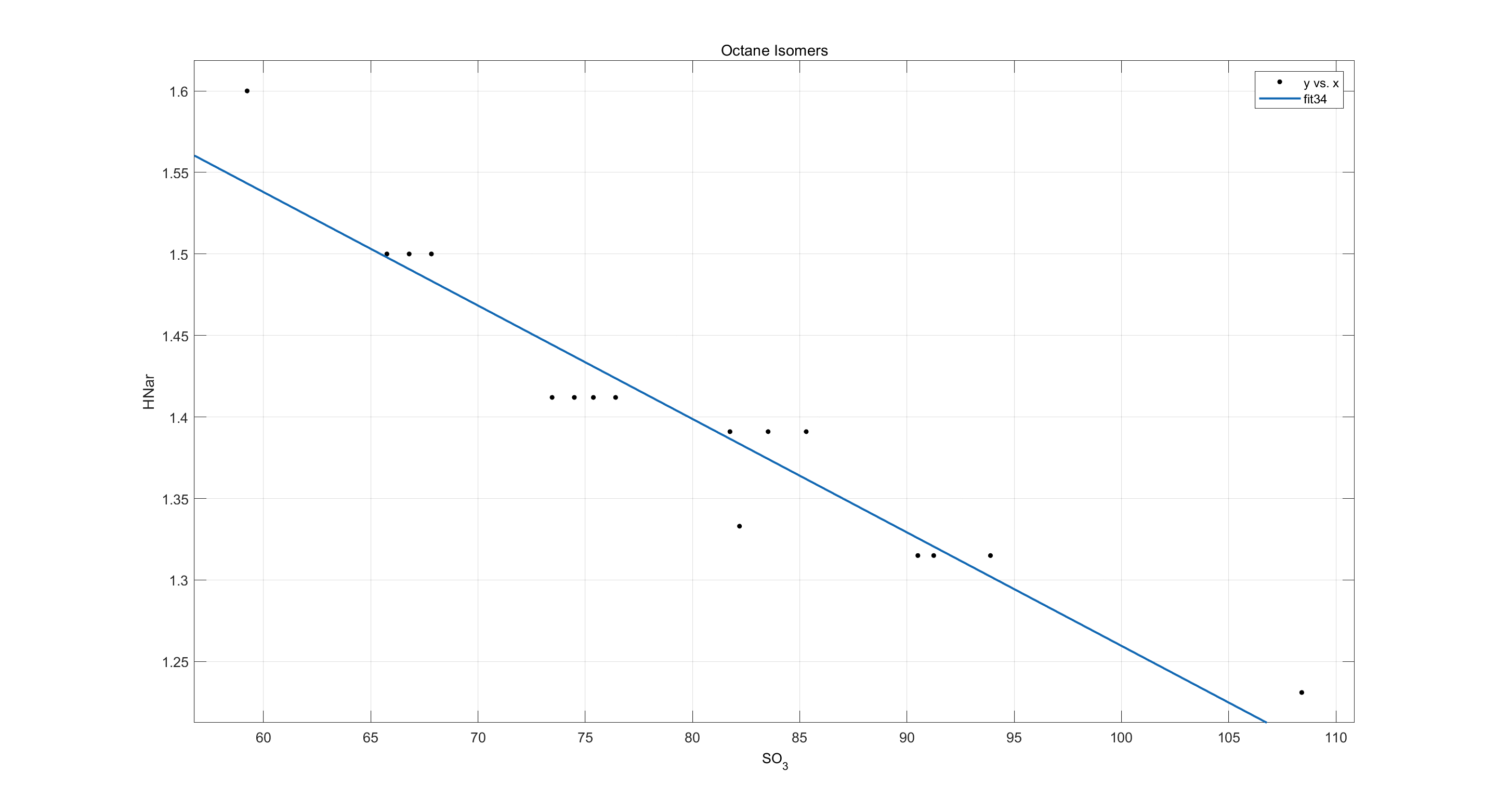}}
  \scalebox{.08}[.08]{\includegraphics{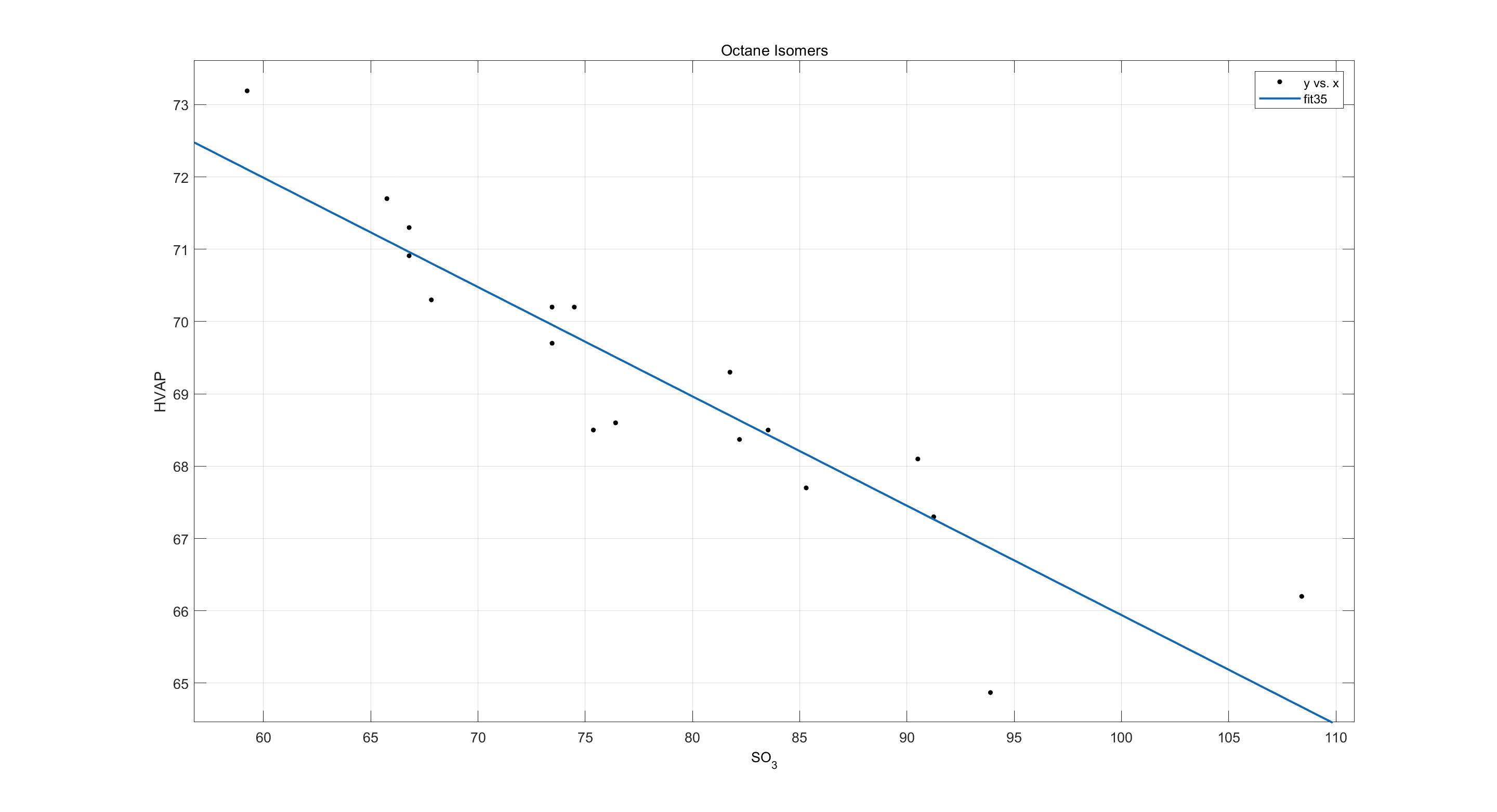}}
  \scalebox{.08}[.08]{\includegraphics{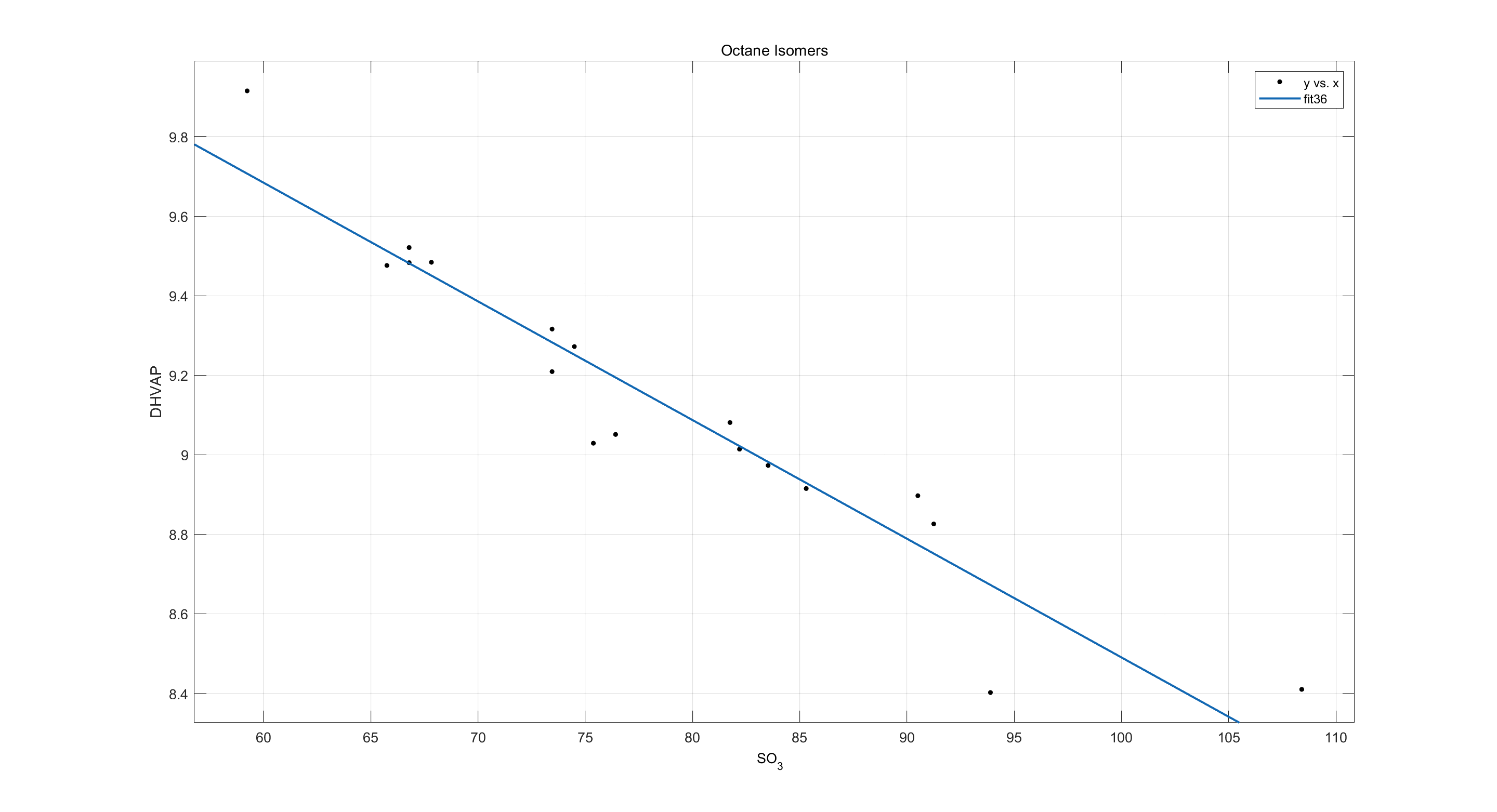}}
  \caption{Scatter plot between AcenFac (resp. Entropy, SNar, HNar, HVAP, DHVAP) of octane isomers and $SO_{3}(G)$.}
 \label{fig-77}
\end{figure}

\begin{figure}[ht!]
  \centering
  \scalebox{.08}[.08]{\includegraphics{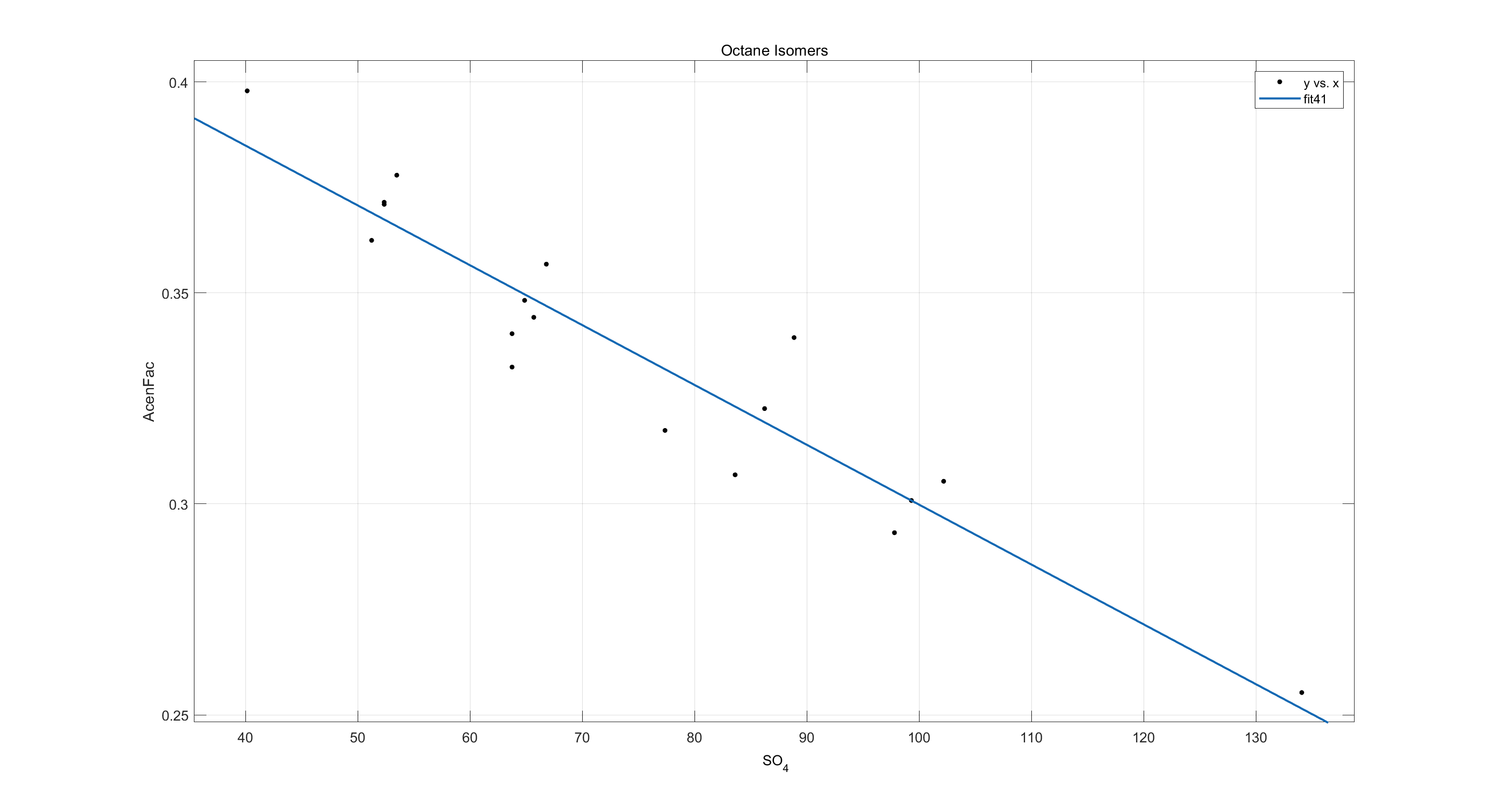}}
  \scalebox{.08}[.08]{\includegraphics{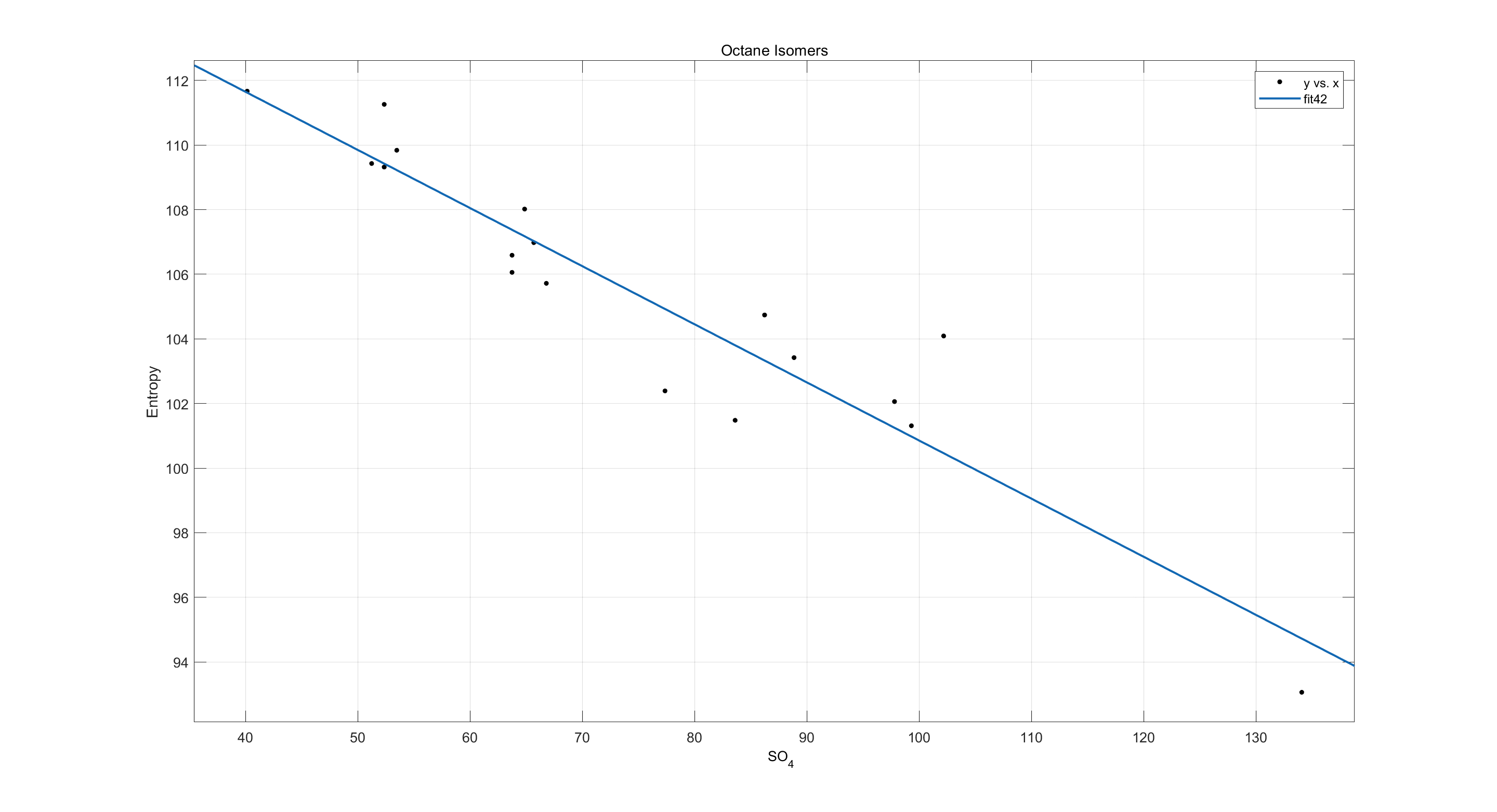}}
  \scalebox{.08}[.08]{\includegraphics{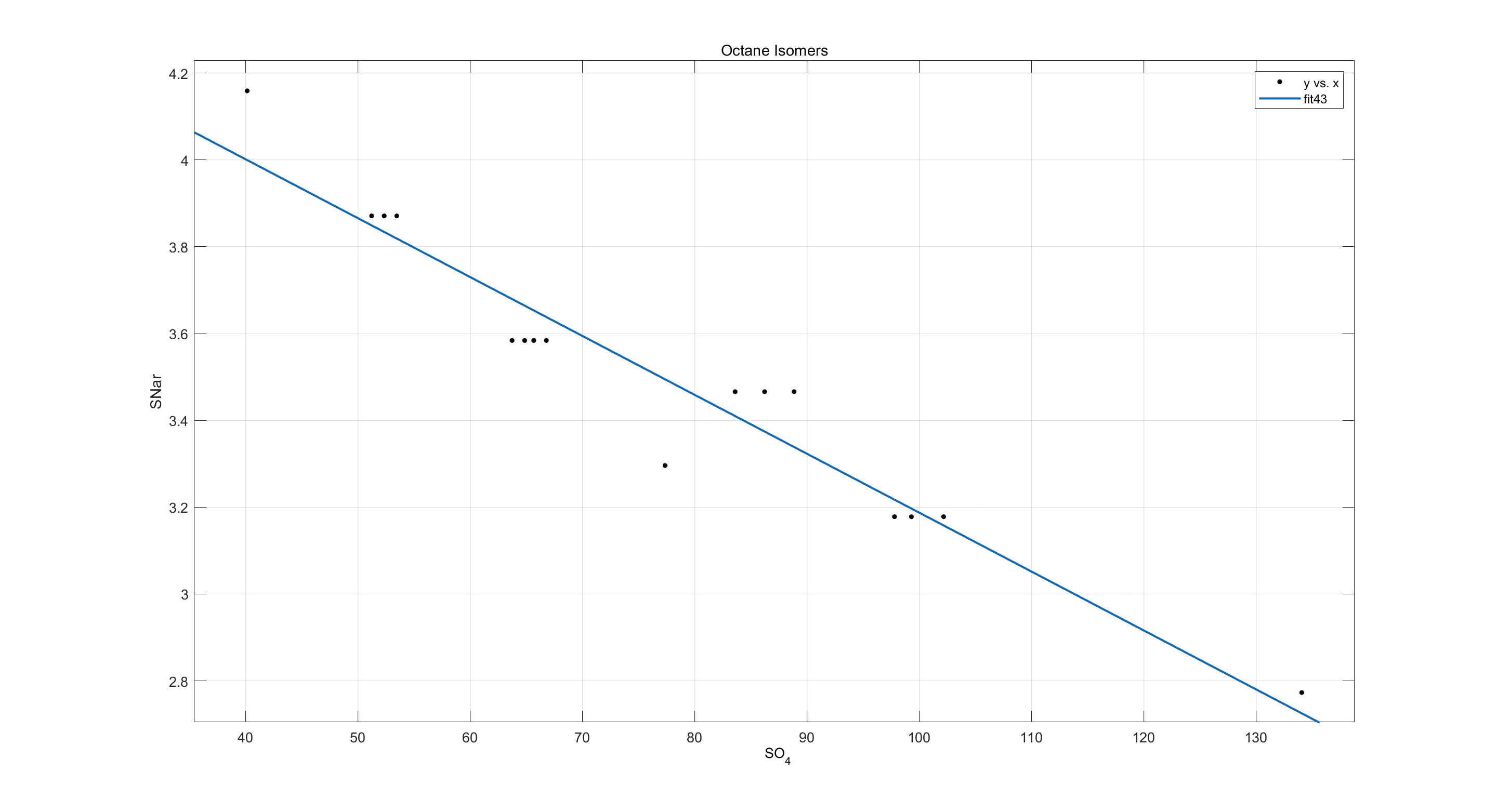}}
  \scalebox{.08}[.08]{\includegraphics{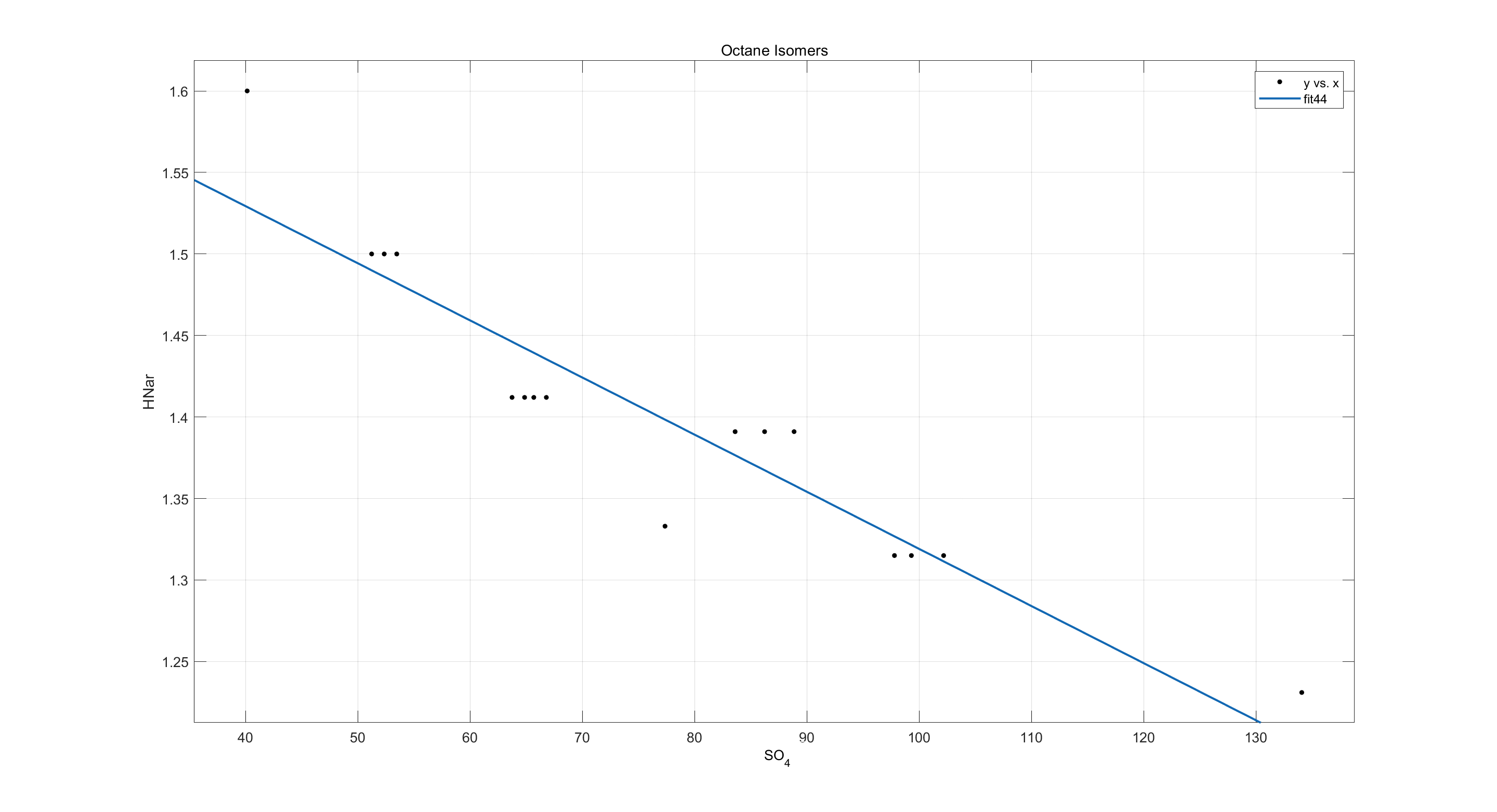}}
  \scalebox{.08}[.08]{\includegraphics{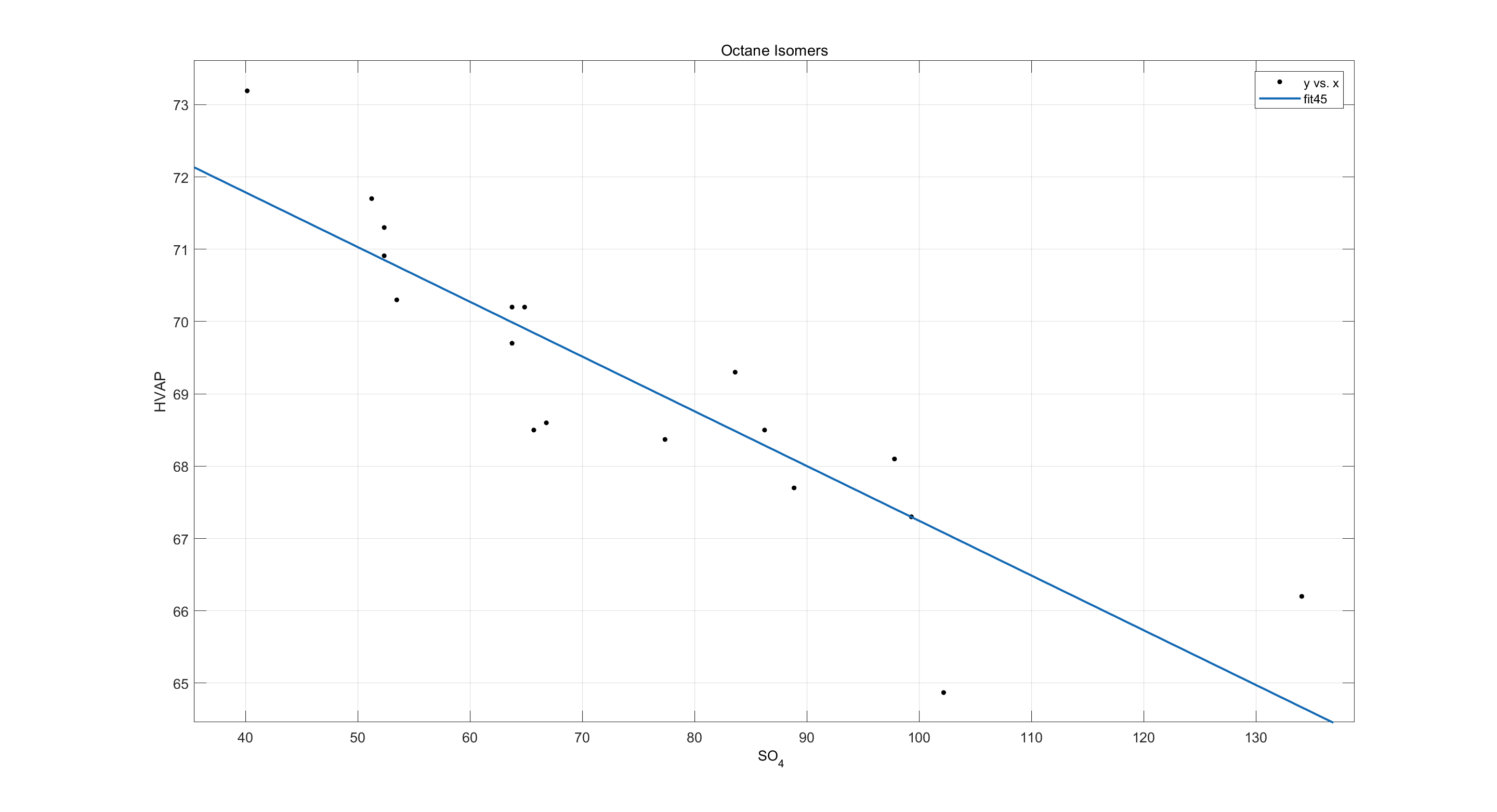}}
  \scalebox{.08}[.08]{\includegraphics{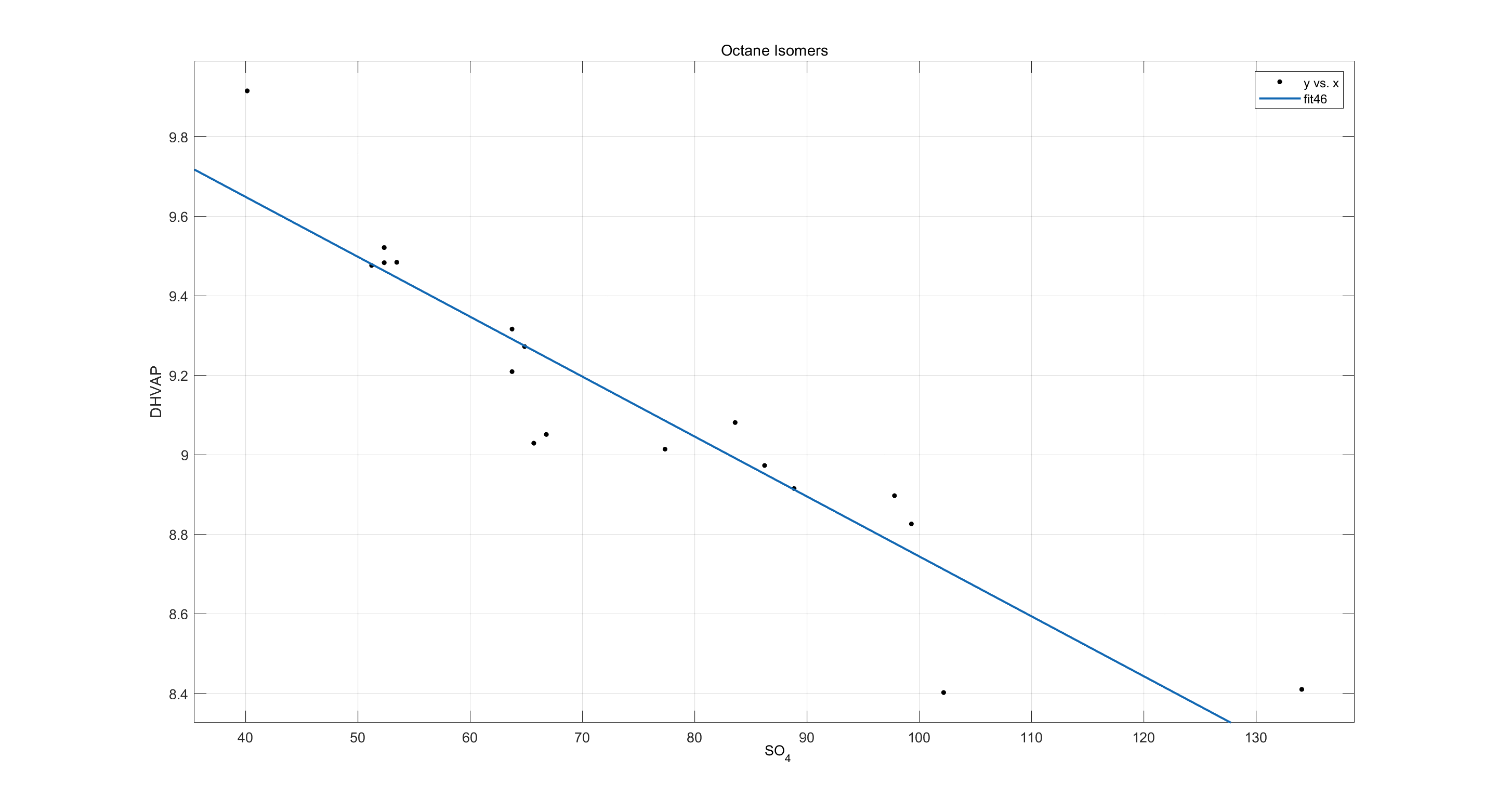}}
  \caption{Scatter plot between AcenFac (resp. Entropy, SNar, HNar, HVAP, DHVAP) of octane isomers and $SO_{4}(G)$.}
 \label{fig-78}
\end{figure}

\begin{figure}[ht!]
  \centering
  \scalebox{.08}[.08]{\includegraphics{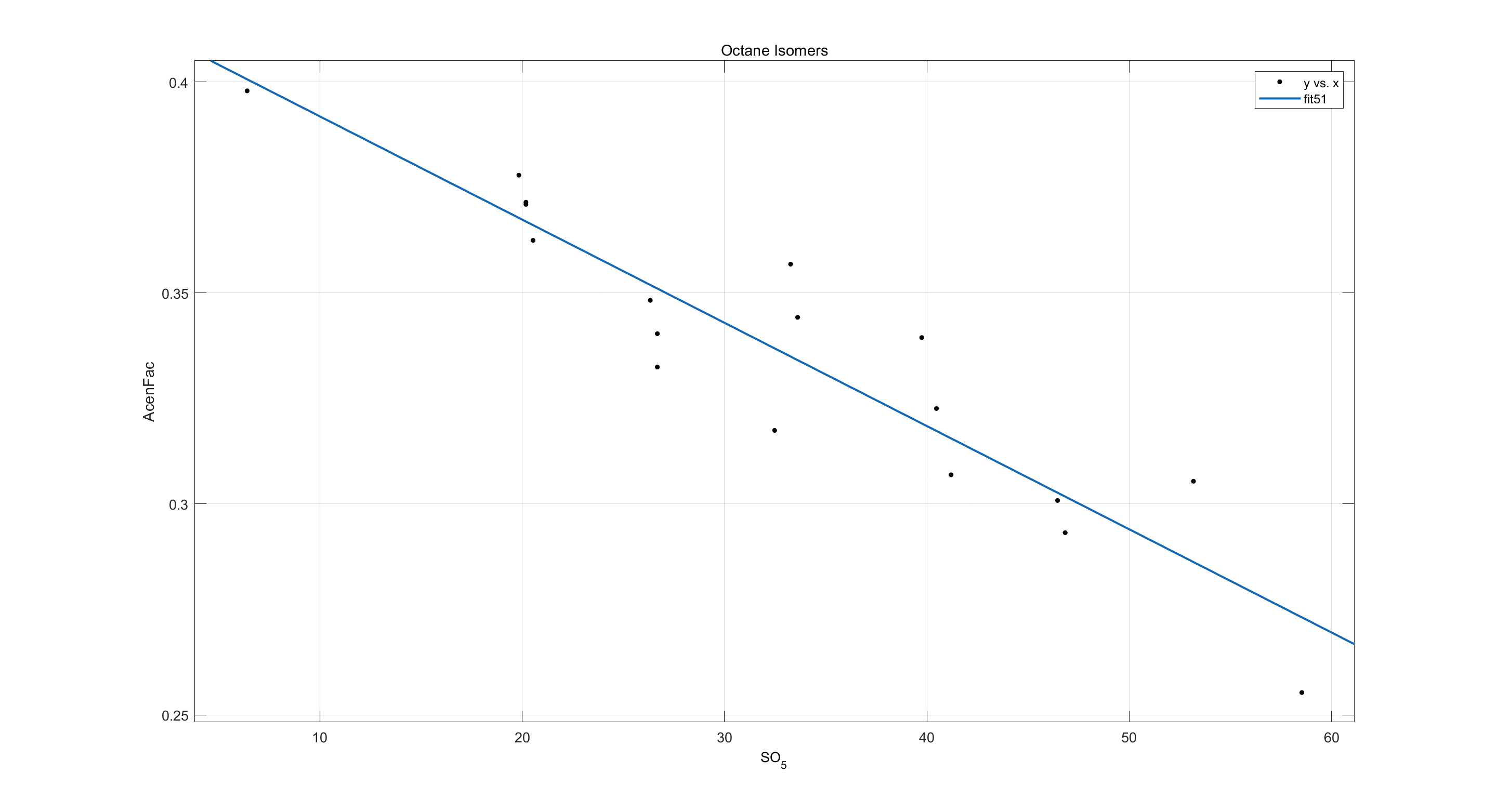}}
  \scalebox{.08}[.08]{\includegraphics{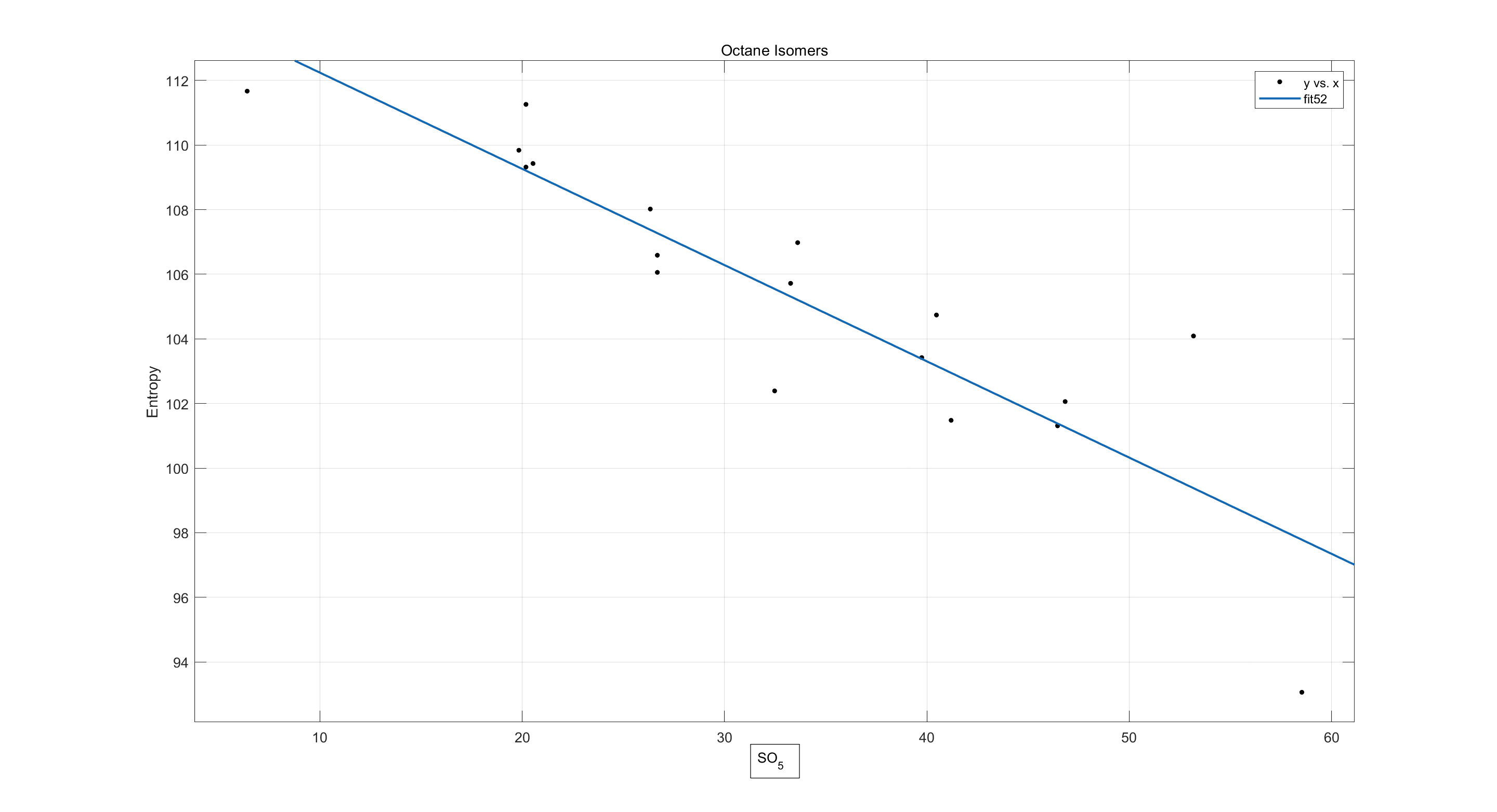}}
  \scalebox{.08}[.08]{\includegraphics{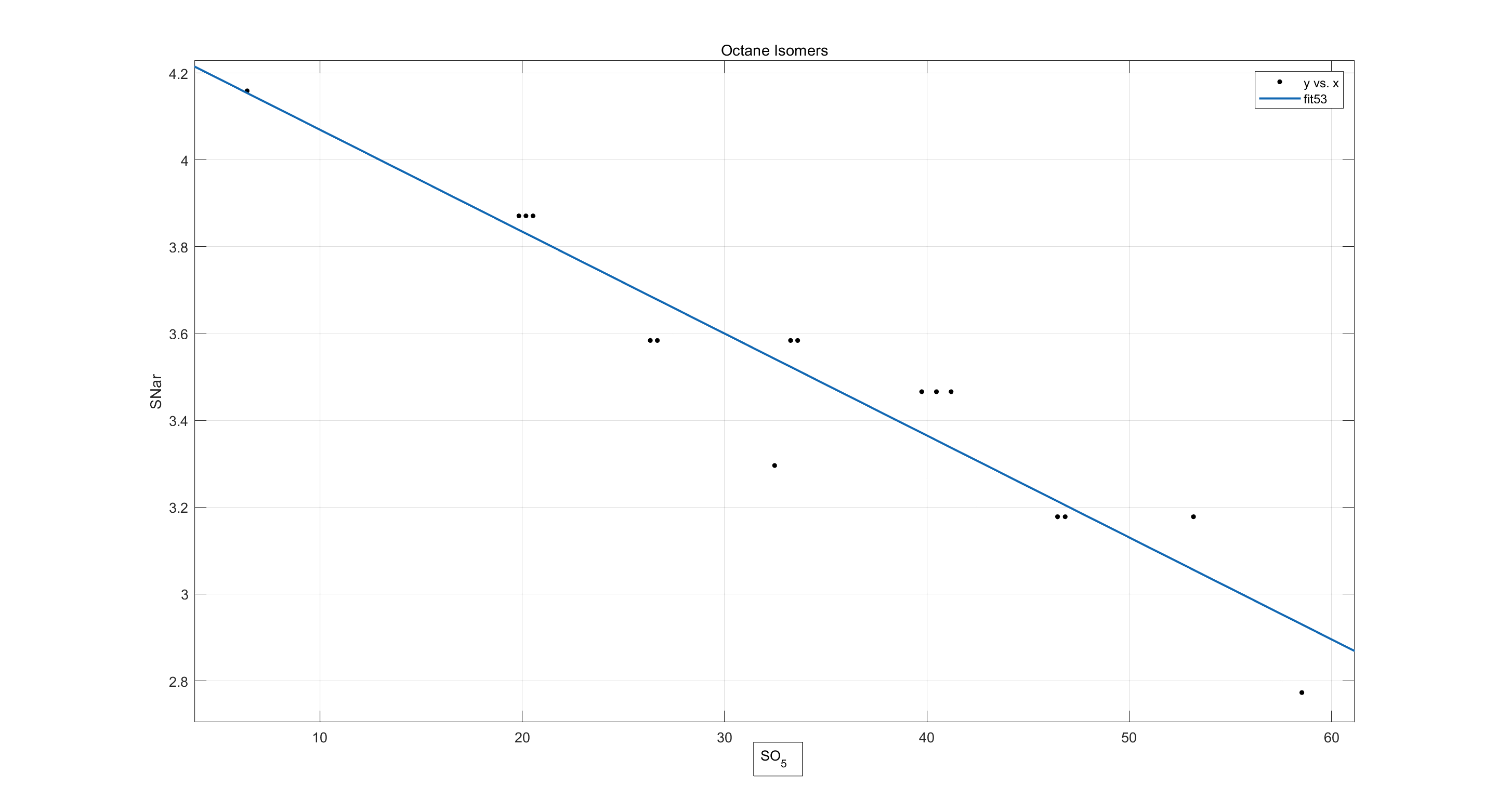}}
  \scalebox{.08}[.08]{\includegraphics{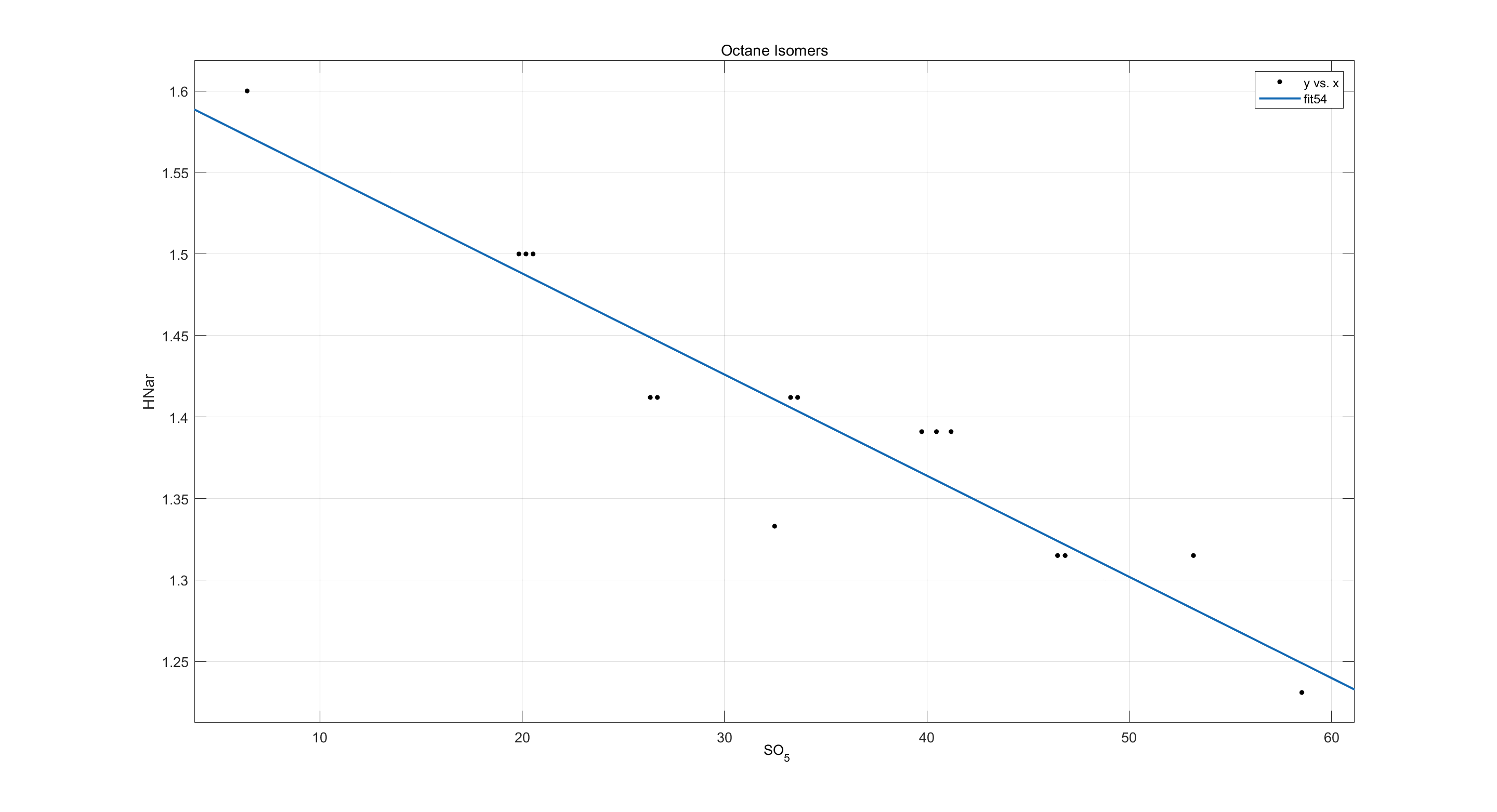}}
  \scalebox{.08}[.08]{\includegraphics{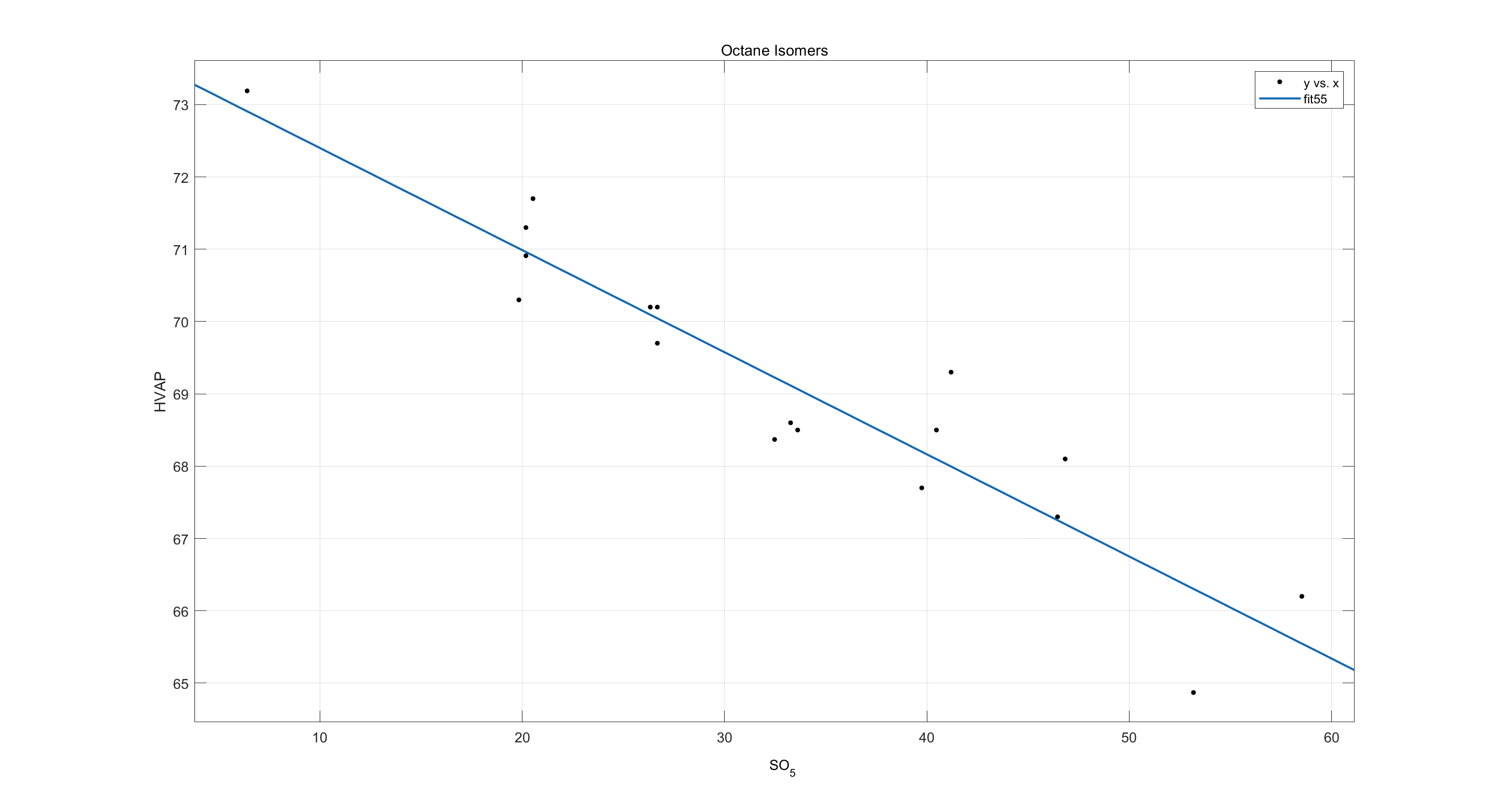}}
  \scalebox{.08}[.08]{\includegraphics{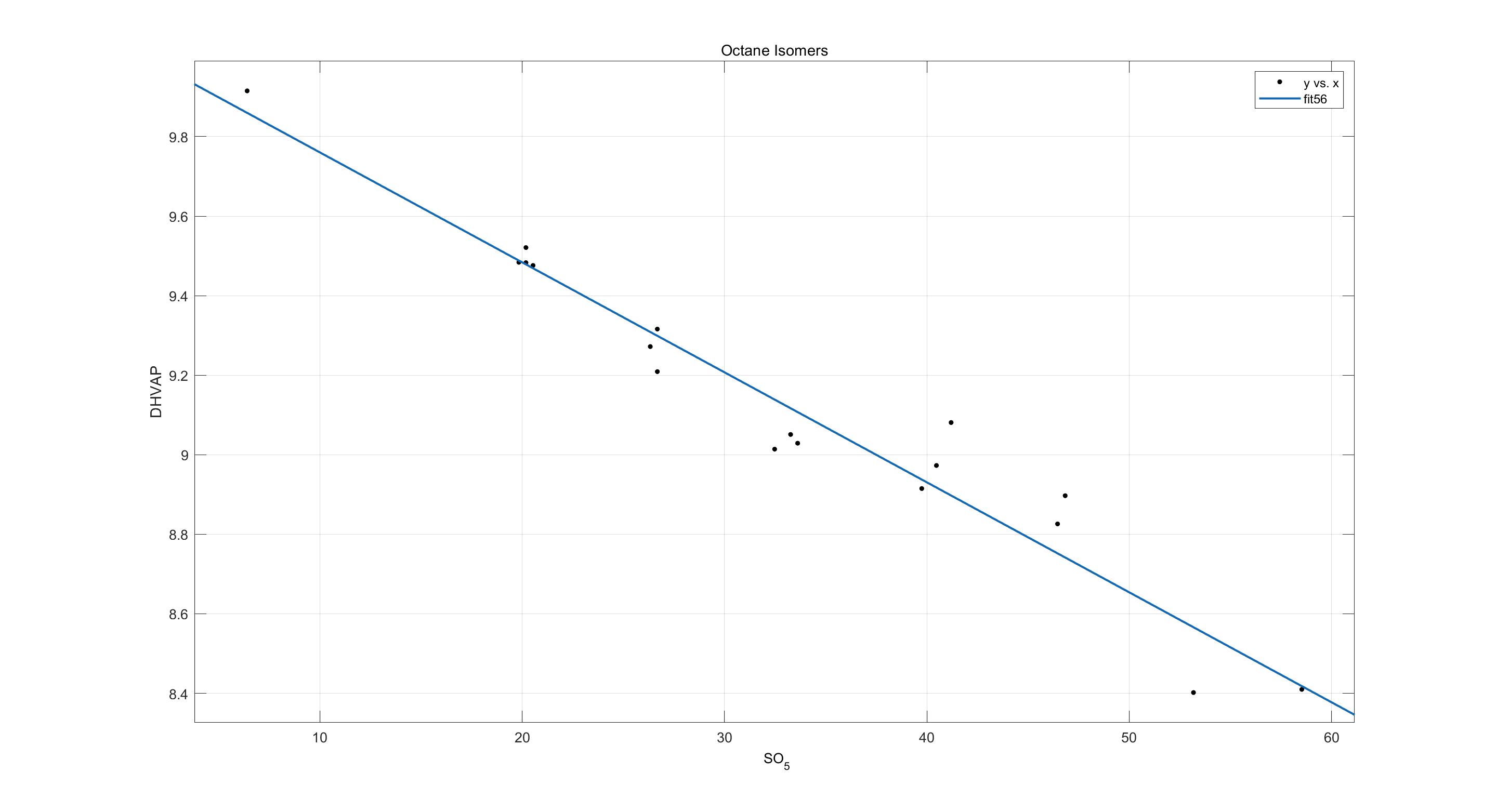}}
  \caption{Scatter plot between AcenFac (resp. Entropy, SNar, HNar, HVAP, DHVAP) of octane isomers and $SO_{5}(G)$.}
 \label{fig-78}
\end{figure}

\begin{figure}[ht!]
  \centering
  \scalebox{.08}[.08]{\includegraphics{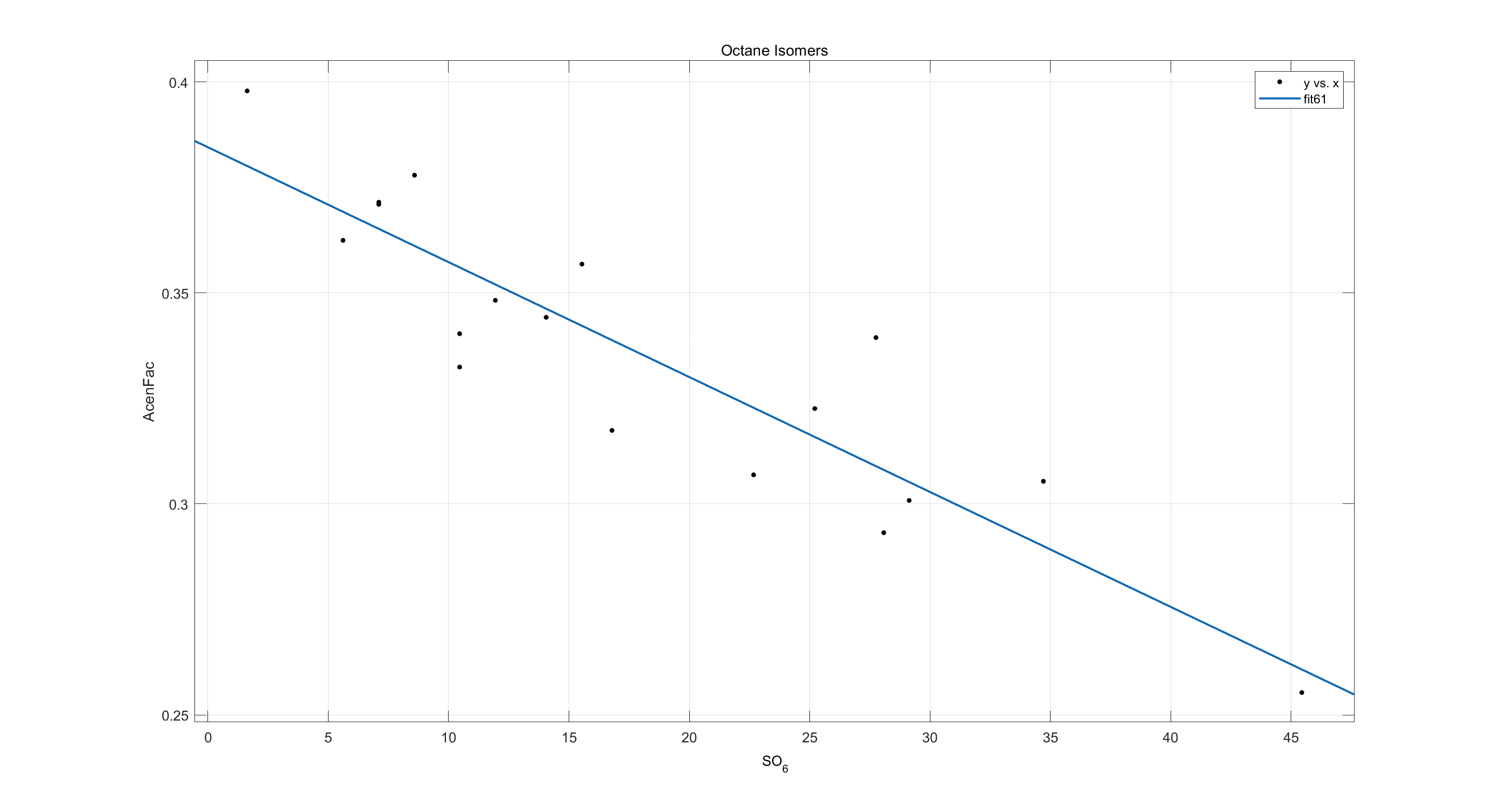}}
  \scalebox{.08}[.08]{\includegraphics{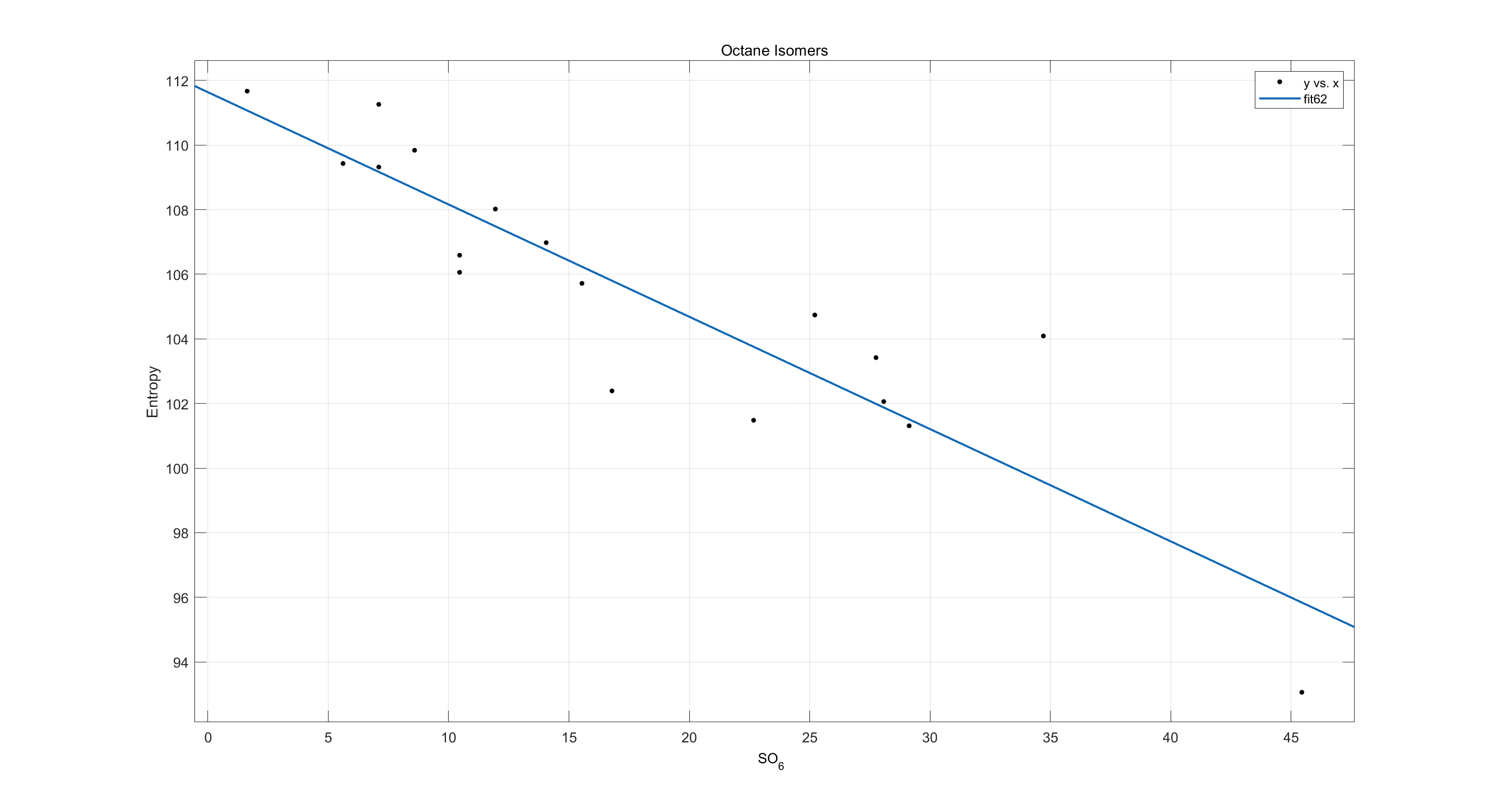}}
  \scalebox{.08}[.08]{\includegraphics{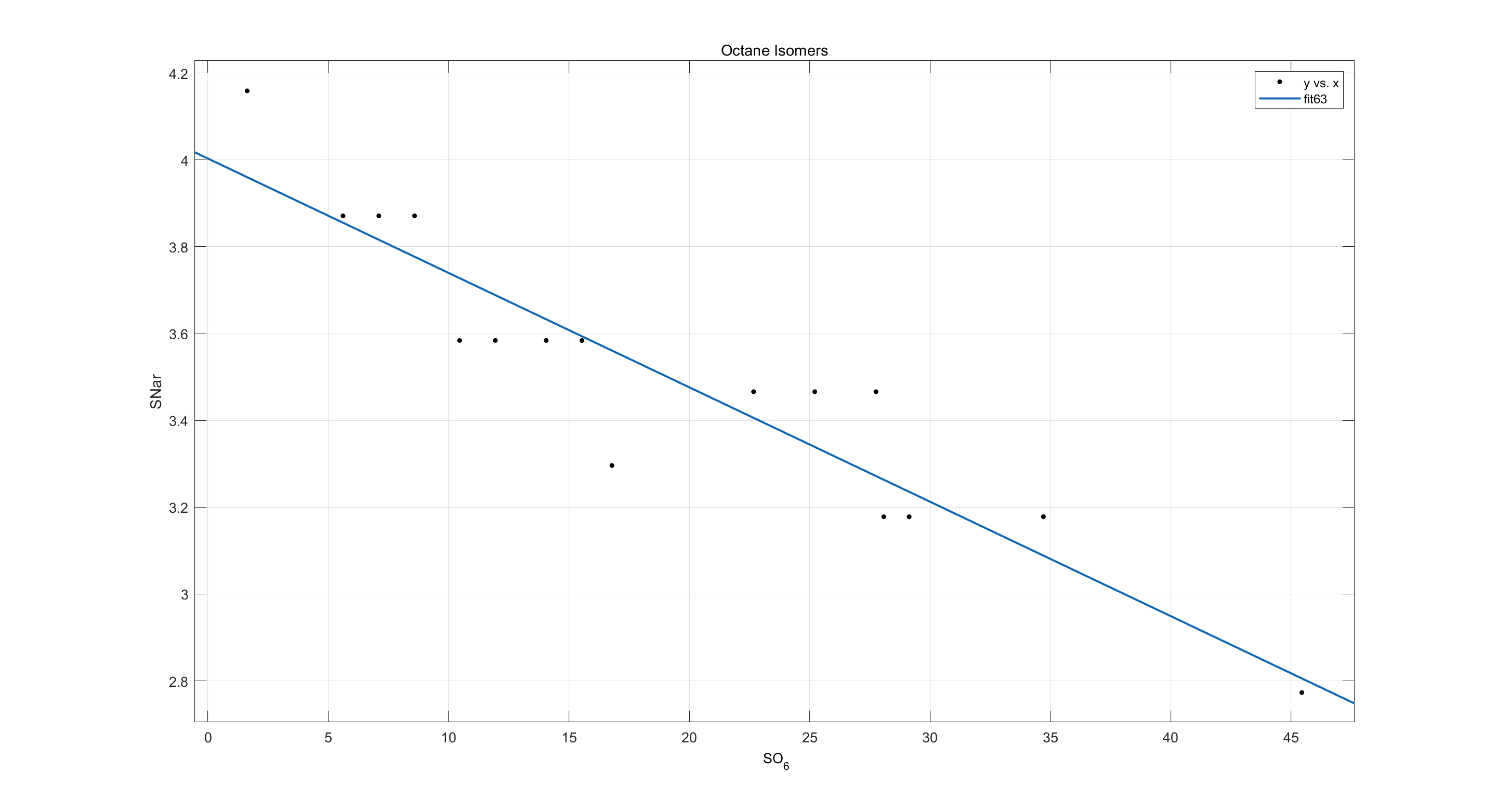}}
  \scalebox{.08}[.08]{\includegraphics{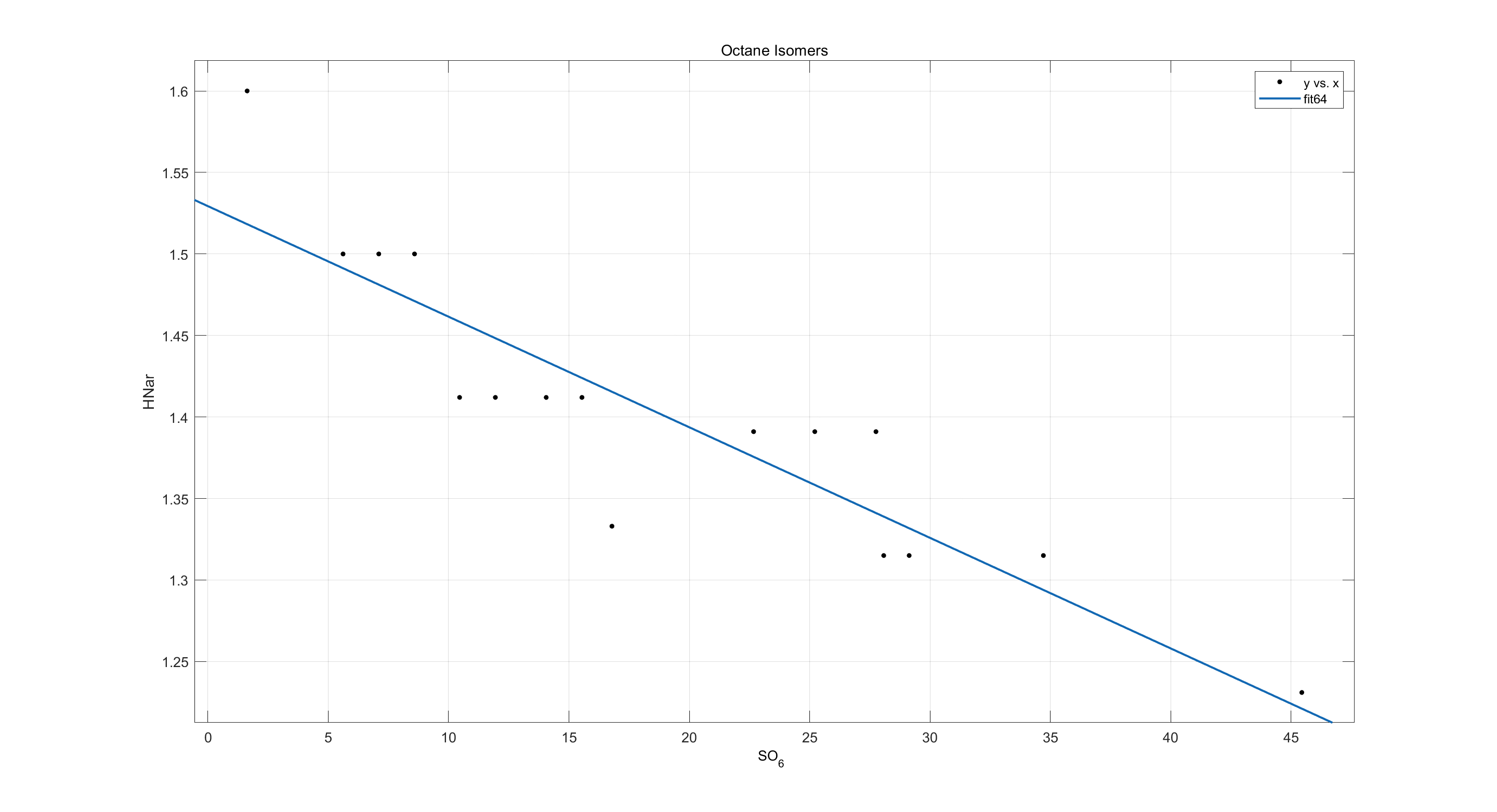}}
  \scalebox{.08}[.08]{\includegraphics{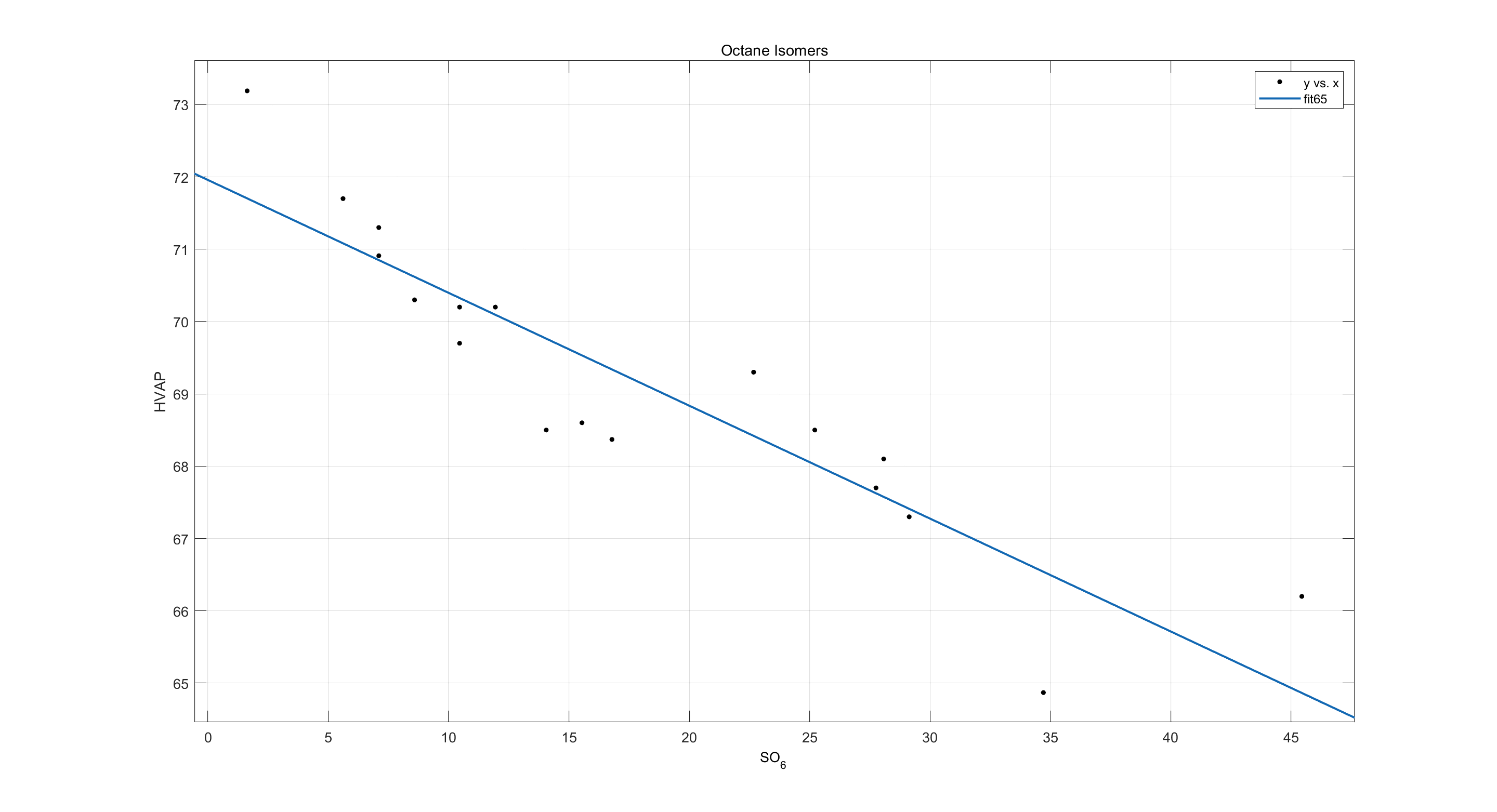}}
  \scalebox{.08}[.08]{\includegraphics{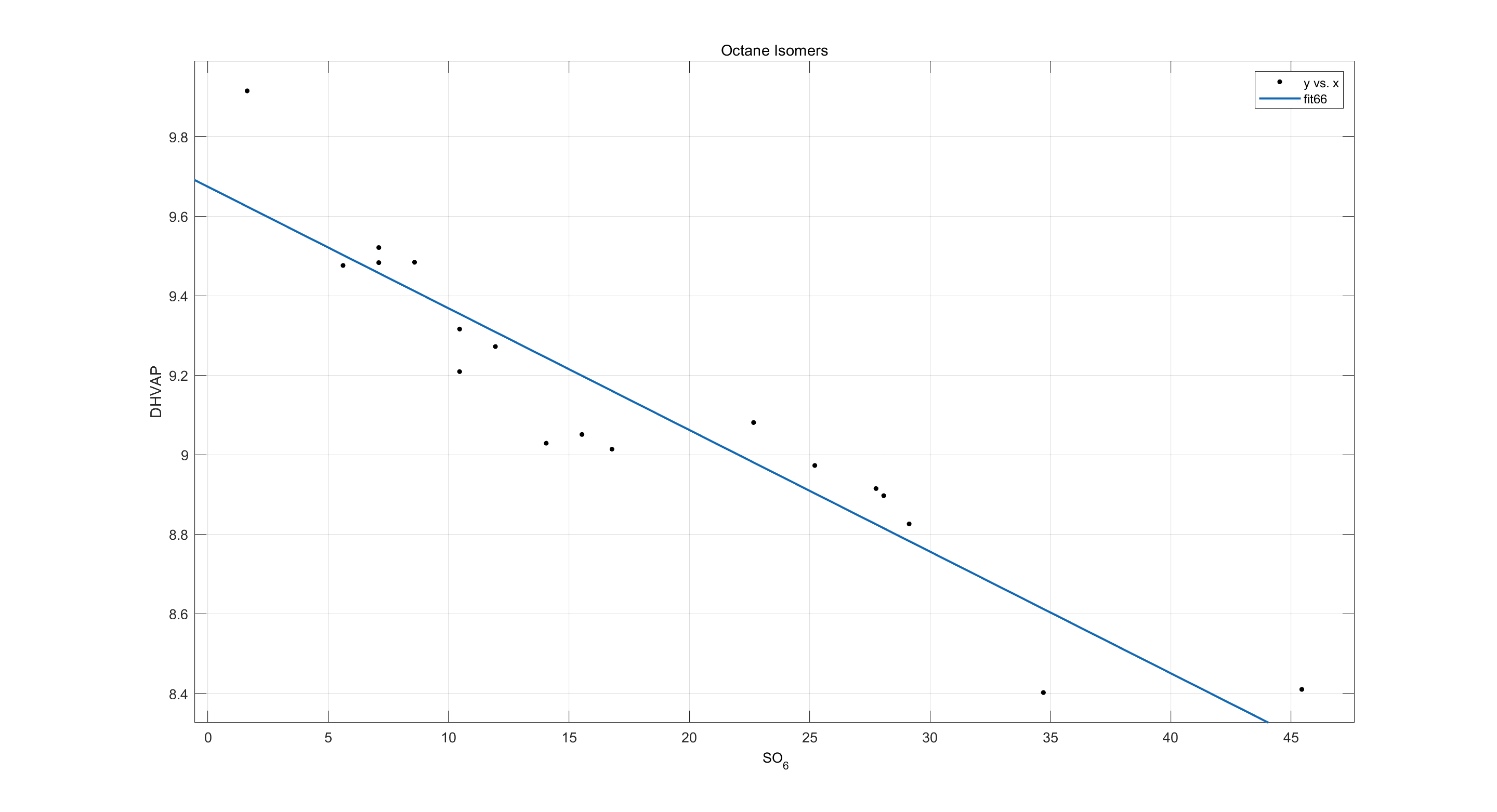}}
  \caption{Scatter plot between AcenFac (resp. Entropy, SNar, HNar, HVAP, DHVAP ) of octane isomers and $SO_{6}(G)$.}
 \label{fig-79}
\end{figure}

\begin{table}[!htb]
	\centering
    \caption{$|R|$ between Sombor-index-like invariants and Acentric Factors, Entropy, SNar, HNar of 18 octane isomers.}
     \setlength{\tabcolsep}{0.7mm}{
	\begin{tabular}{c|ccccccc}\hline
	Physico-chemical property    &	   $SO$  &      $SO_{1}$    &    $SO_{2}$    &    $SO_{3}$       &    $SO_{4}$  &           $SO_{5}$       &    $SO_{6}$          \\ \hline

	Acentric Factors             &     $0.9594$    &	  $0.9192$   &	  $0.9202$    &	  $0.9482$  &    $0.9466$  &      $0.9303$    &    $0.9029$       \\ \hline

    Entropy                      &     $0.9465$    &	  $0.8898$   &	  $0.8433$   &	  $0.9399$  &    $0.9422$  &      $0.8890$    &    $0.9043$      \\ \hline

    SNar                         &     $0.9842$    &	  $0.9311$   &	  $0.9355$   &	  $0.9791$  &    $0.9634$  &      $0.9501$    &    $0.9295$       \\ \hline

    HNar                         &     $0.9618$    &	  $0.9098$   &	  $0.9512$   &	  $0.9551$  &    $0.9318$  &      $0.9402$    &    $0.8959$       \\ \hline

    HVAP                        &     $0.9031$    &	  $0.9174$   &	  $0.9198$   &	  $0.9112$  &    $0.8837$  &      $0.9393$    &    $0.9053$       \\ \hline

    DHVAP                         &     $0.9469$    &	  $0.9522$   &	  $0.9560$   &	  $0.9509$  &    $0.93$  &      $0.9726$    &    $0.9379$       \\ \hline
	\end{tabular}}
	
	\label{table3}
\end{table}

It reveals the supremacy of Sombor index compared to other Sombor-index-like invariants in Table \ref{table3} in modeling AcenFac, SNar, HNar, HVAP and DHVAP. Sombor index always shows better predictive capability than other Sombor-index-like invariants for AcenFac, SNar and HNar.

\section{Conclusions}
\hskip 0.6cm
In this paper, we consider the mathematical and chemistry properties of these geometry-based invariants. We determined the maximum trees (resp. unicyclic graphs) with given diameter, the maximum trees with given matching number, the maximum trees with given pendent vertices, the maximum trees (resp. minimum trees) with given branching number, the minimum trees with given maximum degree and second maximum degree, the minimum unicyclic graphs with given maximum degree and girth, the minimum connected graphs with given maximum degree and pendent vertices, and some properties of maximum connected graphs with given pendent vertices with respect to the first Sombor index $SO_{1}$.

As an application, we inaugurate these geometry-based invariants and verify their chemical applicability. We used these geometry-based invariants to model the acentric factor (resp. entropy, enthalpy of vaporization, etc.) of alkanes, and obtained satisfactory predictive potential, which indicates that these geometry-based invariants can be successfully used to model the thermodynamic properties of compounds.

It is also interesting to consider the extremal geometry-based invariants of graphs with given other parameters. We will consider it in the near futher.

\end{document}